\documentclass[12pt,a4paper]{amsart}
\usepackage{amssymb,eucal}
\usepackage{mathrsfs}
\usepackage[cmtip,all]{xy}
\usepackage{hyperref}

\calclayout

\newcommand{\+}{\protect\nobreakdash-}
\renewcommand{\:}{\colon}

\newcommand{\lrarrow}{\mskip.5\thinmuskip\relbar\joinrel\relbar\joinrel
 \rightarrow\mskip.5\thinmuskip\relax}

\DeclareMathOperator{\Hom}{Hom}
\DeclareMathOperator{\Ext}{Ext}
\DeclareMathOperator{\Spec}{Spec}
\DeclareMathOperator{\Spf}{Spf}
\DeclareMathOperator{\im}{im}
\DeclareMathOperator{\coker}{coker}

\newcommand{\Modl}{{\operatorname{\mathsf{--Mod}}}}
\newcommand{\Modr}{{\operatorname{\mathsf{Mod--}}}}
\newcommand{\Contra}{{\operatorname{\mathsf{--Contra}}}}
\newcommand{\Qcoh}{{\operatorname{\mathsf{--Qcoh}}}}
\newcommand{\Tors}{{\operatorname{\mathsf{--Tors}}}}
\newcommand{\Ctrh}{{\operatorname{\mathsf{--Ctrh}}}}
\newcommand{\Discr}{{\operatorname{\mathsf{Discr--}}}}
\newcommand{\Discrinj}{{\operatorname{\mathsf{Discr_{inj}--}}}}
\newcommand{\Separ}{{\operatorname{\mathsf{--Separ}}}}
\newcommand{\Flat}{{\operatorname{\mathsf{--Flat}}}}
\newcommand{\Sys}{{\operatorname{\mathsf{--Sys}}}}
\newcommand{\Sysr}{{\operatorname{\mathsf{Sys--}}}}
\newcommand{\Sysrinj}{{\operatorname{\mathsf{Sys_{inj}--}}}}

\newcommand{\Ab}{\mathsf{Ab}}
\newcommand{\Sets}{\mathsf{Sets}}

\newcommand{\cta}{\mathsf{cta}}
\newcommand{\fl}{\mathsf{fl}}
\newcommand{\prj}{\mathsf{prj}}

\newcommand{\rop}{\mathrm{op}}
\newcommand{\Id}{\mathrm{Id}}

\newcommand{\rarrow}{\longrightarrow}

\newcommand{\ot}{\otimes}
\newcommand{\wot}{\mathbin{\widehat\otimes}}
\newcommand{\ocn}{\odot}
\newcommand{\tim}{\rightthreetimes}
\newcommand{\uds}{{\text{\normalfont\textexclamdown}}}

\newcommand{\cM}{\mathcal M}
\newcommand{\cN}{\mathcal N}

\newcommand{\rN}{\mathscr N}

\newcommand{\fA}{\mathfrak A}
\newcommand{\fC}{\mathfrak C}
\newcommand{\fD}{\mathfrak D}
\newcommand{\fF}{\mathfrak F}
\newcommand{\fG}{\mathfrak G}
\newcommand{\fH}{\mathfrak H}
\newcommand{\fI}{\mathfrak I}
\newcommand{\fJ}{\mathfrak J}
\newcommand{\fK}{\mathfrak K}
\newcommand{\fP}{\mathfrak P}
\newcommand{\fQ}{\mathfrak Q}
\newcommand{\fR}{\mathfrak R}
\newcommand{\fS}{\mathfrak S}
\newcommand{\fU}{\mathfrak U}
\newcommand{\fV}{\mathfrak V}

\newcommand{\ff}{\mathfrak f}

\newcommand{\sA}{\mathsf A}
\newcommand{\sB}{\mathsf B}
\newcommand{\sC}{\mathsf C}
\newcommand{\sD}{\mathsf D}
\newcommand{\sP}{\mathsf P}

\newcommand{\boZ}{\mathbb Z}

\newcommand{\boL}{\mathbb L}

\newcommand{\Section}[1]{\bigskip\section{#1}\medskip}
\theoremstyle{plain}
\newtheorem{thm}{Theorem}[section]
\newtheorem{prop}[thm]{Proposition}
\newtheorem{lem}[thm]{Lemma}
\newtheorem{cor}[thm]{Corollary}
\theoremstyle{definition}
\newtheorem{rem}[thm]{Remark}

\newtheorem{ex}[thm]{Example}
\newtheorem{exs}[thm]{Examples}

\begin{document}

\title{Homomorphisms of topological rings \\
and change-of-scalar functors}

\author{Leonid Positselski}

\address{Institute of Mathematics, Czech Academy of Sciences \\
\v Zitn\'a~25, 115~67 Prague~1 \\ Czech Republic} 

\email{positselski@math.cas.cz}

\begin{abstract}
 We consider homomorphisms of complete, separated right or two-sided
linear topological rings with countable bases of neighborhoods of
zero $\ff\:\fR\rarrow\fS$.
 Taut maps of right linear topological rings, strongly right taut
maps of two-sided linear topological rings, left proflat continuous
ring maps, and topological ring epimorphisms are discussed.
 For a left proflat topological ring epimorphism~$\ff$, we show
that the functor of restriction of scalars on the categories of
left contramodules $\ff_\sharp\:\fS\Contra\rarrow\fR\Contra$ is
fully faithful.
 Assuming that the contramodule-to-module forgetful functor
$\fR\Contra\rarrow\fR\Modl$ is fully faithful and the topological
ring map~$\ff$ is left proflat, we prove that the commutative square of
forgetful functors between the left contramodule and module categories
over $\fS$ and $\fR$ is a pseudopullback diagram.
 This provides a description of the essential image of~$\ff_\sharp$
under the conjunction of the respective assumptions.
 The left adjoint functor to~$\ff_\sharp$ always exists, but is not
exact even when $\ff$~is (pro)flat.
 A right adjont functor to~$\ff_\sharp$ does not always exist, but
for a left proflat map~$\ff$ we construct it explicitly and show that
it has good exactness properties.
 This work is motivated by the theory of contraherent cosheaves of
contramodules on formal schemes.
\end{abstract}

\maketitle

\tableofcontents

\section{Introduction}
\medskip

 We start with several sections explaining the geometric motivation
before passing to the more technical and algebraic parts of
the Introduction.

\subsection{Schemes and quasi-coherent sheaves}
 Schemes are glued from affine schemes.
 Affine schemes are described by their rings of functions.
 To any commutative ring $R$, such as the ring of polynomials
$R=k[x_1,\dotsc,x_m]$ in several variables $x_1$,~\dots, $x_m$
over a field~$k$, the affine scheme $\Spec R$ is assigned.
{\hbadness=1150\par}

 Quasi-coherent sheaves on schemes are glued from modules over
commutative rings using the localization functors.
 The abelian category of quasi-coherent sheaves on an affine scheme
$U=\Spec R$ is equivalent to the category of $R$\+modules.
 Let $V=\Spec S$ be an affine open subscheme in~$U$.
 Then the functor of restriction of quasi-coherent sheaves to
the open subscheme ${-}|_V\:U\Qcoh\rarrow V\Qcoh$ is described in
terms of modules as the functor of extension of scalars
$S\ot_R{-}\,\:R\Modl\rarrow S\Modl$ with respect to the ring
homomorphism $R\rarrow S$.

 The example of a \emph{principal affine open subscheme} is instructive.
 In this case, the commutative $R$\+algebra $S$ is obtained by
inverting an element $r\in R$, that is $S=R[r^{-1}]$.
 The restriction of quasi-coherent sheaves from $\Spec R$ to
$\Spec S$ takes an $R$\+module $M$ to the $S$\+module $M[r^{-1}]$.
 This is the localization functor mentioned in the previous paragraph.

\subsection{Formal schemes and quasi-coherent torsion sheaves}
 Formal schemes~~\cite[Section~II.9]{Har}, \cite[Sections~0.7
and~I.10]{EGAI}, \cite[Chapter Tag~0AHW]{SP} are glued from affine
formal schemes.
 Affine formal schemes are described by their topological rings of
functions, such as the ring of formal power series
$\fR=k[x_1,\dotsc,x_m][[y_1,\dotsc,y_n]]$ in several variables
$y_1$,~\dots, $y_n$ with coefficients in a ring of polynomials
$k[x_1,\dotsc,x_m]$.

 The relevant class of topological commutative rings $\fR$ is specified
by several kinds of assumptions.
 First of all, the topological ring $\fR$ must be separated and
complete, and have a \emph{linear topology}.
 The latter condition means that open ideals form a base of
neighborhoods of zero in~$\fR$.
 A further nilpotency condition is imposed~\cite[Sections~0.7.1
and~I.10.1]{EGAI}.
 To any suitable topological commutative ring $\fR$, an affine formal
scheme $\Spf\fR$ is assigned.

 \emph{Quasi-coherent torsion sheaves} on formal schemes~\cite{AJL,AJPV}
are glued from torsion modules over topological commutative rings.
 The abelian category $\fU\Tors$ of quasi-coherent torsion sheaves on
an affine formal scheme $\fU=\Spf\fR$ is equivalent to the full
subcategory $\fR\Tors\subset\fR\Modl$ of torsion modules (known also as
\emph{discrete modules}) in the category of $\fR$\+modules.
 An $\fR$\+module $\cM$ is said to be \emph{torsion}
(or \emph{discrete}) if the annihilator of any element of $\cM$ is
an open ideal in~$\fR$.

 What are the affine open formal subschemes of an affine formal
scheme $\fU=\Spf\fR$\,?
 Let $\ff\:\fR\rarrow\fS$ be a morphism of topological commutative
rings.
 Then the related morphism of affine formal schemes $\Spf\fS\rarrow
\Spf\fR$ is said to be a \emph{formal open immersion} if, for every
open ideal $\fI\subset\fR$, the closure $\fJ$ of the ideal
$\fS\ff(\fI)\subset\fS$ is an open ideal in $\fS$, the ideals $\fJ$
form a base of neighborhoods of zero in $\fS$, and the induced maps of
affine schemes $\Spec\fS/\fJ\rarrow\Spec\fR/\fI$ are open immersions.

 Once again, the example of a \emph{principal affine open formal
subscheme} is instructive.
 In this case, there is an element $r\in\fR$ such that $\fS$ is
the completion of the commutative ring $\fR[r^{-1}]$ with respect to
the topology induced by the topology on~$\fR$.

 Quasi-coherent torsion sheaves are glued from torsion modules using
the localization functors.
 Let $\fV=\Spf\fS$ be an affine open formal subscheme in $\fU=\Spf\fR$.
 Then the functor of restriction of quasi-coherent torsion sheaves to
the open formal subscheme ${-}|_\fV\:\fU\Tors\rarrow\fV\Tors$ is
described in terms of modules as the functor of extension of
scalars $\fS\wot_\fR{-}\,\:\fR\Tors\rarrow\fS\Tors$.
 Here the notation~$\wot_\fR$ stands for the topological tensor product
functor, which in this context (under suitable countability assumptions
on the ring topologies) can be also interprested as the 
\emph{contratensor product} functor~$\ocn_\fR$.

 The key observation is that, for every torsion $\fR$\+module $\cM$,
the $\fS$\+module $\fS\wot_\fR\cM$ is also torsion.
 This holds because an $\fS$\+module is torsion if and only if its
underlying $\fR$\+module is torsion.
 For example, in the case of a principal affine open subscheme $\fV
\subset\fU$ corresponding to an element $r\in\fR$, one has
$\fS\wot_\fR\cM=\cM[r^{-1}]$.

\subsection{Contramodules} \label{introd-contramodules-subsecn}
 Contramodules over topological rings~\cite[Section~1.2]{Pweak},
\cite[Section~2.1]{Prev}, \cite[Sections~1.2 and~5]{PR},
\cite[Section~2.7]{Pcoun}, \cite[Sections~2.6\+-2.7]{Pproperf},
\cite[Section~6]{PS1} are the dual-analogous concept to discrete
or torsion modules.

 In the context of a complete, separated topological commutative ring
$\fR$ with linear topology, an \emph{$\fR$\+contramodule} is
a set/abelian group/$\fR$\+module $\fP$ endowed with \emph{infinite
summation operations}
$$
 (p_x\in\fP)_{x\in X}\longmapsto\sum\nolimits_{x\in X}r_xp_x\,\in\,\fP,
$$
assigning an element denoted formally by $\sum_{x\in X}r_xp_x\in\fP$
to every family of elements $(r_x\in\fR)_{x\in X}$ converging to zero
in the topology of $\fR$ and an arbitrary family of elements
$(p_x\in\fP)_{x\in X}$.
 Suitable contraassociativity and contraunitality axioms are imposed.

 For any discrete $\fR$\+module $\cM$ and any $\fR$\+module $W$,
the Hom $\fR$\+module $\Hom_\fR(\cM,W)$ has a natural
$\fR$\+contramodule structure given by the simple formula
\begin{equation} \label{Hom-contraaction-formula}
 \left(\sum\nolimits_{x\in X}r_xh_x\right)(m)
 =\sum\nolimits_{x\in X}h_x(r_xm)\,\in\,W
\end{equation}
for all $h_x\in\Hom_\fR(\cM,W)$ and $m\in\cM$.
 Here the sum in the right-hand side is finite, because $r_xm=0$ for
all but a finite subset of indices $x\in X$ (as $r_x\to0$ in $\fR$
and the annihilator of~$m$ is open in~$\fR$).

 The category $\fR\Contra$ of $\fR$\+contramodules is abelian.
 It is endowed with an exact, faithful forgetful functor
$\fR\Contra\rarrow\fR\Modl$ (preserving the infinite products, but
\emph{not} the infinite direct sums).

\subsection{Contraherent cosheaves}
 Contraherent cosheaves~\cite{Pcosh} are the dual-analogous concept to
quasi-coherent sheaves.
 Contraherent cosheaves on schemes are glued from modules over
commutative rings using the colocalization functors.

 The exact category $U\Ctrh$ of contraherent cosheaves on an affine
scheme $U=\Spec R$ is equivalent to the full subcategory of
\emph{contraadjusted} $R$\+modules $R\Modl^\cta\subset R\Modl$ in
the category of $R$\+modules.
 Here an $R$\+module $P$ is said to be
contraadjusted~\cite[Section~1.1]{Pcosh}, \cite[Section~2]{Pcta}
if $\Ext_R^1(R[r^{-1}],P)=0$ for all $r\in R$.
 There is no need to impose the $\Ext^i$ vanishing condition for
$i\ge2$ here, as the projective dimension of the $R$\+module
$R[r^{-1}]$ never exceeds~$1$ \,\cite[proof of Lemma~2.1]{Pcta}.

 Let $V=\Spec S$ be an affine open subscheme in~$U$.
 Then the functor of restriction of contraherent cosheaves to
the open subscheme ${-}|_V\:U\Ctrh\rarrow V\Ctrh$ is described in terms
of modules as the functor of \emph{coextension of scalars}
$\Hom_R(S,{-})\:R\Modl^\cta\rarrow S\Modl^\cta$.

 In the case of a principal affine open subscheme $V=\Spec R[r^{-1}]$,
this means that the restriction of contraherent cosheaves from $U$
to $V$ takes a contraadjusted $R$\+module $P$ to the contraadjusted
$R[r^{-1}]$\+module $\Hom_R(R[r^{-1}],P)$.
 This is the colocalization functor mentioned in the first paragraph
of this section.

\subsection{Contraherent cosheaves of contramodules}
\label{introd-cosheaves-of-contramodules-subsecn}
 Dual-analogously to the quasi-coherent torsion sheaves, one would like
to define \emph{contraherent cosheaves of contramodules} on formal
schemes (cf.~\cite[Section~2]{Pphil}).
 But how does one construct the gluing functors?

 Let $\fU=\Spf\fR$ be an affine formal scheme.
 We skip the discussion of the definition of the full subcategory
of \emph{contraadjusted\/ $\fR$\+contramodules} $\fR\Contra^\cta
\subset\fR\Contra$.
 It has been worked out in~\cite[Section~E.4]{Pcosh}
and~\cite[Example~7.12(3)]{PR}, under mild assumptions on~$\fR$.
 We would like the category $\fU\Ctrh$ of contraherent cosheaves of
contramodules on $\fU$ to be equivalent to $\fR\Contra^\cta$.
 But what would the functor of restriction to an affine open formal
subscheme do?

 Let $\fV=\Spf\fS$ be an affine open formal subscheme in an affine
formal scheme $\fU=\Spf\fR$.
 The functor ${-}|_\fV\:\fU\Ctrh\rarrow\fV\Ctrh$ is supposed to be
described on the level of contramodules as a functor of coextension
of scalars $\Hom^\fR(\fS,{-})$ (where $\Hom^\fR$ is our notation for
the Hom groups in $\fR\Contra$).
 
 Now let $\fP$ be a (contraadjusted) $\fR$\+contramodule.
 How does one construct an $\fS$\+contramodule structure on
$\Hom^\fR(\fS,\fP)$\,?

 There is a natural explicit general formula defining a contramodule
structure on the Hom group \emph{from a torsion module},
see~\eqref{Hom-contraaction-formula}.
 But $\fS$ is \emph{not} a torsion $\fR$\+module; it is
an $\fR$\+contramodule.
 There does \emph{not} seem to be a comparable explicit formula
defining a contramodule structure on the group of Hom \emph{from
a contramodule}.

 The Hom \emph{into} a contramodule is a contramodule, with infinite
summation operations on functions defined elementwise.
 So it is easy to construct an \emph{$\fR$\+contramodule} structure
on $\Hom^\fR(\fS,\fP)$ in terms of the $\fR$\+contramodule structure
on~$\fP$.
 But how does an \emph{$\fS$\+contramodule} structure on
$\Hom^\fR(\fS,\fP)$ emerge?

 Notice that the functor $\Hom^\fR(\fS,{-})\:\fR\Contra\rarrow
\fS\Contra$ is supposed to be right adjoint to the forgetful functor
$\fS\Contra\rarrow\fR\Contra$.
 But a right adjoint functor to the forgetful functor on
the contramodule categories, for a continuous homomorphism of
topological rings $\ff\:\fR\rarrow\fS$, does \emph{not} exist
in general, if only because the forgetful functor need not preserve
infinite direct sums~\cite[Section~2.9]{Pproperf}.
 (A more detailed discussion of a specific counterexample of this kind
is presented in
Example~\ref{nontaut-adjoints-need-not-exist} in this paper.)
 Furthermore, a mere existence of such a right adjoint would not be
sufficient for the aims of the theory of contraherent cosheaves
of contramodules.
 An explicit construction, allowing to establish exactness and
other properties, is very desirable.

 The main aim of the present paper is to work out general techniques
for constructing the desired $\fS$\+contramodule structure on
$\Hom^\fR(\fS,\fP)$.
 The reader can find the definition of a contraherent cosheaf of
contramodules over a formal scheme, worked out with complete details
(even though in much more narrow generality than the techniques
developed in the present paper may allow), in the recent
preprint~\cite{Pform}.

\subsection{The full-and-faithfulness approach}
 Succinctly put, one approach that we suggest works as follows.
 There may be no simple formula defining an $\fS$\+con\-tra\-mod\-ule
structure on $\Hom^\fR(\fS,\fP)$, but there are such simple formulas
making $\Hom^\fR(\fS,\fP)$ an $\fR$\+contramodule and
an $\fS$\+module.

 The idea is to prove that one of the forgetful functors
$\fS\Contra\rarrow\fR\Contra$ or $\fS\Contra\rarrow\fS\Modl$ is
fully faithful.
 Specifically, the functor $\fS\Contra\rarrow\fR\Contra$ is considered
in detail in this paper.

 Under certain assumptions, we not only prove that the functor
$\fS\Contra\rarrow\fR\Contra$ is fully faithful, but also describe
its essential image.
 In fact, we show that an $\fR$\+contramodule arises from
an $\fS$\+contramodule if and only if its underlying $\fR$\+module
arises from an $\fS$\+module.
 Then it follows immediately that there is a naturally defined
$\fS$\+contramodule structure on $\Hom^\fR(\fS,\fP)$, even if
this definition is not quite explicit.

 Under less restrictive assumptions, we have no explicit description
of the essential image, but the functor $\fS\Contra\rarrow\fR\Contra$ is
still provable to be fully faithful.
 It follows that the essential image enjoys certain closure properties;
specifically, it is a full subcategory closed under all limits and
finite colimits in $\fR\Contra$.
 Using such properties, one can show that the $\fR$\+contramodule
$\Hom^\fR(\fS,\fP)$ belongs to the full subcategory of
$\fS$\+contramodules in $\fR\Contra$.

\subsection{Toy example: adic Noetherian rings}
\label{introd-adic-Noetherian-toy-example-subsecn}
 Let $\fR$ be a Noetherian commutative ring, separated and complete
in the adic topology of an ideal $\fI\subset\fR$.
 What the $\fR$\+contramodules are in this context?

 More generally, let $I$ be an ideal in a commutative ring~$R$.
 An $R$\+module $P$ is said to be
an \emph{$I$\+contramodule}~\cite[Section~2]{Pmgm},
\cite[Sections~2 and~7]{Pcta} if one has $\Hom_R(R[s^{-1}],P)=0=
\Ext_R^1(R[s^{-1}],P)$ for all $s\in I$.
 It suffices to check this condition for the element~$s$ ranging over
any chosen set of generators of the ideal $I\subset R$
\,\cite[Theorem~5.1]{Pcta}.
 In the case of a finitely generated ideal $I\subset R$,
the $I$\+contramodule $R$\+modules are also known
as ``derived $I$\+complete modules''~\cite[Section~3.4]{BS},
\cite[Section Tag~091N]{SP}.

 Let $R$ be a Noetherian commutative ring and $I\subset R$ be an ideal.
 Denote by $\fR=\varprojlim_{n\ge1}R/I^n$ the $I$\+adic completion of
the ring $R$, endowed with the $I$\+adic topology, or equivalently,
the topology of projective limit of discrete rings $R/I^n$.
 Then the composition of forgetful functors $\fR\Contra\rarrow
\fR\Modl\rarrow R\Modl$ is a fully faithful functor $\fR\Contra
\rarrow R\Modl$.
 The essential image of the latter functor consists precisely of all
the $I$\+contramodule $R$\+modules~\cite[Theorem~B.1.1]{Pweak}.

 Now let $\fR$ and $\fS$ be Noetherian commutative rings, both
separated and complete in their respective adic topologies of
ideals $\fI\subset\fR$ and $\fJ\subset\fS$.
 Let $\ff\:\fR\rarrow\fS$ be a ring homomorphism such that
the ideal $\fJ$ is the extension of the ideal $\fI$ in $\fS$,
that is, $\fJ=\fS\ff(\fI)$.
 Let $\fP$ be an $\fR$\+contramodule (or equivalently,
an $\fI$\+contramodule $\fR$\+module).
 Why is $\Hom^\fR(\fS,\fP)=\Hom_\fR(\fS,\fP)$ an $\fS$\+contramodule?

 This is easy to prove~\cite[Lemma~D.4.2(a)]{Pcosh}.
 Clearly, $\Hom^\fR(\fS,\fP)$ is an $\fS$\+module.
 We need to show that it is a $\fJ$\+contramodule $\fS$\+module.
 By the result of~\cite[Theorem~5.1]{Pcta} cited above, it suffices
to check that the underlying $\fR$\+module of $\Hom^\fR(\fS,\fP)$ is
an $\fI$\+contramodule (here we are using the assumption that
the ideal $\fJ\subset\fS$ is spanned by the image of~$\fI$).
 The rest is an easy lemma: for any $\fR$\+module $M$, the functor
$\Hom_\fR(M,{-})$ takes $\fI$\+contramodules to
$\fI$\+contramodules~\cite[Lemma~6.1(b)]{Pcta}.

 What happens when one drops the Noetherianity assumption?
 Let $R$ be a commutative ring, $I\subset R$ be a finitely generated
ideal, and $\fR=\varprojlim_{n\ge1}R/I^n$ be the $I$\+adic completion
of~$R$ (endowed with the $I$\+adic~$=$~projective limit topology).
 Then the forgetful functor $\fR\Contra\rarrow R\Modl$ is still
fully faithful (cf.\ the discussion in
Section~\ref{introd-contramodule-to-module-subsecn} below).
 But the essential image of that forgetful functor is now the full
subcategory of all \emph{quotseparated $I$\+contramodule
$R$\+modules}~\cite[Proposition~1.5]{Pdc}.
 Here an $I$\+contramodule $R$\+module is said to be
\emph{quotseparated} if it is a quotient $R$\+module of
an $I$\+adically separated and complete $R$\+module.

 All $I$\+contramodule $R$\+modules are quotseparated whenever
the ring $R$ is Noetherian, or more generally, whenever the ideal
$I\subset R$ is \emph{weakly proregular} (in the sense
of~\cite{Sch,PSY}) \cite[Corollary~3.7]{Pdc}, or even under a somewhat
weaker assumption on the ideal~$I$ \,\cite[Remark~3.8]{Pdc}.
 But generally speaking, \emph{not} every $I$\+contramodule $R$\+module
is quotseparated~\cite[Example~2.6]{Pmgm}, \cite[Examples~1.8]{Pdc}.

 Let $\fR$ and $\fS$ be (not necessarily Noetherian) commutative rings,
both separated and complete in their respective adic topologies of
finitely generated ideals $\fI\subset\fR$ and $\fJ\subset\fS$.
 As above, we consider a ring homomorphism $\ff\:\fR\rarrow\fS$ such
that $\fJ=\fS\ff(\fI)$.
 Let $\fP$ be an $\fR$\+contramodule, or equivalently,
a quotseparated $\fI$\+con\-tra\-mod\-ule $\fR$\+module.
 Is $\Hom^\fR(\fS,\fP)=\Hom_\fR(\fS,\fP)$ an $\fS$\+contramodule?
 In other words, is the $\fJ$\+contramodule $\fS$\+module
$\Hom^\fR(\fS,\fP)$ quotseparated?
 We are \emph{unable} to answer this question in general.
 Under a suitable (pro)flatness assumption on the map~$\ff$,
the answer is positive; this is a special case of the much more
general assertion~(c) from 
Section~\ref{introd-contramodule-structure-on-Hom-general-case} below.

\subsection{Ring epimorphisms}
 What is entailed in the forgetful functor $\fS\Contra\rarrow\fR\Contra$
being fully faithful, for a morphism of topological rings
$\fR\rarrow\fS$\,?
 The case of abstract (discrete, nontopological) rings $R$ and $S$ is
relevant.

 Let $f\:R\rarrow S$ be a homomorphism of associative rings.
 Then the following six conditions are
equivalent~\cite[Section~XI.1]{Sten}:
\begin{enumerate}
\item the forgetful functor $S\Modl\rarrow R\Modl$ for
the categories of left modules is fully faithful;
\item the forgetful functor $\Modr S\rarrow\Modr R$ for
the categories of right modules is fully faithful;
\item any one or both of the two maps $S\rightrightarrows S\ot_RS$
induced by~$f$ are isomorphisms (or equivalently, surjective);
\item the two maps $S\rightrightarrows S\ot_RS$ are equal to each
other;
\item the multiplication map $S\ot_RS\rarrow S$ is an isomorphism
(or equivalently, injective);
\item $f$~is an epimorphism in the category of rings.
\end{enumerate}
 If the rings $R$ and $S$ are commutative, then conditions~(1\+-6)
are also equivalent to the following condition:
\begin{enumerate}
\setcounter{enumi}{6}
\item $f$~is an epimorphism in the category of commutative rings.
\end{enumerate}
 Notice that all the three maps $S\rightrightarrows S\ot_RS\rarrow S$
are ring homomorphisms if the rings $R$ and $S$ are commutative (but
not in the general case of associative rings).

\subsection{Topological rings}
 The main results of this paper are stated in the context of
associative, noncommutative rings.

 The maximal natural generality for the definition of a contramodule
is that of complete, separated right linear topological rings~$\fR$.
 The words ``right linear'' mean that open right ideals form a base of
neighborhoods of zero.
 To such a topological ring $\fR$, one assigns the abelian categories
of \emph{discrete right\/ $\fR$\+modules} $\Discr\fR$ and
\emph{left\/ $\fR$\+contramodules} $\fR\Contra$.

 A right $\fR$\+module $\cN$ is said to be \emph{discrete} if
the annihilator of any element of $\cN$ is an open right ideal in~$\fR$.
 The left $\fR$\+contramodules are the sets/abelian groups/left
$\fR$\+modules with infinite summation operations, similarly to
the discussion in Section~\ref{introd-contramodules-subsecn} above.
 For every family of elements $(r_x\in\fR)_{x\in X}$ converging to
zero in the topology of $\fR$, an infinite summation operation with
this family of coefficients is defined in any $\fR$\+contramodule~$\fP$.

 The setting in this paper is a bit more restrictive.
 For the main results, we make two additional assumptions.
 Firstly, our topological rings need to have a \emph{countable} base
of neighborhoods of zero.
 Secondly, we assume a base of neighborhoods of zero consisting of
open \emph{two-sided} ideals.

\subsection{Change of scalars and gluing over a formal scheme}
 Let $f\:R\rarrow S$ be a continuous homomorphism of right linear
topological rings.
 Then the forgetful/restriction-of-scalars functor
$f_\diamond\:\Discr S\rarrow\Discr R$ is well-defined.
 The functor~$f_\diamond$ assigns to a discrete right $S$\+module
$\cN$ its underlying discrete right $R$\+module~$\cN$.
 The functor~$f_\diamond$ always has a right adjoint functor
$f^\diamond\:\Discr R\allowbreak\rarrow\Discr S$, called the functor
of \emph{coextension of scalars}~\cite[Section~2.9]{Pproperf}.

 Let $\ff\:\fR\rarrow\fS$ be a continuous homomorphism of complete,
separated right linear topological rings.
 Then the forgetful/restriction-of-scalars functor
$\ff_\sharp\:\fS\Contra\rarrow\fR\Contra$ is well-defined.
 The functor~$\ff_\sharp$ assigns to a left $\fS$\+contramodule
$\fQ$ its underlying left $\fR$\+contramodule~$\fQ$.
 The functor~$\ff_\sharp$ always has a left adjoint functor
$\ff^\sharp\:\fR\Contra\rarrow\fS\Contra$, called the functor
of \emph{contraextension of scalars}~\cite[Section~2.9]{Pproperf}.

 Why are we so interested in a right adjoint functor to~$\ff_\sharp$,
which we want to compute as $\Hom^\fR(\fS,{-})$, and which does not
even always exist?
 What's wrong with the functor~$\ff^\sharp$ left adjoint
to~$\ff_\sharp$, which exists for any morphism of (complete,
separated, right linear) topological rings~$\ff$\,?

 The discussion in
Section~\ref{introd-cosheaves-of-contramodules-subsecn} suggests
considering \emph{cosheaves of contramodules}.
 This involves gluing contramodules over nonaffine formal schemes
using colocalization functors, and $\Hom^\fR(\fS,{-})$ are
the required colocalization functors.

 Why not consider \emph{sheaves of contramodules}?
 This would mean gluing contramodules over nonaffine formal schemes
using localization functors, and the localization functors~$\ff^\sharp$
are always available.
 The answer is that the functors of contraextension of
scalars~$\ff^\sharp$ do always exist, but even the most stringent
of reasonable assumptions on the morphism~$\ff$ do \emph{not} make them
\emph{well-behaved enough}.

 For the glued category of sheaves or cosheaves to be well-behaved
from the homological algebra standpoint, the functors used for
the gluing must be exact, if only on the suitable exact subcategories
of ``adjusted'' modules or contramodules.
 The open immersion morphisms in the theory of schemes are flat, and so
the related localization/extension-of-scalars functors are exact.

 But the functors of contraextension of scalars~$\ff^\sharp$ are
\emph{not} exact for morphisms of topological commutative rings
$\ff\:\fR\rarrow\fS$ corresponding to open immersions of affine formal
schemes $\Spf\fS\rarrow\Spf\fR$.
 A simple counterexample to this effect, involving Noetherian
commutative rings with adic topologies, is presented in
Example~\ref{contraextension-not-exact-counterex} at the end of
this paper (see also the preceding
Remark~\ref{exten-coexten-contraexten-exactness-properties-rem}).

 The functor of contraextension of scalars~$\ff^\sharp$ only seems to
be exact on the exact category of \emph{flat} $\fR$\+contramodules (see
Lemma~\ref{flat-contramodules-contraextension-of-scalars}).
 Flat contramodules can indeed be glued over formal schemes using
the functors of contraextension of scalars.
 The resulting exact category is known as the category of
\emph{flat pro-quasi-coherent pro-sheaves} on a formal scheme, or
more generally, on an ind-scheme~\cite[Section~7.11.3]{BD2},
\cite[Chapter~3]{Psemten}.
 But flat (contra)modules are relatively rare.

 The functors $\Hom^\fR(\fS,{-})$, which we construct in this paper,
are \emph{not} exact on the whole abelian categories of contramodules,
either.
 But the functor $\Hom^\fR(\fS,{-})$ is exact on the exact category
of \emph{cotorsion left\/ $\fR$\+contramodules} whenever $\fS$ is
a flat left $\fR$\+contramodule.
 The latter condition is satisfied for \emph{left proflat} morphisms
of topological rings~$\ff$; see
Section~\ref{introd-morphisms-of-topol-rings-subsecn} below.
 It is satisfied, in particular, in the case of a morphism of
commutative topological rings $\ff\:\fR\rarrow\fS$ corresponding to
an open immersion of affine formal schemes $\Spf\fS\rarrow\Spf\fR$.
 In fact, the functor $\Hom^\fR(\fS,{-})$ is exact on the even wider
exact category of \emph{contraadjusted} $\fR$\+contramodules in
the latter case.
 Contraadjusted (contra)modules exist in abundance, and even cotorsion
(contra)modules are numerous enough for many purposes.

 See also the discussion in~\cite[Section~2]{Pphil}.

\subsection{Classes of morphisms of topological rings}
\label{introd-morphisms-of-topol-rings-subsecn}
 Several classes of morphisms of topological rings are discussed in
detail in this paper.
 Borrowing the terminology from~\cite[Section Tag~0GX1]{SP}, we
consider \emph{taut maps} of right linear topological rings.
 All rings in~\cite{SP} are generally assumed to be commutative;
in the context of not necessarily commutative two-sided linear
topological rings, we are interested in what we call \emph{strongly
right taut} continuous ring maps.
 We use the later concept in order to define a more narrow class of
\emph{left proflat} continuous ring maps.
 We also discuss \emph{proepimorphisms} of two-sided linear
topological rings.

 More explicitly, the left proflat maps of complete, separated
two-sided linear topological rings with countable bases of neighborhoods
of zero $\ff\:\fR\rarrow\fS$ are described by the following sets of
data.
 There is a morphism $f\:(R_n)_{n\ge1}\rarrow (S_n)_{n\ge1}$ of
projective systems of rings, indexed by the positive integers~$n$.
 The transition maps $R_{n+1}\rarrow R_n$ and $S_{n+1}\rarrow S_n$
in both the projective systems are surjective ring maps.
 The ring maps $f_n\:R_n\rarrow S_n$ have the property that they
make $S_n$ a flat left $R_n$\+module for every $n\ge1$.
 An important assumption is that the induced $R_n$\+$S_{n+1}$\+bimodule
maps $R_n\ot_{R_{n+1}}S_{n+1}\rarrow S_n$ must be isomorphisms.

 The map $\ff\:\fR\rarrow\fS$ is the projective limit
$\ff=\varprojlim_{n\ge1}f_n$.
 The topologies on $\fR$ and $\fS$ are the topologies of projective
limit of discrete rings $R_n$ and~$S_n$.
 Then the map~$\ff$ is left proflat; and all the left proflat maps
of complete, separated two-sided linear topological rings with
countable bases of neighborhoods of zero can be obtained in this way.
 The map~$\ff$ is a a proepimorphism of topological rings if and only
if the maps~$f_n$ are ring epimorphisms for all $n\ge1$.

 More generally, a continuous homomorphism of two-sided linear
topological rings $f\:R\rarrow S$ is called a \emph{topological ring
proepimorphism} if, for every pair of open two-sided ideals
$I\subset R$ and $J\subset S$ such that $f(I)\subset J$, the induced
homomorphism of discrete rings $R/I\rarrow S/J$ is a ring epimorphism.

 Here are our results in this context:
\begin{enumerate}
\renewcommand{\theenumi}{\alph{enumi}}
\item A strongly right taut continuous map of two-sided linear
topological rings $f\:R\rarrow S$ is a topological ring proepimorphism
if and only if the functor of restriction of scalars on the categories
of discrete right modules $f_\diamond\:\Discr S\rarrow\Discr R$
is fully faithful
(Proposition~\ref{full-and-faithfulness-equivalent-conditions-prop-I}).
\item A strongly right taut continuous map of complete, separated
two-sided linear topological rings with countable bases of neighborhoods
of zero $\ff\:\fR\rarrow\fS$ is a topological ring proepimorphism if
and only if the functor of restriction of scalars on the categories of
separated left contramodules $\ff_\sharp\:\fS\Separ\rarrow\fR\Separ$ is
fully faithful
(Proposition~\ref{full-and-faithfulness-equivalent-conditions-prop-II}).
\item For a morphism of complete, separated right linear topological
rings with countable bases of neighborhoods of zero $\ff\:\fR\rarrow\fS$
such that the functor of restriction of scalars on the categories of
left contramodules $\ff_\sharp\:\fS\Contra\rarrow\fR\Contra$ takes flat
$\fS$\+contramodules to flat $\fR$\+contramodules,
the functor~$\ff_\sharp$ is fully faithful if and only if the similar
functor on the categories of separated left contramodules
$\ff_\sharp\:\fS\Separ\rarrow\fR\Separ$ is fully faithful, and if and
only if  he similar functor on the categories of flat left contramodules
$\ff_\sharp\:\fS\Flat\rarrow\fR\Flat$ is fully faithful
(Theorem~\ref{full-and-faithfulness-equivalent-conditions-thm}).
\item Consequently, for a left proflat map of complete, separated
two-sided linear topological rings with countable bases of neighborhoods
of zero $\ff\:\fR\rarrow\fS$, the functor of restriction of scalars on
the categories of left contramodules $\ff_\sharp\:\fS\Contra\rarrow
\fR\Contra$ is fully faithful if and only if $\ff$~is a topological ring
proepimorphism
(Corollary~\ref{full-and-faithfulness-equivalent-conditions-cor}).
\end{enumerate}

 Let us make a small comment concerning the comparison between
assertions~(b) and~(d) above.
 All the main results of this paper require the \emph{strong right
tautness assumption} on a ring map~$\ff$ in the case of \emph{separated}
contramodules.
 This already presumes that the topological rings $\fR$ and $\fS$ are
two-sided linear.

 In the context of the whole abelian categories (of \emph{not}
necessarily separated) contramodules, the even stronger \emph{left
proflatness} assumption on~$\ff$ is needed to make our arguments work.
 For some of the results, we even present counterexamples showing that
the strong right tautness assumption in lieu of the left proflatness
is \emph{not} sufficient; see
Section~\ref{introd-characterizing-the-essential-image-subsecn} below
and Examples~\ref{nonflat-pseudopullback-not-true-counterex}.

\subsection{Contramodule-to-module forgetful functors}
\label{introd-contramodule-to-module-subsecn}
 Let $\fR$ be a complete, separated right linear topological ring,
$R$ be a ring, and $\rho\:R\rarrow\fR$ be a ring homomorphism
whose image is dense in~$\fR$.
 Endow $R$ with the induced topology (where a base of neighborhoods of
zero is formed by the open right ideals $I=\rho^{-1}(\fI)$, where
$\fI$ ranges over the open right ideals in~$\fR$).
 Consider the composition of forgetful functors $\fR\Contra\rarrow
\fR\Modl\rarrow R\Modl$.

 When is the functor $\fR\Contra\rarrow R\Modl$ fully faithful?
 This question has been studied in~\cite[Theorem~B.1.1]{Pweak},
\cite[Theorem~2.1]{Psm}, \cite[Section~3]{Pper}, and in the greatest
generality in~\cite[Section~6]{Pcoun} (for a survey,
see~\cite[Section~3.8]{Prev}).
 In the case of a complete, separated right linear topological ring
$\fR$ with a countable base of neighorhoods of zero, the two main
results are:
\begin{enumerate}
\renewcommand{\theenumi}{\alph{enumi}}
\item The forgetful functor $\fR\Contra\rarrow R\Modl$ is fully
faithful if and only if the forgetful functor $\fR\Separ\rarrow
R\Modl$ is fully faithful, and if and only if the natural surjective map
$$
 \cN\ot_R\fP\lrarrow\cN\ocn_\fR\fP
$$
from the tensor product to the \emph{contratensor product} of
a discrete right $\fR$\+mod\-ule $\cN$ and a left $\fR$\+contramodule
$\fP$ is an isomorphism for all $\cN$ and~$\fP$, or equivalently,
for all cyclic discrete right $\fR$\+modules $\cN=\fR/\fI$ and
all free left $\fR$\+contramodules $\fP=\fR[[X]]$
\,\cite[Lemma~7.11]{PS1}, \cite[Theorem~6.2 and its proof]{Pcoun},
and Lemma~\ref{contramodule-to-module-full-and-faithful-lemma} below;
\item The equivalent conditions of item~(a) hold whenever the topology
on $R$ is a right Gabriel topology (i.~e., the full subcategory
$\Discr R$ is closed under extensions in $\Modr R$) that has
a (countable) base of neighborhoods of zero consisting of finitely
generated open right ideals in~$R$ \,\cite[Theorem~6.6
and Corollary~6.7]{Pcoun}.
\end{enumerate}

\subsection{Characterizing the essential image}
\label{introd-characterizing-the-essential-image-subsecn}
 In the context of assertion~(d) from
Section~\ref{introd-morphisms-of-topol-rings-subsecn},
the forgetful functor $\ff_\sharp\:\fS\Contra\rarrow\fR\Contra$ is
fully faithful, so the question of describing its essential image
becomes relevant.
 We answer this question under the additional assumption that
the contramodule-to-module forgetful $\fR\Contra\rarrow\fR\Modl$
is fully faithful, too.

 Here is the formulation of our result.
 Let $\ff\:\fR\rarrow\fS$ be a left proflat proepimorphism of
complete, separated two-sided linear topological rings with countable
bases of neighborhoods of zero.
 Assume additionally that the forgetful functor $\fR\Contra\rarrow
\fR\Modl$ is fully faithful.
 Then a left $\fR$\+contramodule $\fQ$ arises from
an $\fS$\+contramodule if and only if the underlying left
$\fR$\+module of $\fQ$ arises from an $\fS$\+module.
 Moreover, the $\fS$\+module structure on $\fQ$ is unique if it exists.
 It follows that the forgetful functor $\fS\Contra\rarrow
\fS\Modl$ is fully faithful, too.

 Notice that we do \emph{not} claim here that an $\fS$\+module
structure extending the given $\fR$\+module structure on
a set/abelian group $Q$ is always unique if it exists.
 This is \emph{not} true, as a left proflat proepimorphism of
topological rings \emph{need not} be an epimorphism of the underlying
abstract rings (see
Example~\ref{proepimorphism-not-epimorphism-example}
and Remark~\ref{non-R-contramodule-S-module-structure-not-unique}).
 Rather, the claim is that there is at most one $\fS$\+module structure
extending the given $\fR$\+module structure \emph{if the given\/
$\fR$\+module structure arises from an\/ $\fR$\+contramodule structure}.

 More generally, the topological proepimorphism assumption on
the map~$\ff$ can be dropped.
 Then the functor of restriction of scalars
$\ff_\sharp\:\fS\Contra\rarrow\fR\Contra$ is no longer fully faithful.
 In this context, we prove the following results, stated in terms of
the commutative square diagrams of forgetful functors.
 In all the three assertions, the forgetful functor $\fR\Contra\rarrow
\fR\Modl$ is assumed to be fully faithful:
\begin{enumerate}
\renewcommand{\theenumi}{\alph{enumi}}
\item For a taut continuous map of complete, separated right linear
topological rings with countable bases of neighborhoods of zero
$\ff\:\fR\rarrow\fS$, a right $\fS$\+module is discrete if and only if
its underlying right $\fR$\+module is discrete
(Proposition~\ref{topol-ring-map-discreteness-characterized}).
\item For a strongly right taut continuous map of complete, separated
two-sided linear topological rings with countable bases of neighborhoods
of zero $\ff\:\fR\rarrow\fS$, the commutative diagram of forgetful
functors
$$
 \xymatrix{
  \fS\Separ \ar@{>->}[r] \ar[d]_-{\ff_\sharp}
  & \fS\Modl \ar[d]^-{\ff_*} \\
  \fR\Separ \ar@{>->}[r] & \fR\Modl
 }
$$
is a pseudopullback square (up to category equivalence).
 Here $\ff_*\:\fS\Modl\rarrow\fR\Modl$ is the functor of restriction
of scalars acting on the module categories.
 This is our Proposition~\ref{separated-pseudopullback-prop}.
\item For a left proflat map of complete, separated two-sided linear
topological rings with countable bases of neighborhoods of zero
$\ff\:\fR\rarrow\fS$, the commutative diagram of forgetful functors
$$
 \xymatrix{
  \fS\Contra \ar@{>->}[r] \ar[d]_-{\ff_\sharp}
  & \fS\Modl \ar[d]^-{\ff_*} \\
  \fR\Contra \ar@{>->}[r] & \fR\Modl
 }
$$
is a pseudopullback square (up to category equivalence).
 This is our Theorem~\ref{pseudopullback-thm}.
\end{enumerate}
 The horizontal arrows with tails on both the diagrams above show fully
faithful functors.

 Once again, similarly to the comments at the end of
Section~\ref{introd-morphisms-of-topol-rings-subsecn},
the assertion~(b) (concerning separated contramodules) is provable
under the strong right tautness assumption on~$\ff$, while our proof
of the assertion~(c) (concerning arbitrary contramodules) requires
the left proflatness assumption on~$\ff$.

 However, \emph{unlike} in the contexts of
Section~\ref{introd-morphisms-of-topol-rings-subsecn} above and
Section~\ref{introd-contramodule-structure-on-Hom-general-case} below,
we actually present a counterexample showing that the assertion~(c)
of the present section is \emph{not} true with the left proflatness
assumption on~$\ff$ replaced by the strong right tautness assumption.
 This counterexample arises in the setting of quotseparated
contramodules over a (non-Noetherian) commutative ring with a finitely
generated ideal, which was discussed in
Section~\ref{introd-adic-Noetherian-toy-example-subsecn}. 
 See Examples~\ref{nonflat-pseudopullback-not-true-counterex}.

\subsection{Constructing the $\fS$-contramodule structure on the $\Hom$}
\label{introd-contramodule-structure-on-Hom-general-case}
 In the final analysis, it turns out that one can construct
the left $\fS$\+contramodule structure on $\Hom^\fR(\fS,\fP)$ even when
\emph{no} relevant forgetful functors are fully faithful.
 In the case of a separated left $\fR$\+contramodule $\fP$, one
computes the abelian group $\Hom^\fR(\fS,\fP)$ by representing $\fP$
naturally as the projective limit of left modules over the discrete
quotient rings of~$\fR$.
 The computation shows that $\Hom^\fR(\fS,\fP)$ is a projective limit
of left modules over the discrete quotient rings of~$\fS$, hence
a separated left $\fS$\+contramodule.

 In the case of a nonseparated $\fR$\+contramodule $\fP$, we use
a \emph{flat special precover} short exact sequence for~$\fP$,
whose existence was established in~\cite[Section~7]{PR}.
 This means a presentation of $\fP$ as the quotient $\fR$\+contramodule
of a flat $\fR$\+contramodule by a cotorsion one.
 It is, of course, important here that all flat contramodules are
separated (and all subcontramodules of separated contramodules
are separated).
 The assumption that $\ff$~is a left proflat map guarantees that $\fS$
is flat left $\fR$\+contramodule, so the functor $\Hom^\fR(\fS,{-})$
takes flat special precover short exact sequences in $\fR\Contra$ to
short exact sequences of abelian groups.
 Our results are:
\begin{enumerate}
\renewcommand{\theenumi}{\alph{enumi}}
\item For a taut continuous map of complete, separated right linear
topological rings with countable bases of neighborhoods of zero
$\ff\:\fR\rarrow\fS$, the contratensor product functor
${-}\ocn_\fR\fS$ takes discrete right $\fR$\+modules to discrete
right $\fS$\+modules.
 The functor $\ff^\bullet={-}\ocn_\fR\fS\:\Discr\fR\rarrow\Discr\fS$
is left adjoint to the functor of restriction of scalars
$\ff_\diamond\:\Discr\fS\rarrow\Discr\fR$
(Proposition~\ref{discrete-module-extension-of-scalars-prop}).
\item Given a strongly right taut continuous map of complete, separated
two-sided linear topological rings with countable bases of neighborhoods
of zero $\ff\:\fR\rarrow\fS$, for any separated left $\fR$\+contramodule
$\fP$ there is a natural separated left $\fS$\+contramodule structure
on the set/abelian group $\Hom^\fR(\fS,\fP)$, extending the obvious
left $\fS$\+module structure on $\Hom^\fR(\fS,\fP)$. 
 The functor $\ff^\uds=\Hom^\fR(\fS,{-})\:\fR\Separ\rarrow\fS\Separ$
is right adjoint to the functor of restriction of scalars
$\ff_\sharp\:\fS\Separ\rarrow\fR\Separ$
(Proposition~\ref{separated-contramodule-coextension-of-scalars-prop}).
\item Given a left proflat map of complete, separated two-sided linear
topological rings with countable bases of neighborhoods of zero
$\ff\:\fR\rarrow\fS$, for any left $\fR$\+contramodule $\fP$ there is
a natural left $\fS$\+contramodule structure on the set/abelian group
$\Hom^\fR(\fS,\fP)$, extending the obvious left $\fS$\+module structure
on $\Hom^\fR(\fS,\fP)$. 
 The functor $\ff^\uds=\Hom^\fR(\fS,{-})\:\fR\Contra\rarrow\fS\Contra$
is right adjoint to the functor of restriction of scalars
$\ff_\sharp\:\fS\Contra\allowbreak\rarrow\fR\Contra$
(Theorem~\ref{contramodule-coextension-of-scalars-thm}).
\end{enumerate}

\subsection*{Acknowledgement}
 I~am grateful to Michal Hrbek for very helpful discussions.
 In particular, the idea of the argument in
Example~\ref{proepimorphism-not-epimorphism-example} was suggested
by Michal.
 This research is supported by the project LQ100192601 Lumina
quaeruntur, funded by the Czech Academy of Sciences (RVO~67985840).

\Section{Preliminaries} \label{preliminaries-secn}

 The definition of a contramodule over topological ring goes back
to~\cite[Remark~A.3]{Psemi}, \cite[Section~1.2]{Pweak},
\cite[Section~2.1]{Prev}, and~\cite[Sections~1.2 and~5]{PR}.
 More detailed expositions of the material sketched in this section
can be found in~\cite[Section~2]{Pcoun}, \cite[Section~2]{Pproperf},
or~\cite[Sections~6\+-7]{PS1}.

 We denote by $R\Modl$ and $\Modr R$ the abelian categories of left
and right $R$\+modules, by $\Ab=\boZ\Modl$ the abelian category of
abelian groups, and by $\Sets$ the category of sets.
 The groups of morphisms in $R\Modl$ are denoted by $\Hom_R({-},{-})$,
while the groups of morphisms in $\Modr R$ are denoted by
$\Hom_{R^\rop}({-},{-})$ (where $R^\rop$ is the opposite ring to~$R$).

 A topological abelian group $A$ is said to have a \emph{linear
topology} if open subgroups form a base of neighborhoods of zero in~$A$.
 All the topological abelian groups in this paper are presumed to have
linear topologies.

 The \emph{completion} of a topological abelian group $A$ is
the abelian group $\fA=\varprojlim_{U\subset A}A/U$, where
the projective limit is taken over the directed poset of all open
subgroups $U\subset A$.
 The abelian group $\fA$ is endowed with the topology of projective
limit of the discrete abelian groups~$A/U$.
 The natural \emph{completion map} $\lambda_A\:A\rarrow\fA$ is
a continuous homomorphism of topological abelian groups.

 A topological abelian group $A$ is said to be \emph{separated} if
the map~$\lambda_A$ is injective, and \emph{complete} if
the map~$\lambda_A$ is surjective.
 The completion $\fA$ of any topological abelian group $A$ is
a complete and separated topological abelian group.

 A topological ring $R$ is said to be \emph{right linear} (or to have
a \emph{right linear topology}) if open right ideals form a base of
neighborhoods of zero in~$R$.
 Similarly, the topological ring $R$ is said to be \emph{two-sided
linear} (or to have a \emph{two-sided linear topology}) if open
two-sided ideals form a base of neighborhoods of zero.
 In the case of a commutative ring $R$, we speak simply of
a \emph{linear ring topology} or a \emph{linear topological ring}.

 For any right linear topological ring $R$, there exists a unique
topological ring structure on the completion $\fR$ of the topological
abelian group $R$ such that the completion map $\lambda_R\:R\rarrow\fR$
is a ring homomorphism.
 This ring topology makes $\fR$ a right linear topological ring.
 If the topological ring $R$ is two-sided linear, then so is its
completion~$\fR$.

 Let $R$ be a right linear topological ring.
 A right $R$\+module $\cN$ is said to be \emph{discrete} if, endowed
with the discrete topology, $\cN$ becomes a topological $R$\+module.
 The latter condition means that the action map $\cN\times R\rarrow\cN$
is continuous in the given topology on $R$ and the discrete topology
on~$\cN$.
 Equivalently, a right $R$\+module $\cN$ is discrete if and only if
the annihilator of every element of $\cN$ is an open right ideal in~$R$.

 The full subcategory of discrete right $R$\+modules is denoted by
$\Discr R\subset\Modr R$.
 The full subcategory $\Discr R$ is closed under subobjects, quotients,
and infinite direct sums in $\Modr R$.
 In other words, one can say that $\Discr R$ is a \emph{hereditary
pretorsion class} in $\Modr R$ \,\cite[Sections~VI.1 and~VI.4]{Sten}.
 It follows that $\Discr R$ is a Grothendieck abelian category.
 The inclusion functor $\Discr R\rarrow\Modr R$ is exact and preserves
infinite direct sums (but \emph{not} infinite products).

 Let $\fR$ be the completion of~$R$.
 Then the $R$\+module structure on any discrete right $R$\+module $\cN$
can be extended to a discrete right $\fR$\+module structure on $\cN$ in
a unique way.
 Hence the functor of restriction of scalars $\Modr\fR\rarrow\Modr R$
induces an equivalence of categories $\Discr\fR\simeq\Discr R$.

 For any abelian group $A$ and any set $X$, we denote by $A[X]=A^{(X)}$
the direct sum of copies of $A$ indexed by the elements of~$X$.
 The elements of $A[X]$ are interpreted as finite formal linear
combinations $\sum_{x\in X}a_xx$ of elements of $X$ with
the coefficients $a_x\in A$.
 The word ``finite'' means that $a_x=0$ for all but a finite subset of
indices $x\in X$.

 Let $\fA$ be a complete, separated topological abelian group, and
let $X$ be a set.
 We put $\fA[[X]]=\varprojlim_{\fU\subset\fA}(\fA/\fU)[X]$, where
the projective limit is taken over the directed poset of all open
subgroups $\fU\subset\fA$.
 The elements of $\fA[[X]]$ are interpreted as infinite formal linear
combinations $\sum_{x\in X}a_xx$ of elements of $X$ whose families
of coefficients $a_x\in\fA$ converge to zero in the topology of~$\fA$.
 The latter condition means that, for every open subgroup
$\fU\subset\fA$, the set of indices $x\in X$ for which $a_x\notin\fU$
is finite.

 The assignment $X\longmapsto\fA[[X]]$ is a covariant functor
$\Sets\rarrow\Ab$, but we will mostly consider it as a functor
$\Sets\rarrow\Sets$.
 Specifically, any map of sets $f\:X\rarrow Y$ induces a map of
sets/abelian groups $\fA[[f]]\:\fA[[X]]\rarrow\fA[[Y]]$ given by
the rule $\sum_{x\in X}a_xx\longmapsto\sum_{y\in Y}b_yy$, where
$b_y=\sum_{x\in X}^{f(x)=y}a_x$.
 Here the infinite sum in the latter formula is computed as
the limit of finite partial sums in the topology of~$\fA$.
 The limit is well-defined and exists due to the completeness and
separatedness conditions on $\fA$ and the zero-convergence condition
on the family of elements $(a_x\in\fA)_{x\in X}$.

 Let $\fR$ be a complete, separated right linear topological ring.
 Then, for any set $X$, there is a natural map of sets
$\epsilon_X\:X\rarrow\fR[[X]]$ taking every element $x\in X$
to the formal linear combination $\sum_{y\in X}r_yy$ with
the coefficients $r_x=1$ and $r_y=0$ for $y\ne x$.
 Furthermore, for any set $X$, there is a natural map of sets/abelian
groups $\phi_X\:\fR[[\fR[[X]]]]\rarrow\fR[[X]]$, taking a formal
linear combination of formal linear combinations of elements of $X$
to a formal linear combination of elements of~$X$.
 The ``opening of parentheses'' map~$\phi_X$ is constructed in terms of
the operations of multiplication of pairs of elements in $\fR$ and
infinite sums of zero-convergent families of elements in~$\fR$.
 The conditions that the multiplication map $\fR\times\fR\rarrow\fR$
is continuous and open right ideals form a base of neighborhoods of zero
in $\fR$ guarantee the zero-convergence of the families of elements
whose sums one needs to compute when opening the parentheses.

 The natural transformations~$\epsilon_X$ and~$\phi_X$ define
the structure of a \emph{monad} on the category of sets on the functor
$\fR[[{-}]]$ (we suggest the discussion
in~\cite[Section~6]{PS1} for an introductory exposition).
 One usually speaks about ``algebras over a monad'', but in our context
we prefer to call them \emph{modules over a monad}.
 The category of left $\fR$\+contramodules $\fR\Contra$ is defined
as the category of modules over the monad $\fR[[{-}]]$ on $\Sets$.
 Explicitly, this means that a \emph{left\/ $\fR$\+contramodule} $\fP$
is a set endowed with a map of sets $\pi_\fP\:\fR[[\fP]]\rarrow\fP$,
called the \emph{left\/ $\fR$\+contraaction map}.
 The usual axioms of (contra)unitality and (contra)associativity,
involving the maps~$\epsilon_\fP$ and~$\phi_\fP$, are imposed
on the contraaction map~$\pi_\fP$.

 For any abstract (nontopological) ring $R$, there is a similar monad
structure on the functor $X\longmapsto R[X]$ on the category of sets,
and the category of modules over the monad $R[{-}]$ is naturally
equivalent (in fact, isomorphic) to the category of left
$R$\+modules~\cite[Section~6.1]{PS1}.
 For a complete, separated topological ring $\fR$, the set $\fR[X]$
is naturally embedded into the set $\fR[[X]]$ as the subset of finite
formal linear combinations in the set of zero-convergent infinite ones.
 The inclusion $\fR[X]\rarrow\fR[[X]]$ is a natural transformation that
is a morphism of monads on $\Sets$, and it induces a forgetful functor
$\fR\Contra\rarrow\fR\Modl$, which assigns to every $\fR$\+contramodule
$\fP$ its underlying $\fR$\+module~$\fP$.
 Essentially, the structure of an $\fR$\+contramodule is the structure
of infinite summation operations on $\fP$ with zero-convergent families
of coefficients from~$\fR$; the functor $\fR\Contra\rarrow\fR\Modl$
forgets the infinite summation operations and only keeps
the finite ones.

 The category of left $\fR$\+contramodules $\fR\Contra$ is a locally
presentable abelian category with enough projective objects.
 The forgetful functor $\fR\Contra\rarrow\fR\Modl$ is exact and
preserves infinite products (but \emph{not} infinite direct sums).
 The projective objects of $\fR\Contra$ are precisely the direct
summands of the \emph{free left\/ $\fR$\+contramodules} $\fR[[X]]$,
where $X\in\Sets$.
 Here the left $\fR$\+contramodule structure on $\fR[[X]]$ is given
by the contraaction map $\pi_{\fR[[X]]}=\phi_X$ (this is a special
case of the general construction of the ``free algebra/module over
a monad'').
 For any left $\fR$\+contramodule $\fQ$ and any set $X$, there is
a natural isomorphism of abelian groups
$$
 \Hom^\fR(\fR[[X]],\fQ)\simeq\fQ^X=\Hom_\Sets(X,\fQ),
$$
where $\Hom^\fR({-},{-})$ denotes the group of morphisms in the abelian
category $\fR\Contra$.

 Let $\cN$ be a discrete right $\fR$\+module and $W$ be an abelian
group.
 Then the group of abelian group maps $\fQ=\Hom_\boZ(\cN,W)$ has
a natural structure of left $\fR$\+contramodule given by the formula
similar to the one in Section~\ref{introd-contramodules-subsecn}:
$$
 \pi\left(\sum\nolimits_{q\in\fQ}r_qq\right)(y)=
 \sum\nolimits_{q\in\fQ}q(yr_q)\,\in\,W
$$
for all $y\in\cN$.
 Here $(r_q)_{q\in\fQ}$ is a zero-convergent family of elements
in $\fR$, and the sum in the right-hand side of the formula above
is finite because $yr_q=0$ in $\cN$ for all but a finite subset of
indices $q\in\fQ$.

 Let $\cN$ be a discrete right $\fR$\+module and $\fP$ be a left
$\fR$\+contramodule.
 Then the \emph{contratensor product} $\cN\ocn_\fR\fP$ is an abelian
group constructed as the cokernel of (the difference of) the natural
pair of abelian group maps
$$
 \cN\ot_\boZ\fR[[\fP]]\,\rightrightarrows\,\cN\ot_\boZ\fP.
$$
 Here the first map $\cN\ot_\boZ\fR[[\fP]]\rarrow\cN\ot_\boZ\fP$ is
induced by the contraaction map $\pi_\fP\:\fR[[\fP]]\rarrow\fP$.
 The second map $\cN\ot_\boZ\fR[[\fP]]\rarrow\cN\ot_\boZ\fP$,
constructed in terms of the discrete right action of $\fR$ in $\cN$,
is given by the formula
$$
 y\ot\sum\nolimits_{p\in\fP}r_pp\,\longmapsto\,
 \sum\nolimits_{p\in\fP}yr_p\ot p.
$$
 The sum in the right-hand side is finite because $yr_p=0$ in $\cN$
for all but a finite subset of indices $p\in\fP$.

 For any discrete right $\fR$\+module $\cN$, andy left
$\fR$\+contramodule $\fP$, and any abelian group $W$, there is
a natural adjunction isomorphism of abelian groups
$$
 \Hom_\boZ(\cN\ocn_\fR\fP,\>W)\simeq
 \Hom^\fR(\fP,\Hom_\boZ(\cN,W)).
$$
 For any discrete right $\fR$\+module $\cN$ and any set $X$, there
is a natural isomorphism of abelian groups
$$
 \cN\ocn_\fR\fR[[X]]\simeq\cN[X]=\cN^{(X)}.
$$
 The contratensor product functor
$$
 \ocn_\fR\:\Discr\fR\times\fR\Contra\lrarrow\Ab
$$
preserves all colimits in both of its arguments.

 Let $\fP$ be a left $\fR$\+contramodule and $\fA\subset\fR$ be
a closed subgroup in the topological abelian group~$\fR$.
 Then we denote by $\fA\tim\fP\subset\fP$ the image of the composition
$$
 \fA[[\fP]]\lrarrow\fR[[\fP]]\lrarrow\fP
$$
of the natural inclusion $\fA[[\fP]]\hookrightarrow\fR[[\fP]]$
with the contraaction map $\pi_\fP\:\fR[[\fP]]\rarrow\fP$.
 Denoting by $A\cdot P\subset P$ the subgroup spanned by the products
$ap\in P$ ($a\in A$, \,$p\in P$) for any given ring $R$,
left $R$\+module $P$, and additive subgroup $A\subset R$, one obviously
has $\fA\cdot\fP\subset\fA\tim\fP$.

 For any closed left ideal $\fJ\subset\fR$ and any left
$\fR$\+contramodule $\fP$, the subgroup $\fJ\tim\fP$ is
a left $\fR$\+subcontramodule in~$\fP$.
 For any closed right ideal $\fJ\subset\fR$ and any set $X$, one has
$$
 \fJ\tim\fR[[X]]=\fJ[[X]]\subset\fR[[X]].
$$
 For any open right ideal $\fI\subset\fR$, there is a natural
isomorphism of abelian groups
$$
 (\fR/\fI)\ocn_\fR\fP\simeq\fP/(\fI\tim\fP).
$$

 For any open two-sided ideal $\fI\subset\fR$, the left
$\fR$\+contramodule structure on the quotient contramodule
$\fP/(\fI\tim\fP)$ comes from a left module structure over
the discrete quotient ring $\fR/\fI$.
 In this context, for any right $\fR/\fI$\+module $N$, there is
a natural isomorphism of abelian groups
$$
 N\ocn_\fR\fP\simeq N\ot_{\fR/\fI}\bigl(\fP/(\fI\tim\fP)\bigr).
$$

\Section{Separated and Flat Contramodules}
\label{separated-and-flat-secn}

 This section is based on the results of~\cite[Section~E.1]{Pcosh}
and~\cite[Section~6]{PR}.
 Another exposition can be found in~\cite[Section~8]{Pflcc}.

 Let $\fR$ be a complete, separated right linear topological ring.
 A left $\fR$\+contramodule $\fF$ is said to be \emph{flat} if
the contratensor product functor
$$
 {-}\ocn_\fR\fF\:\Discr\fR\lrarrow\Ab
$$
is exact.
 Since the direct limit functors are exact in the abelian category
$\Ab$, the class of flat left $\fR$\+contramodules is closed under
direct limits in $\fR\Contra$.

 For a complete, separated two-sided linear topological ring $\fR$,
all discrete right $\fR$\+modules are direct limits of right modules
over the discrete quotient rings of $\fR$, and all short exact sequences
of discrete right $\fR$\+modules are direct limits of short exact
sequences of right modules over the discrete quotient rings of~$\fR$.
 It follows that a left $\fR$\+contramodule $\fF$ is flat if and only
if the left $\fR/\fI$\+module $\fF/(\fI\tim\fF)$ is flat for every open
two-sided ideal $\fI\subset\fR$.
 It suffices to check this condition for ideals $\fI$ belonging to
any chosen base of neighborhoods of zero in $\fR$ consisting of open
two-sided ideals.

 For any $\fR$\+contramodule $\fP$, consider the projective limit of
abelian groups
$$
 \Lambda_\fR(\fP)=\varprojlim\nolimits_{\fI\subset\fR}
 \fP/(\fI\tim\fP)
$$
taken over the directed poset of all open right ideals $\fI\subset\fR$.
 The abelian group $\Lambda_\fR(\fP)$ comes together with the natural
\emph{completion map} $\lambda_{\fR,\fP}\:\fP\rarrow\Lambda_\fR(\fP)$.
 In fact, there is a natural left $\fR$\+contramodule structure on
$\Lambda_\fR(\fP)$, making~$\lambda_{\fR,\fP}$ an $\fR$\+contramodule
morphism~\cite[Proposition~5.2]{Pcoun}.
 Put $\Omega_\fR(\fP)=\ker(\fP\to\Lambda_\fR(\fP))=
\bigcap_{\fI\subset\fR}(\fI\tim\fP)\subset\fP$.
 It follows that $\Omega_\fR(\fP)$ is an $\fR$\+subcontramodule
of~$\fP$.

 The left $\fR$\+contramodule $\fP$ is said to be \emph{separated} if
$\lambda_{\fR,\fP}$~is an injective map, and $\fP$ is said to be
\emph{complete} if the map~$\lambda_{\fR,\fP}$ is surjective.
 The left $\fR$\+contramodule $\Lambda_\fR(\fP)$ is always separated.
 The subcontramodule $\Omega_\fR(\fP)\subset\fP$ is called
the \emph{nonseparatedness kernel} of~$\fP$.

\begin{lem} \label{complete-separated-flat-contramodules-lemma}
 Let\/ $\fR$ be a complete, separated right linear topological ring
with a \emph{countable} base of neighborhoods of zero.
 Then \par
\textup{(a)} all left\/ $\fR$\+contramodules are complete; \par
\textup{(b)} for any nonzero left\/ $\fR$\+contramodule\/ $\fP$,
there exists an open right ideal\/ $\fI\subset\fR$ such that\/
$\fI\tim\fP\varsubsetneq\fP$; \par
\textup{(c)} all flat left\/ $\fR$\+contramodules are separated.
\end{lem}

\begin{proof}
 Part~(a) is~\cite[Lemma~6.3(b)]{PR} or~\cite[Theorem~5.3]{Pcoun}.
 Part~(b) is~\cite[Lemma~6.14]{PR}.
 Part~(c) is~\cite[Corollary~6.15]{PR}.
 The case of a two-sided linear topological ring $\fR$ is covered
by~\cite[Lemmas~E.1.1 and~E.1.2, and Corollary~E.1.7]{Pcosh}.
 See~\cite[Lemmas~8.1, 8.2, and~8.3]{Pflcc} for further details
and references.
\end{proof}

 The now-classical counterexample of a nonseparated contramodule over
the topological ring $k[[x]]$ of formal power series in one variable~$x$
over a field~$k$ (with the $x$\+adic topology) or the topological ring
of $p$\+adic integers $\boZ_p$ (with the $p$\+adic topology) is
discussed, under various guises, in~\cite[Example~2.5]{Sim},
\cite[Example~3.20]{Yek0}, \cite[Example~4.33]{PSY},
\cite[Section~A.1.1]{Psemi}, \cite[Section~1.5]{Prev},
\cite[Example~2.7(1)]{Pcta}.

\begin{cor} \label{separated-quotient-contratensor-flatness-cor}
 Let\/ $\fR$ be a complete, separated right linear topological ring with
a \emph{countable} base of neighborhoods of zero.
 In this setting: \par
\textup{(a)} For any discrete right\/ $\fR$\+module\/ $\cN$ and any
left\/ $\fR$\+contramodule\/ $\fP$, the surjective\/
$\fR$\+contramodule morphism\/ $\fP\rarrow\Lambda_\fR(\fP)$ induces
an isomorphism
$$
 \cN\ocn_\fR\fP\simeq\cN\ocn_\fR\Lambda_\fR(\fP).
$$ \par
\textup{(b)} A left\/ $\fR$\+contramodule\/ $\fF$ is flat if and only
if the left\/ $\fR$\+contramodule\/ $\Lambda_\fR(\fF)$ is flat.
 If this is the case, then the morphism\/ $\fF\rarrow\Lambda_\fR(\fF)$
is an isomorphism and\/ $\Omega_\fR(\fF)=0$.
\end{cor}

\begin{proof}
 Part~(a): it follows from the constructions that the inclusion of
$\fR$\+con\-tra\-mod\-ules $\Omega_\fR(\fP)\rarrow\fP$ induces a zero
map $\cN\ocn_\fR\Omega_\fR(\fP)\rarrow\cN\ocn_\fR\fP$.
 Indeed, if $y\in\cN$ and $p\in\fP$ are elements for which there
exists an open right ideal $\fI\subset\fR$ such that $y\fI=0$
and $p\in\fI\tim\fP$, then the element $y\ot p$ is annihilated
by the surjective map $\cN\ot_\boZ\fP\rarrow\cN\ocn_\fR\fP$.
 This observation does not depend on the assumption of a countable
base of neighborhoods of zero in $\fR$ yet.

 Now under the latter assumption, we have a right exact sequence of
$\fR$\+con\-tra\-mod\-ules $\Omega_\fR(\fP)\rarrow\fP\rarrow
\Lambda_\fR(\fP)\rarrow0$ by
Lemma~\ref{complete-separated-flat-contramodules-lemma}(a).
 Since the functor $\cN\ocn_\fR{-}\,\:\fR\Contra\rarrow\Ab$ is right
exact, the desired isomorphism of contratensor products follows.

 Part~(b): the first assertion follows immediately from part~(a).
 The second assertion is
Lemma~\ref{complete-separated-flat-contramodules-lemma}(c).
\end{proof}

 We denote the full subcategories of flat left $\fR$\+contramodules
and separated left $\fR$\+contramodules by $\fR\Flat\subset
\fR\Separ\subset\fR\Contra$.
 The category of separated left $\fR$\+contramodules $\fR\Separ$ is
convenient to work with for various technical purposes, but it is
\emph{not} well-behaved as a category or as a full subcategory in
$\fR\Contra$ (see~\cite[Example~2.5]{Sim}, \cite[Example~2.7(1)]{Pcta},
and~\cite[Remark~8.5]{Pflcc}).
 Perhaps the only good properties that the full subcategory $\fS\Separ
\subset\fS\Contra$ has are that it is closed under subobjects and
infinite products (hence under all limits).
 The full subcategory of flat left $\fR$\+contramodules $\fR\Flat
\subset\fR\Contra$ is well-behaved, on the other hand.

\begin{lem} \label{flat-contramodules-well-behaved-lemma}
 Let\/ $\fR$ be a complete, separated right linear topological ring with
a \emph{countable} base of neighborhoods of zero.
 In this context: \par
\textup{(a)} The full subcategory of flat\/ $\fR$\+contramodules\/
$\fR\Flat$ is closed under extensions and kernels of epimorphisms
in\/ $\fR\Contra$.
 All projective left\/ $\fR$\+contramodules are flat. \par
\textup{(b)} For any discrete right\/ $\fR$\+module\/ $\cN$,
the contratensor product functor\/ $\cN\ocn_\fR{-}$ takes short exact
sequences of flat left\/ $\fR$\+contramodules to short exact sequences
of abelian groups.
 Moreover, for any short exact sequence of left\/ $\fR$\+contramodules\/
$0\rarrow\fQ\rarrow\fP\rarrow\fF\rarrow0$ with a flat left\/
$\fR$\+contramodule\/ $\fF$, the short sequence of abelian groups\/
$0\rarrow\cN\ocn_\fR\fQ\rarrow\cN\ocn_\fR\fP\rarrow\cN\ocn_\fR\fF
\rarrow0$ is exact.
\end{lem}

\begin{proof}
 This is~\cite[Corollary~6.8, Lemmas~6.9 and~6.10, and
Corollaries~6.13 and~6.15]{PR}.
 The case of a two-sided linear topological ring $\fR$ is
partly covered by~\cite[Lemmas~E.1.4, E.1.5, E.1.6, and
Corollary~E.1.7]{Pcosh}.
 See also~\cite[Lemma~8.4]{Pflcc}.
\end{proof}

 The full subcategory of injective discrete right $\fR$\+modules will
be denoted by $\Discrinj\fR\subset\Discr\fR$, and the full subcategory
of projective left $\fR$\+contramodules by $\fR\Contra_\prj\subset
\fR\Contra$.

 Let $\fR$ be a complete, separated right linear topological ring
with a countable base of neighborhoods of zero.
 Then there exists a descending chain of open right ideals
$\fR\supset\fI_1\supset\fI_2\supset\fI_3\supset\dotsb$ in $\fR$,
indexed by the integers $n\ge1$, such that the ideals $\fI_n$ form
a base of neighborhoods of zero in~$\fR$.
 So we have $\fR=\varprojlim_{n\ge1}\fR/\fI_n$, and the topology on
$\fR$ is the topology of projective limit of discrete abelian
groups $\fR/\fI_n$.

 Now let $\fR$ be a complete, separated two-sided linear topological
ring.
 Then one can choose $\fI_n$ to be open two-sided ideals in~$\fR$.
 Denote the related discrete quotient rings by $R_n=\fR/\fI_n$.
 Then $\fR=\varprojlim_{n\ge1}R_n$ is the projective limit of
the projective system of rings $(R_n)_{n\ge1}$, and the topology on
$\fR$ is the topology of projective limit of discrete rings~$R_n$.

 By a \emph{right $(R_n)$\+system} we will mean a collection of right
$R_n$\+modules $N_n$ endowed with isomorphisms of right $R_n$\+modules
$N_n\simeq\Hom_{R_{n+1}^\rop}(R_n,N_{n+1})$ for all $n\ge1$.
 A right $(R_n)$\+system $(K_n)_{n\ge1}$ is said to be \emph{injective}
if the $R_n$\+module $K_n$ is injective for every $n\ge1$ (or
equivalently, for infinitely many values of~$n$).
 We will denote the category of right $(R_n)$\+systems by
$\Sysr(R_n)$ and the full subcategory of injective right
$(R_n)$\+systems by $\Sysrinj(R_n)\subset\Sysr(R_n)$.

 Dual-analogously, a \emph{left $(R_n)$\+system} is a collection of
left $R_n$\+modules $P_n$ endowed with isomorphisms of left
$R_n$\+modules $P_n\simeq R_n\ot_{R_{n+1}}P_{n+1}$ for all $n\ge1$.
 A left $(R_n)$\+system $(P_n)_{n\ge1}$ is said to be \emph{projective}
if the $R_n$\+module $P_n$ is projective for every $n\ge1$ (or
equivalently, for infinitely many values of~$n$).
 Similarly, a left $(R_n)$\+system $(F_n)_{n\ge1}$ is said to be
\emph{flat} if the $R_n$\+module $F_n$ is flat for every $n\ge1$
(or equivalently, for infinitely many values of~$n$).
 We will denote the category of left $(R_n)$\+systems by
$(R_n)\Sys$ and the full subcategories of projective and flat
left $(R_n)$\+systems by $(R_n)\Sys_\prj\subset(R_n)\Sys_\fl
\subset\Sysr(R_n)$.

\begin{lem} \label{left-right-systems-discr-contra-category-equivs}
 Let $(R_n)_{n\ge1}$ be a projective system of surjective morphisms of
rings, indexed by the integers $n\ge1$, and let\/
$\fR=\varprojlim_{n\ge1}R_n$ be its projective limit, endowed with
the topology of projective limit of discrete rings~$R_n$.
 Denote by\/ $\fI_n\subset\fR$ the kernel of the natural surjective
ring homomorphism\/ $\fR\rarrow R_n$.
 In this context: \par
\textup{(a)} The category of discrete right\/ $\fR$\+modules is
naturally equivalent to the category of right $(R_n)$\+systems,
$$
 \Discr\fR\simeq\Sysr(R_n).
$$
 To every discrete right\/ $\fR$\+module\/ $\cN$, the right
$(R_n)$\+system $(N_n)_{n\ge1}$ with $N_n=\Hom_\fR(R_n,\cN)$
is assigned.
 To every right $(R_n)$\+system $(N_n)_{n\ge1}$, the discrete right\/
$\fR$\+module\/ $\cN=\varinjlim_{n\ge1}N_n$ is assigned.
 This equivalence of abelian categories identifies the full subcategory
of injective objects in\/ $\Discr\fR$ with the full subcategory of
injective right $(R_n)$\+systems in\/ $\Sysr(R_n)$,
$$
 \Discrinj\fR\simeq\Sysrinj(R_n).
$$ \par
\textup{(b)} The category of separated left\/ $\fR$\+contramodules is
naturally equivalent to the category of left $(R_n)$\+systems,
$$
 \fR\Separ\simeq(R_n)\Sys.
$$
 To every separated left\/ $\fR$\+contramodule\/ $\fP$, the left
$(R_n)$\+system $(P_n)_{n\ge1}$ with $P_n=\fP/(\fI_n\tim\fP)$
is assigned.
 To every left $(R_n)$\+system $(P_n)_{n\ge1}$, the separated left\/
$\fR$\+contramodule\/ $\fP=\varprojlim_{n\ge1}P_n$ is assigned.
 This equivalence of categories identifies the full subcategory of flat
left\/ $\fR$\+contramodules with the full subcategory of flat left
$(R_n)$\+systems in $(R_n)\Sys$,
$$
 \fR\Flat\simeq(R_n)\Sys_\fl.
$$
 It also identifies the full subcategory of projective objects in\/
$\fR\Contra$ with the full subcategory of projective left
$(R_n)$\+systems in $(R_n)\Sys$,
$$
 \fR\Contra_\prj\simeq(R_n)\Sys_\prj.
$$
\end{lem}

\begin{proof}
 The first assertion of part~(a) is essentially obvious.
 For the assertion concerning injectivity, one can use the suitable
version of Baer criterion of injectivity for discrete right
$\fR$\+modules.
 See the discussion in~\cite[Section~E.2]{Pcosh}.
 A generalization to right linear topological rings $\fR$ can be found
in~\cite[Proposition~5.1]{Pcoun}.

 The first assertion of part~(b) is~\cite[Lemma~E.1.3]{Pcosh}.
 The assertion concerning flatness follows from the first one in view
of Lemma~\ref{complete-separated-flat-contramodules-lemma}(c) above.
 The assertion about projectivity is~\cite[Corollary~E.1.10(a)]{Pcosh}.
 A generalization to right linear topological rings $\fR$ can be found
in~\cite[Corollary~6.4]{PR} or~\cite[Theorem~5.3]{Pcoun}.
\end{proof}

\Section{Taut and Strongly Right Taut Maps of Topological Rings}
\label{taut-maps-secn}

 We start with the general setting of right linear topological rings.

\begin{lem} \label{continuous-ring-maps-lemma}
 Let $R$ and $S$ be right linear topological rings, and let
$f\:R\rarrow S$ be a ring map.
 Then the following three conditions are equivalent:
\begin{enumerate}
\item $f$~is continuous;
\item the open right ideals $J\subset S$ for which the right ideal $I=f^{-1}(J)\subset R$ is open in $R$ form a base of neighborhoods
of zero in~$S$;
\item any discrete right $S$\+module is also discrete as a right
$R$\+module.
\end{enumerate}
\end{lem}

\begin{proof}
 (1)~$\Longleftrightarrow$~(2) Obvious.

 (1)~$\Longrightarrow$~(3) Let $\cN$ be a discrete right $S$\+module
and $y\in\cN$ be an element.
 Let $J$ be the annihilator of~$y$ in~$S$; then $J$ is an open
right ideal in~$S$.
 The annihilator of~$y$ in $R$ is the right ideal $f^{-1}(J)\subset R$.
 Since $f$~is continuous, $f^{-1}(J)$ is an open right ideal in~$R$.

 (3)~$\Longrightarrow$~(1) Let $J\subset S$ be an open right ideal.
 Then the right $R$\+module $\cN=R/J$ is discrete.
 The ideal $J\subset S$ is the annihilator of the coset $y=1+J\in\cN$
in~$S$.
 Following the previous paragraph, $f^{-1}(J)$ is the annihilator of~$y$
in~$R$.
 By assumption, $\cN$ is a discrete right $R$\+module; so $f^{-1}(J)$
is an open right ideal in~$R$.
 As open right ideals form a base of neighborhoods of zero in $S$
and the preimage of any such ideal is open in $R$, it follows that
$f$~is a continuous map.
\end{proof}

 In the following two lemmas, topological right modules over right
linear topological rings are considered.
 For this reason, and in order to avoid any ambiguity, we refrain from
relying on our usual terminology of ``discrete modules'' over
topological rings, and utilize the longer expression ``endowed with
the discrete topology, remains/becomes a topological module'' instead.

\begin{lem} \label{taut-ring-maps-lemma}
 Let $R$ and $S$ be right linear topological rings, and let
$f\:R\rarrow S$ be a ring map.
 Consider the following three conditions:
\begin{enumerate}
\item for any open right ideal $I\subset R$, the closure $J$ of
the right ideal $f(I)S\subset S$ is an open right ideal in~$S$;
\item the open right ideals $I\subset R$ for which the closure $J$ of
the right ideal $f(I)S\subset S$ is an open right ideal in $S$ form
a base of neighborhoods of zero in~$R$;
\item a separated topological right $S$\+module remains a topological
right $S$\+module when endowed with the discrete topology whenever,
viewed as right $R$\+module and endowed with a discrete topology,
it becomes a topological right $R$\+module.
\end{enumerate}
 The equivalence and implication \textup{(1)~$\Longleftrightarrow$~(2)}
\textup{$\Longrightarrow$~(3)} always hold.
 When the ring $S$ is commutative, one has
\textup{(1)~$\Longleftrightarrow$~(2)}
\textup{$\Longleftrightarrow$~(3)}.
\end{lem}

\begin{proof}
 (1)~$\Longleftrightarrow$~(2) Obvious.

 (1)~$\Longrightarrow$~(3) Let $\rN$ be a separated topological
right $S$\+module and $y\in\rN$ be an element.
 Since the action map $\rN\times S\rarrow\rN$ is continuous,
the embedding $S=\{y\}\times S\rarrow\rN\times S$ is continuous,
and $0\in\rN$ is a closed point, it follows that the annihilator
$J'\subset S$ of~$y$ in $S$ is a closed right ideal in~$S$.
 Now $I=f^{-1}(J')$ is the annihilator of~$y$ in~$R$.
 By assumption, $I$ is an open right ideal in~$R$.
 We have $f(I)S\subset J'$ and $J'$ is a closed right ideal in $S$,
hence the closure $J$ of $f(I)S$ in $S$ is contained in $J'$, that is
$J\subset J'$.
 By~(1), \,$J$ is an open right ideal in~$S$; hence $J'$ is an open
right ideal in $S$, too.

 (3)~$\Longrightarrow$~(1) Consider the cyclic $S$\+module $\rN=S/J$,
and endow it with the quotient topology.
 Then $\rN$ is separated topological $S$\+module (since
$J\subset S$ is a closed ideal).
 The annihilator of any element $y\in\rN$ in $S$ contains the ideal
$J\subset S$, hence the annihilator of~$y$ in $R$ contains
the open ideal $I\subset R$.
 Thus, endowed with the discrete topology, $\rN$ becomes a topological
$R$\+module.
 By~(3), it follows that, endowed with the discrete topology, $\rN$ is
also a topological $S$\+module.
 Hence the annihilator $J$ of the coset $1+J\in\rN$ is an open ideal
in~$S$.
\end{proof}

 We will say that a ring map $f\:R\rarrow S$ is \emph{taut}
(cf.~\cite[Section Tag~0GX1]{SP}) if it satisfies any one of
the equivalent conditions of Lemma~\ref{taut-ring-maps-lemma}(1) or~(2).
 Let us emphasize that, in our terminology, a taut ring map
\emph{need not be continuous}.

\begin{lem} \label{continuous-taut-ring-maps-lemma}
 Let $R$ and $S$ be right linear topological rings, and
let $f\:R\rarrow S$ be a ring map.
 Consider the following four conditions:
\begin{enumerate}
\item $f$~is continuous and taut;
\item for any open right ideal $I\subset R$, the closure $J$ of
the right ideal $f(I)S\subset S$ is an open right ideal in $S$, and
such open right ideals $J\subset S$ form a base of neighborhoods
of zero in~$S$;
\item open right ideals $I\subset R$ for which the closure $J$ of
the right ideal $f(I)S\subset S$ is an open right ideal in $S$ form
a base of neighborhoods of zero in $R$, and such open right ideals
$J\subset S$ form a base of neighborhoods of zero in~$S$;
\item a separated topological right $S$\+module remains a topological
right $S$\+module when endowed with the discrete topology if and only
if, viewed as right $R$\+module and endowed with a discrete topology,
it becomes a topological right $R$\+module.
\end{enumerate}
 The equivalences and implication \textup{(1)~$\Longleftrightarrow$~(2)}
\textup{$\Longleftrightarrow$~(3)} \textup{$\Longrightarrow$~(4)} always
hold.
 When the ring $S$ is commutative, one has
\textup{(1)~$\Longleftrightarrow$~(2)}
\textup{$\Longleftrightarrow$~(3)}
\textup{$\Longleftrightarrow$~(4)}.
\end{lem}

\begin{proof}
 (1)~$\Longrightarrow$~(2) Let $J'\subset S$ be an open right ideal
in~$S$.
 Then $I=f^{-1}(J')\subset R$ is an open right ideal in $R$
(since $f$~is continuous).
 Hence $J'$ contains the closure $J$ of the right ideal
$f(I)S\subset S$.
 The right ideal $J\subset S$ is open (since $f$~is taut).
 Thus the open right ideals $J$ form a base of neighborhoods of zero
in~$S$.

 (2)~$\Longrightarrow$~(1) Obviously, (2)~implies that $f$~is taut.
 Now let $J'\subset S$ be an open right ideal.
 By~(2), there exists an open right ideal $I\subset R$ such that $J'$
contains the closure $J$ of the right ideal $f(I)S\subset S$.
 Hence $I$ is contained in the right ideal $I'=f^{-1}(J')\subset R$.
 Thus the right ideal $I'\subset R$ is open, too; so $f$~is continuous.

 (2)~$\Longleftrightarrow$~(3) Obvious.

 (1)~$\Longrightarrow$~(4) The condition that $f$~is continuous
implies the ``only if'' implication in~(4) by
Lemma~\ref{continuous-ring-maps-lemma}\,(1)\,$\Rightarrow$\,(3).
 The condition that $f$~is taut implies the ``if'' implication in~(4)
by Lemma~\ref{taut-ring-maps-lemma}\,(1)\,$\Rightarrow$\,(3).

 (4)~$\Longrightarrow$~(1): the condition that $f$~is continuous
follows from the ``only if'' implication in~(4) by
Lemma~\ref{continuous-ring-maps-lemma}\,(3)\,$\Rightarrow$\,(1)
(apply the ``only if'' implication in~(4) to a topological right
$S$\+module with discrete topology).
 The condition that $f$~is taut follows from the ``if'' implication
in~(4) by Lemma~\ref{taut-ring-maps-lemma}\,(3)\,$\Rightarrow$\,(1).
\end{proof}

 Now we pass to the more restrictive setting of \emph{two-sided
linear} topological rings $R$ and~$S$.

\begin{lem} \label{reductions-tensor-product-formula}
 Let $R$ and $S$ be right linear topological rings, and
let $f\:R\rarrow S$ be a ring map.
 Let $I''\subset I'\subset R$ be two open two-sided ideals in $R$ such
that the closure $J'$ of the right ideal $f(I')S\subset S$ is an open
two-sided ideal in $S$ and the closure $J''$ of the right ideal
$f(I'')S\subset S$ is also an open two-sided ideal in~$S$.
 Denote the related discrete quotient rings by $R'=R/I'$, \,$R''=R/I''$,
\,$S'=S/J'$, and $S''=S/J''$.
 Then the induced $R'$\+$S''$\+bimodule map $R'\ot_{R''}S''\rarrow S'$
is an isomorphism.
\end{lem}

\begin{proof}
 Denote by $f'\:R'\rarrow S'$ and $f''\:R''\rarrow S''$ the maps of
quotient rings induced by~$f$.
 As $J''$ is an open two-sided ideal in $S$, we have
$J'/J''=(f(I')S+J'')/J''=f''(I'/I'')S''\subset S''$.
 Hence $S'=S''/(f''(I'/I'')S'')$.
 This means precisely that the map $R'\ot_{R''}S''\rarrow S'$ is
an isomorphism.
\end{proof}

\begin{lem} \label{continuous-strongly-taut-ring-map}
 Let $R$ and $S$ be two-sided linear topological rings, and
let $f\:R\rarrow S$ be a ring map.
 Consider the following three conditions:
\begin{enumerate}
\item $f$~is continuous and taut;
\item $f$~is continuous, and open two-sided ideals $I\subset R$ for
which the closure $J$ of the right ideal $f(I)S\subset S$ is an open
two-sided ideal in $S$ form a base of neighborhoods of zero in~$R$;
\item open two-sided ideals $I\subset R$ for which the closure $J$ of
the right ideal $f(I)S\subset S$ is an open two-sided ideal in $S$ form
a base of neighborhoods of zero in $R$, and such open two-sided ideals
$J\subset S$ form a base of neighborhoods of zero in~$S$.
\end{enumerate}
 The implication and equivalence \textup{(1)~$\Longleftarrow$~(2)}
\textup{$\Longleftrightarrow$~(3)}  always hold.
 When the image of the map~$f$ is contained in the center of
the ring $S$, one has \textup{(1)~$\Longleftrightarrow$~(2)}
\textup{$\Longleftrightarrow$~(3)}.
\end{lem}

\begin{proof}
 (2)~$\Longrightarrow$~(3) Let $J'\subset S$ be an open two-sided ideal
in~$S$.
 Then $I'=f^{-1}(J')\subset R$ is an open two-sided ideal in $R$
(since $f$~is continuous).
 By~(2), there exists an open two-sided ideal $I\subset I'$ for which
the closure $J$ of the right ideal $f(I)S\subset S$ is an open
two-sided ideal in~$S$.
 Since $J'\subset S$ is an open two-sided ideal and $f(I)S\subset
f(I')S\subset J'$, we have $J\subset J'$, as desired.

 (3)~$\Longrightarrow$~(2) We only need to prove that $f$~is continuous.
 Let $J'\subset S$ be an open two-sided ideal.
 By~(3), there exists an open two-sided ideal $I\subset R$ such that
the closure $J$ of the right ideal $f(I)S\subset S$ is an open
two-sided ideal in $S$ and $J\subset J'$.
 Hence $I$ is contained in the two-sided ideal $I'=f^{-1}(J')\subset R$.
 Thus the two-sided ideal $I'\subset R$ is open, too; so $f$~is
continuous.

 (3)~$\Longrightarrow$~(1) Follows from
Lemma~\ref{continuous-taut-ring-maps-lemma}\,(3)\,$\Rightarrow$\,(1).

 (1)~$\Longrightarrow$~(2) Follows from
Lemma~\ref{continuous-taut-ring-maps-lemma}\,(1)\,$\Rightarrow$\,(2).
\end{proof}

 Let $R$ and $S$ be two-sided linear topological rings.
 We will say that a ring map $f\:R\rarrow S$ is \emph{strongly right
taut} if open two-sided ideals $I\subset R$ for which the closure $J$
of the right ideal $f(I)S\subset S$ is an open two-sided ideal in $S$
form a base of neighborhoods of zero in~$R$.
 A ring map $f\:R\rarrow S$ is strongly right taut \emph{and}
continuous if and only if it satisfies any one of the equivalent
conditions of Lemma~\ref{continuous-strongly-taut-ring-map}(2) or~(3).

 Now we impose the additional assumption of a \emph{countable} base of
neighborhoods of zero in~$R$.

\begin{lem} \label{continuous-strongly-taut-ring-map-countable-base}
 Let $R$ and $S$ be two-sided linear topological rings, and
let $f\:R\rarrow S$ be a ring map.
 Assume that the topological ring $R$ has a countable base of
neighborhoods of zero.
 Then the following two conditions are equivalent:
\begin{enumerate}
\item $f$~is continuous and strongly right taut;
\item there exist descending chains of open two-sided ideals
$R\supset I_1\supset I_2\supset I_3\supset\dotsb$ and
$S\supset J_1\supset J_2\supset J_3\supset\dotsb$ in $R$ and $S$,
indexed by the positive integers, such that the open two-sided ideal
$I_n$ form a base of neighborhoods of zero in $R$, the open two-sided
ideals $J_n$ form a base of neighborhoods of zero in $S$, one has
$f(I_n)\subset J_n$, and, using the notation $R_n=R/I_n$ and $S_n=S/J_n$
for the discrete quotient rings, the induced $R_n$\+$S_{n+1}$-bimodule
map $R_n\ot_{R_{n+1}}S_{n+1}\rarrow S_n$ is an isomorphism
for all $n\ge1$. 
\end{enumerate}
\end{lem}

\begin{proof}
 (1)~$\Longrightarrow$~(2) The condition of
Lemma~\ref{continuous-strongly-taut-ring-map}(3) defines a certain
base of neighborhoods of zero in~$R$ (consisting of open two-sided
ideals).
 Since $R$ has a countable base of neighborhoods of zero,
any base of neighborhoods of zero in $R$ contains a descending chain
$I_1\supset I_2\supset I_3\supset\dotsb$ indexed by the positive
integers and also forming a base of neighborhoods of zero in~$R$.
 It follows that the corresponding open two-sided ideals
$J_n\subset S$ produced by the construction spelled out in
Lemma~\ref{continuous-strongly-taut-ring-map}(3) also form a chain
$J_1\supset J_2\supset J_3\supset\dotsb$ indexed by the positive
integers that is a base of neighborhoods of zero in~$S$.
 So $J_n$ is the closure of the right ideal $f(I_n)S$ in $S$ for
every $n\ge1$.
 Now it remains to refer to
Lemma~\ref{reductions-tensor-product-formula} for the assertion
that the map $R_n\ot_{R_{n+1}}S_{n+1}\rarrow S_n$ is an isomorphism.

 (2)~$\Longrightarrow$~(1) We need to prove that $J_n$ is the closure
of the right ideal $f(I_n)S$ in~$S$.
 Indeed, denote by $f_n\:R_n\rarrow S_n$ the maps of the quotient rings
induced by~$f$.
 The condition that the map $R_n\ot_{R_{n+1}}S_{n+1}\rarrow S_n$ is
an isomorphism means that $S_n=S_{n+1}/(f_{n+1}(I_n/I_{n+1})S_{n+1})$,
or equivalently, $J_n/J_{n+1}=f_{n+1}(I_n/I_{n+1})S_{n+1}\subset
S_{n+1}$ for every $n\ge1$.
 It follows that
\begin{equation*}
\begin{split}
 J_n &= J_{n+1}+f(I_n)S \\
       &= J_{n+2}+f(I_{n+1})S+f(I_n)S \\
       &= J_{n+2}+f(I_n)S \\
       &= J_{n+3}+f(I_{n+2})S+f(I_n)S \\
       &= J_{n+3}+f(I_n)S \\
       &= \dotsb \\
       &= J_m+f(I_n)S
\end{split}
\end{equation*}
for every $m>n$.
 As the open two-sided ideals $J_m$ form a base of neighborhoods
of zero in $S$, the desired conclusion follows.
\end{proof}

\Section{Left Proflat Morphisms and Proepimorphisms
of~Topological~Rings} \label{proflat-proepi-secn}

 All left proflat maps of topological rings are continuous and
strongly right taut by definition.
 Let us spell out that definition.

\begin{lem} \label{left-proflat-strongly-taut-ring-map}
 Let $R$ and $S$ be two-sided linear topological rings, and
let $f\:R\rarrow S$ be a strongly right taut continuous ring map.
 Then the following two conditions are equivalent:
\begin{enumerate}
\item for every open two-sided ideal $I\subset R$ such that the closure
$J$ of the right ideal $f(I)S\subset S$ is an open two-sided ideal
in $S$, the discrete quotient ring $S/J$ is flat as a left module over
the discrete quotient ring~$R/I$;
\item open two-sided ideals $I\subset R$ for which the closure $J$ of
the right ideal $f(I)S\subset S$ is an open two-sided ideal in $S$
\emph{and} the quotient ring $S/J$ is flat as a left module over
the quotient ring $R/I$ form a base of neighborhoods of zero in~$R$.
\end{enumerate}
\end{lem}

\begin{proof}
 In the notation of Lemma~\ref{reductions-tensor-product-formula},
if $S''$ is flat as a left $R''$\+module, then it follows from
the isomorphism $S'\simeq R'\ot_{R''}S''$ that $S'$ is flat as
a left $R'$\+module.
\end{proof}

 We will say that a strongly right taut continuous ring map
$f\:R\rarrow S$ is \emph{left proflat} if it satisfies the equivalent
conditions of Lemma~\ref{left-proflat-strongly-taut-ring-map}.

 Our next topic is proepimorphisms of topological rings.

\begin{lem} \label{topological-ring-proepimorphism-lemma}
 Let $R$ and $S$ be two-sided linear topological rings, and
let $f\:R\rarrow S$ be a continuous ring map.
 Consider the following two conditions:
\begin{enumerate}
\item for any open two-sided ideals $I\subset R$ and $J\subset S$
such that $f(I)\subset J$, the induced map of the discrete quotient
rings $R/I\rarrow S/J$ is a ring epimorphism;
\item there exist a base of neighborhoods of zero in $R$ consisting
of open two-sided ideals $I_\alpha\subset R$ and a base of neighborhoods
of zero in $S$ consisting of open two-sided ideals $J_\alpha\subset S$,
indexed by the same directed poset of indices\/~$\{\alpha\}$, such that
$f(I_\alpha)\subset J_\alpha$ for every\/~$\alpha$, the induced maps of
the quotient rings $R/I_\alpha\rarrow S/J_\alpha$ are ring epimorphisms,
and $I_{\alpha''}\subset I_{\alpha'}$, \ $J_{\alpha''}\subset
J_{\alpha'}$ for all pairs of indices\/ $\alpha'<\alpha''$.
\end{enumerate}
 The implication \textup{(2)~$\Longrightarrow$~(1)} always holds.
 When the topological rings $R$ and $S$ have countable bases of
neighborhoods of zero, one has \textup{(1)~$\Longleftrightarrow$~(2)}.
 Moreover, one can choose the directed poset\/~$\{\alpha\}$ to be
the linearly ordered set of positive integers in the latter case.
\end{lem}

\begin{proof}
 (2)~$\Longrightarrow$~(1) Pick two indices~$\alpha'$ and~$\alpha''$
such that $I_{\alpha'}\subset I$ and $J_{\alpha''}\subset J$.
 Pick an index~$\alpha$ such that $\alpha\ge\alpha'$ and
$\alpha\ge\alpha''$.
 Then we have $I_\alpha\subset I$, \ $J_\alpha\subset J$, \
$f(I_\alpha)\subset J_\alpha$, and the map $R/I_\alpha\rarrow
S/J_\alpha$ is a ring epimorphism.
 Consider the commutative diagram in the category of rings
$$
 \xymatrix{
  R/I_\alpha \ar[r] \ar@{->>}[d] & S/J_\alpha \ar@{->>}[d] \\
  R/I \ar[r] & S/J
 }
$$
where the vertical maps are the natural surjections and the horizontal
maps are induced by~$f$.
 All surjective maps of rings are, of course, ring epimorphisms.
 Now the maps $R/I_\alpha\rarrow S/J_\alpha\rarrow S/J$ are ring
epimorphisms, hence so is their composition $R/I_\alpha\rarrow S/J$.
 Hence the composition $R/I_\alpha\rarrow R/I\rarrow S/J$ is
a ring epimorphism, and it follows that so is the  map $R/I\rarrow S/J$.

 (1)~$\Longrightarrow$~(2) Let $\{I'_n\subset R\}_{n=1}^\infty$ be
a countable base of open neighborhoods of zero in $R$ and
$\{J'_n\subset S\}_{n=1}^\infty$ be a countable base of open
neighborhoods of zero in~$S$.
 Put $I_0=R$ and $J_0=S$.
 Proceeding by induction on $n\ge1$, pick an open two-sided ideal
$J_n\subset S$ such that $J_n\subset J'_n\cap J_{n-1}$.
 Then pick an open two-sided ideal $I_n\subset R$ such that
$I_n\subset I'_n\cap I_{n-1}\cap f^{-1}(J_n)$.
\end{proof}

 We will say that a continuous homomorphism of two-sided linear
topological rings $f\:R\rarrow S$ is a \emph{topological ring
proepimorphism} if it satisfies the condition of
Lemma~\ref{topological-ring-proepimorphism-lemma}(1).

\begin{ex} \label{proepimorphism-not-epimorphism-example}
 A topological ring proepimorphism \emph{need not} be an epimorphism
of abstract rings (with the topologies disregarded), not even in
the case of Noetherian commutative rings with adic topologies.
 Indeed, let $R=k[x][[y]]$ be the ring of formal power series in
one variable~$y$ with the coefficients in the ring of polynomials
in one variable~$x$ over a field~$k$.
 Let $S=k[x,x^{-1}][[y]]$ be the ring of formal power series in~$y$
with the coefficients in the ring of Laurent polynomials $k[x,x^{-1}]$.
 Endow both the rings $R$ and $S$ with the $y$\+adic topologies, and
consider the natural continuous embedding of adic topological rings
$f\:R\rarrow S$.
 Put $I_n=Ry^n\subset R$ and $J_n=Sy^n\subset S$, and consider
the quotient rings $R_n=R/I_n=k[x,y]/(y^n)$ and
$S_n=S/J_n=k[x,x^{-1},y]/(y^n)$.
 Then the induced ring map $f_n\:R_n\rarrow S_n$ is a flat ring
epimorphism, since $S_n=R_n[x^{-1}]$ is the localization of $R_n$ at
the element $x\in R_n$.
 Furthermore, we have $J_n=f(I_n)S\subset S$ for every $n\ge1$.
 So $f\:R\rarrow S$ is a proflat proepimorphism of topological rings
by Lemmas~\ref{left-proflat-strongly-taut-ring-map}(2)
and~\ref{topological-ring-proepimorphism-lemma}(2).

 We claim that the map~$f$ is \emph{not} an epimorphism in the category
of abstract (nontopological) rings.
 Indeed, denote by $Q=(R\setminus\{0\})^{-1}R$ the field of fractions
of the commutative integral domain~$R$.
 If $f$~were a ring epimorphism, then so would be the pushout
$Q\ot_Rf\:Q\rarrow Q\ot_RS$ of the map~$f$ in the category of
commutative rings.
 Any ring epimorphism whose domain is a field is either an isomorphism,
or the map to the zero ring.
 We have $Q\ot_RS=(R\setminus\{0\})^{-1}S\ne0$, since $S$ is also
a commutative integral domain and $f$~is an injective map.
 Thus, if $f$~were a ring epimorphism, the map $Q\ot_Rf$ would be
an isomorphism.
 This would mean that, for every element $s\in S$, there exist two
elements $p$, $q\in R$ such that $q\ne0$ and $f(q)s=f(p)$ in~$S$.

 Consider the formal power series $s=\sum_{n=0}^\infty x^{-2^n}y^n
\in k[x,x^{-1}][[y]]=S$.
 For the sake of contradiction, assume that there exist two formal
power series $p=\sum_{n=0}^\infty p_n(x)y^n$ and
$q=\sum_{n=0}^\infty q_n(x)y^n$ with polynomial coefficients
$p_n$, $q_n\in k[x]$ such that $q\ne0$ and $qs=p$ in $k[x,x^{-1}][[y]]$.
 Without loss of generality we can assume that $q_0\ne0$.
 Now we have the system of equations in~$k[x]$
$$
 p_n(x)=\sum\nolimits_{m=0}^n x^{-2^m}q_{n-m}(x), \qquad n\ge0,
$$
expressing the equation $qs=p$ in $k[x,x^{-1}][[y]]$.
 Equivalently,
$$
 x^{2^n}p_n(x)=\sum\nolimits_{m=0}^n x^{2^n-2^m}q_{n-m}(x),
 \qquad n\ge0.
$$
 Let $l\ge1$ be an integer such that $q_0(x)$ is not divisible
by~$x^{2^l}$ in~$k[x]$.
 Take $n=l+1$; so we have
$$
 q_0(x)=x^{2^{l+1}}p_{l+1}(x)-\sum\nolimits_{m=0}^l
 x^{2^{l+1}-2^m}q_{l+1-m}(x).
$$
 Then $x^{2^{l+1}}p_{l+1}(x)$ is divisible by~$x^{2^l}$
and $x^{2^{l+1}-2^m}q_{l+1-m}(x)$ is divisible by~$x^{2^l}$
in $k[x]$ for all $0\le m\le l$, but $q_0(x)$ is not divisible
by~$x^{2^l}$; a contradiction.
\end{ex}

 Let $f\:R\rarrow S$ be a continuous homomorphism of right linear
topological rings.
 Then, by Lemma~\ref{continuous-ring-maps-lemma}(3), any discrete right
$S$\+module is also discrete as a right $R$\+module.
 So we have the functor of restriction of scalars $f_\diamond\:
\Discr S\rarrow\Discr R$.

 For the rest of this section, we restrict ourselves to \emph{complete
and separated} topological rings.
 Let $\ff\:\fR\rarrow\fS$ be a continuous homomorphism of complete,
separated right linear topological rings.
 Then, for any set $X$, there is the induced map of sets
$\ff[[X]]\:\fR[[X]]\rarrow\fS[[X]]$.
 The collection of maps $\ff[[X]]$ is a morphism of functors
$\Sets\rarrow\Sets$, and in fact, a morphism $\fR[[{-}]]\rarrow
\fS[[{-}]]$ of monads on the category of sets.
 Consequently, every module over the monad $\fS[[{-}]]$ acquires
the underlying structure of a module over the monad $\fR[[{-}]]$.
 So we have the functor of restriction of scalars $\ff_\sharp\:
\fS\Contra\rarrow\fR\Contra$.

 In particular, (the underlying set of) the topological ring $\fS$
itself has a natural structure of left $\fR$\+contramodule.

\begin{lem} \label{continuous-and-taut-contramodule-tim-lemma}
 Let\/ $\fR$ and\/ $\fS$ be complete, separated right linear topological
rings, and let\/ $\ff\:\fR\rarrow\fS$ be a taut continuous ring map.
 Assume that the topological ring\/ $\fR$ has a countable base of
neighborhoods of zero.
 Let\/ $\fI\subset\fR$ be an open right ideal and\/ $\fJ\subset\fS$
be the closure of the right ideal\/ $\ff(\fI)\fS\subset\fS$.
 Then one has\/ $\fI\tim\fS=\fJ\subset\fS$.
\end{lem}

\begin{proof}
 Let $\fI=\fI_1\supset\fI_2\supset\fI_3\supset\dotsb$ be a descending
chain of open right ideals in $\fR$, indexed by the positive integers,
that is a base of neighborhoods of zero in~$\fR$.
 For every $n\ge1$, denote by $\fJ_n\subset\fS$ the closure of
the right ideal $\ff(\fI_n)\fS\subset\fS$.
 By Lemma~\ref{continuous-taut-ring-maps-lemma}(2), the right ideals
$\fJ=\fJ_1\supset\fJ_2\supset\fJ_3\supset\dotsb$ are open in $\fS$
and form a base of neighborhoods of zero in~$\fS$.
 Similarly to the discussion in the proof of
Lemma~\ref{reductions-tensor-product-formula}, we have $\fJ_n/\fJ_{n+1}
=(\ff(\fI_n)\fS+\fJ_{n+1})/\fJ_{n+1}\subset\fS/\fJ_{n+1}$ for every
$n\ge1$ (since $\fJ_{n+1}$ is an open right ideal in~$\fS$).

 The inclusion $\fI\tim\fS\subset\fJ$ follows immediately from
the inclusion $\ff(\fI)\subset\fJ$ and the fact that $\fJ$ is
a closed (in fact, open) right ideal in~$\fS$.
 To prove the converse inclusion, let $s=s_1\in\fJ$ be an element.
 Then we have $\fJ=\ff(\fI)\fS+\fJ_2\subset\fS$, hence
$s_1=\sum_{i=1}^{m_1}\ff(r_i)t_i+s_2$ for some $m_1\ge1$,
\ $r_1$,~\dots, $r_{m_1}\in\fI$, \ $t_1$,~\dots, $t_{m_1}\in\fS$,
and $s_2\in\fJ_2$.
 Similarly, $\fJ_2=\ff(\fI_2)\fS+\fJ_3\subset\fS$, hence
$s_2=\sum_{i=m_1+1}^{m_1+m_2}\ff(r_i)t_i+s_3$ for some $m_2\ge1$,
\ $r_{m_1+1}$,~\dots, $r_{m_1+m_2}\in\fI_2$, \
$t_{m_1+1}$,~\dots, $t_{m_1+m_2}\in\fS$, and $s_2\in\fJ_3$.

 Proceeding in this way, we construct a sequence of integers $m_n\ge1$,
\,$n\ge1$, a sequence of elements $r_i\in\fI$, \,$i\ge1$, such that
$r_i\in\fI_n$ for all $m_1+\dotsb+m_{n-1}+1\le i\le m_1+\dotsb+m_n$,
a sequence of elements $t_i\in\fS$, and a sequence of elements
$s_n\in\fJ_n$, such that
$s_n=\sum_{i=m_1+\dotsb+m_{n-1}+1}^{m_1+\dotsb+m_n}\ff(r_i)t_i+s_{n+1}$
for every $n\ge1$.
 Now we have $s_1=\sum_{i=1}^\infty\ff(r_i)t_i$ in the sense of
the limit of finite partial sums in the topology of $\fS$, since
the open right ideals $\fJ_n$ form a base of neighborhoods of zero
in~$\fS$.
 As the sequence of elements $r_i\in\fR$ converges to zero in
the topology of $\fR$ (since the open right ideals $\fI_n$ form a base
of neighborhoods of zero in~$\fR$), it follows that the similar equation
$s=\sum_{i=1}^\infty r_it_i$ holds in the sense of the contramodule
infinite summation operation in the left $\fR$\+contramodule~$\fS$.
 Thus $s\in\fI\tim\fS$, as desired.
\end{proof}

\begin{cor} \label{left-flat-and-contramodule-flat-cor}
 Let\/ $\fR$ and\/ $\fS$ be complete, separated two-sided linear
topological rings, and let\/ $\ff\:\fR\rarrow\fS$ be a strongly right
taut continuous ring map.
 Assume that the topological ring\/ $\fR$ has a countable base of
neighborhoods of zero.
 Then the map\/~$\ff$ is left proflat if and only if the left\/
$\fR$\+contramodule\/ $\fS$ is flat.
\end{cor}

\begin{proof}
 Let $\fI$ range over all the open two-sided ideals in $\fR$ such that
the closure $\fJ$ of the right ideal $\ff(\fI)\fS\subset\fS$ is an open
two-sided ideal in~$\fS$.
 By Lemma~\ref{continuous-strongly-taut-ring-map}(3), the open two-sided
ideals $\fI$ form a base of neighborhoods of zero in $\fR$, and the open
two-sided ideals $\fJ$ form a base of neighborhoods of zero in~$\fS$.
 By Lemma~\ref{left-proflat-strongly-taut-ring-map}(1), the map~$\ff$
is left proflat if and only if the ring $\fS/\fJ$ is a flat left module
over the ring $\fR/\fI$ for every~$\fI$.
 Following the discussion in the beginning of
Section~\ref{separated-and-flat-secn}, the left $\fR$\+contramodule
$\fS$ is flat if and only if the left $\fR/\fI$\+module
$\fS/(\fI\tim\fS)$ is flat for every $\fI$ as above.
 It remains to point out that $\fJ=\fI\tim\fS$ by
Lemma~\ref{continuous-and-taut-contramodule-tim-lemma}.
\end{proof}

\Section{Change-of-Scalar Functors~I}  \label{change-of-scalars-I-secn}

 Let $f\:R\rarrow S$ be a homomorphism of (abstract, nontopological)
rings.
 Then every $S$\+module has an underlying $R$\+module structure, so
there are exact, faithful functors of restriction of scalars
$f_*\:S\Modl\rarrow R\Modl$ and $f_*\:\Modr S\rarrow\Modr R$.

 The functors~$f_*$ have adjoint functors on both sides.
 In particular, for left modules, the functor of \emph{extension of
scalars} $f^*\:R\Modl\rarrow S\Modl$ left adjoint to~$f_*$ is given
by the formula $f^*(M)=S\ot_RM$ for any left $R$\+module~$M$.
 The functor of \emph{coextension of scalars} $f^!\:R\Modl\rarrow
S\Modl$ right adjoint to~$f_*$ is given by the formula
$f^!(M)=\Hom_R(S,M)$ for all $M\in R\Modl$.

 Let $f\:R\rarrow S$ be a continuous homomorphism of right linear
topological rings.
 Then, as explained in Section~\ref{proflat-proepi-secn}
(based on Lemma~\ref{continuous-ring-maps-lemma}(3)), the functor
$f_*\:\Modr S\rarrow\Modr R$ applied to discrete modules provides
a functor of restriction of scalars $f_\diamond\:\Discr S\rarrow
\Discr R$, taking every discrete right $S$\+module $\cN$ to its
underlying discrete right $R$\+module~$\cN$.

 The functor~$f_\diamond$ is always exact and faithful, and preserves
infinite direct sums.
 So, as any colimit-preserving functor between Grothendieck categories,
the functor~$f_\diamond$ has a right adjoint functor of
\emph{coextension of scalars} $f^\diamond\:\Discr R\rarrow\Discr S$
\,\cite[Section~2.9]{Pproperf}.
 The functor~$f^\diamond$ is left exact, preserves infinite products,
and takes injective objects to injective objects.
 For any discrete right $R$\+module $M$, the discrete right
$S$\+module $f^\diamond(M)$ can be computed as the maximal discrete
$S$\+submodule of the right $S$\+module $f^!(M)=\Hom_{R^\rop}(S,M)$.

 Let $\ff\:\fR\rarrow\fS$ be a continuous homomorphism of complete,
separated right linear topological rings.
 Then, as explained in Section~\ref{proflat-proepi-secn}, there is
a functor of restriction of scalars $\ff_\sharp\:\fS\Contra\rarrow
\fR\Contra$ taking every left $\fS$\+contramodule $\fP$ to its
underlying left $\fR$\+contramodule~$\fP$.
 The functor~$\ff_\sharp$ is always exact and faithful, and preserves
infinite products.

 The functor~$\ff_\sharp$ has a left adjoint functor of
\emph{contraextension of scalars} $\ff^\sharp\:\fR\Contra
\allowbreak\rarrow\fS\Contra$.
 The functor~$\ff^\sharp$ is right exact, preserves infinite direct
sums, and takes projective objects to projective object.
 One can construct this functor using the rules that the free
$\fR$\+contramodule spanned by any set $X$ goes to the free
$\fS$\+contramodule spanned by $X$, that is $\ff^\sharp(\fR[[X]])=
\fS[[X]]$, and the cokernels are preserved by~$\ff^\sharp$.

 For any discrete right $\fS$\+module $\cN$ and any left
$\fR$\+contramodule $\fP$, there is a natural isomorphism of
abelian groups~\cite[Section~2.9]{Pproperf}
\begin{equation} \label{extension-of-scalars-contratensor-formula}
 \ff_\diamond(\cN)\ocn_\fR\fP\simeq\cN\ocn_\fS\ff^\sharp(\fP).
\end{equation}
 As the functor~$\ff_\diamond$ is exact, it follows immediately that
the functor~$\ff^\sharp$ takes flat left $\fR$\+contramodules to flat
left $\fS$\+contramodules.

\begin{lem} \label{separated-restriction-of-scalars}
 Let\/ $\ff\:\fR\rarrow\fS$ be a continuous homomorphism of complete,
separated right linear topological rings.
 Then the functor of restriction of scalars~$f_\sharp$ takes separated
left\/ $\fS$\+contramodules to separated left\/ $\fR$\+contramodules.
\end{lem}

\begin{proof}
 Let $\fP$ be a separated left $\fS$\+contramodule.
 The condition that $\fP$ is separated means that, for every nonzero
element $p\in\fP$, there exists an open right ideal $\fJ\subset\fS$
such that $p\notin\fJ\tim\fP$.
 Since $\ff$~is continuous, there exists an open right ideal
$\fI\subset\fR$ such that $\ff(\fI)\subset\fJ$.
 Now we have $p\notin\fI\tim\ff_\sharp(\fP)$.
\end{proof}

\begin{lem} \label{flat-contramodules-contraextension-of-scalars}
 Let\/ $\ff\:\fR\rarrow\fS$ be a continuous homomorphism of complete,
separated right linear topological rings with \emph{countable} bases
of neighborhoods of zero.
 Then the functor of contraextension of scalars\/ $\ff^\sharp\:
\fR\Contra\rarrow\fS\Contra$ takes short exact sequences of flat
left\/ $\fR$\+contramodules to short exact sequences of flat left\/
$\fS$\+contramodules.
 Moreover, for any short exact sequence of left\/ $\fR$\+contramodules\/
$0\rarrow\fQ\rarrow\fP\rarrow\fF\rarrow0$ with a flat left\/
$\fR$\+contramodule\/ $\fF$, the short sequence of left\/
$\fS$\+contramodules\/ $0\rarrow\ff^\sharp(\fQ)\rarrow\ff^\sharp(\fP)
\rarrow\ff^\sharp(\fF)\rarrow0$ is exact.
\end{lem}

\begin{proof}
 Firstly, let $0\rarrow\fH\rarrow\fG\rarrow\fF\rarrow0$ be a short exact
sequence of flat left $\fR$\+contramodules.
 By Lemma~\ref{flat-contramodules-well-behaved-lemma}(b)
and formula~\eqref{extension-of-scalars-contratensor-formula}, for any
discrete right $\fS$\+module $\cN$, we have a short exact sequence of
abelian groups $0\rarrow\cN\ocn_\fS\ff^\sharp(\fH)\rarrow
\cN\ocn_\fS\ff^\sharp(\fG)\rarrow\cN\ocn_\fS\ff^\sharp(\fF)\rarrow0$.
 Taking $\cN=\fS/\fJ$, where $\fJ$ is an open right ideal in $\fS$,
we see that the map $\ff^\sharp(\fH)/(\fJ\tim\ff^\sharp(\fH))\rarrow
\ff^\sharp(\fG)/(\fJ\tim\ff^\sharp(\fG))$ is injective.
 Therefore, the kernel of the map $\ff^\sharp(\fH)\rarrow
\ff^\sharp(\fG)$ is contained in $\fJ\tim\ff^\sharp(\fH)$.
 The left $\fS$\+contramodule $\ff^\sharp(\fH)$ is flat,
hence by Lemma~\ref{complete-separated-flat-contramodules-lemma}(c)
it is separated; so we have
$\bigcap_{\fJ\subset\fS}(\fJ\tim\ff^\sharp(\fH))=0$.
 Thus the map $\ff^\sharp(\fH)\rarrow\ff^\sharp(\fG)$ is injective,
as desired.

 Now let $\boL_*\ff^\sharp$ denote the left derived functor of
the right exact functor $\ff^\sharp$, constructed as usual in terms
of projective resolutions.
 From the first assertion of the lemma (which we have proved already)
together with Lemma~\ref{flat-contramodules-well-behaved-lemma}(a)
we see that $\boL_i\ff^\sharp(\fF)=0$ for any flat left
$\fR$\+contramodule $\fF$ and all $i>0$.
 In view of the long exact sequence of derived functor
$\boL_*\ff^\sharp$ associated with the short exact sequence of
left $\fR$\+contramodules $0\rarrow\fQ\rarrow\fP\rarrow\fF\rarrow0$,
the second assertion of the lemma follows.
\end{proof}

 We denote by $\Ext^{\fR,*}({-},{-})$ the Ext groups computed in
the abelian category of left $\fR$\+contramodules $\fR\Contra$.
 A left $\fR$\+contramodule $\fC$ is said to be \emph{cotorsion} if
$\Ext^{\fR,1}(\fF,\fC)=0$ for all flat left $\fR$\+contramodules~$\fF$.
 Clearly, the full subcategory of cotorsion $\fR$\+contramodules is
closed under extensions and infinite products in $\fR\Contra$.

\begin{cor} \label{cotorsion-restriction-of-scalars-cor}
 Let\/ $\ff\:\fR\rarrow\fS$ be a continuous homomorphism of complete,
separated right linear topological rings with \emph{countable} bases
of neighborhoods of zero.
 Then the functor of restriction of scalars\/~$\ff_\sharp$ takes
cotorsion left\/ $\fS$\+contramodules to cotorsion left\/
$\fR$\+contramodules.
\end{cor}

\begin{proof}
 By~\cite[Lemma~1.7(e)]{Pal}, it follows from the second assertion of
Lemma~\ref{flat-contramodules-contraextension-of-scalars} that, for
any flat left $\fR$\+contramodule $\fF$ and any left $\fS$\+contramodule
$\fQ$, there is a natural isomorphism of abelian groups
$$
 \Ext^{\fR,1}(\fF,\ff_\sharp(\fQ))\simeq
 \Ext^{\fS,1}(\ff^\sharp(\fF),\fQ).
$$
 For a cotorsion left $\fS$\+contramodule $\fQ$, the right-hand side
of this isomorphism vanishes, hence so does the left-hand side.
\end{proof}

 Let $(R_n)_{n\ge1}$ and $(S_n)_{n\ge1}$ be two projective systems of
surjective morphisms of rings, indexed by the integers $n\ge1$.
 Then $\fR=\varprojlim_{n\ge1}R_n$ and $\fS=\varprojlim_{n\ge1}S_n$
are two complete, separated two-sided linear topological rings with
countable bases of neighborhoods of zero.

 Let $(f_n\:R_n\to S_n)$ be a morphism of projective systems of rings. 
 Then $\ff=\varprojlim_{n\ge1}f_n\:\fR\rarrow\fS$ is a continuous
homomorphism of topological rings.

 We will say that the morphism of projective systems (of surjective
homomorphisms) of rings $(f_n)_{n\ge1}$ is \emph{strongly right taut}
if the induced $R_n$\+$S_{n+1}$\+bimodule map $R_n\ot_{R_{n+1}}S_{n+1}
\rarrow S_n$ is an isomorphism for every $n\ge1$.
 By Lemma~\ref{continuous-strongly-taut-ring-map-countable-base},
a map of complete, separated two-sided linear topological rings with
countable bases of neighborhoods of zero $\ff\:\fR\rarrow\fS$ is
continuous \emph{and} strongly right taut if and only if it arises
from a strongly right taut morphism of projective systems
$(f_n\:R_n\to S_n)$.

 The definitions of \emph{right} and \emph{left $(R_n)$\+systems} of
modules were given in Section~\ref{separated-and-flat-secn}.

\begin{lem} \label{strongly-right-taut-direct-image-right-left-systems}
 Let $(f_n\:R_n\to S_n)_{n\ge1}$ be a strongly right taut morphism of
projective systems of rings.
 Then \par
\textup{(a)} for any right $(S_n)$\+system $(N_n)_{n\ge1}$,
the underlying right $R_n$\+modules of the right $S_n$\+modules $N_n$
form a right $(R_n)$\+system $(N_n)_{n\ge1}$; \par
\textup{(b)} for any left $(S_n)$\+system $(Q_n)_{n\ge1}$,
the underlying right $R_n$\+modules of the left $S_n$\+modules $Q_n$
form a left $(R_n)$\+system $(Q_n)_{n\ge1}$.
\end{lem}

\begin{proof}
 Part~(a): we have
$$
 N_n\simeq\Hom_{S_{n+1}^\rop}(S_n,N_{n+1})\simeq
 \Hom_{S_{n+1}^\rop}(R_n\ot_{R_{n+1}}S_{n+1},\>N_{n+1})
 \simeq\Hom_{R_{n+1}^\rop}(R_n,N_{n+1}).
$$

 Part~(b): we have
$$
 Q_n\simeq S_n\ot_{S_{n+1}}Q_{n+1}\simeq
 (R_n\ot_{R_{n+1}}S_{n+1})\ot_{S_{n+1}}Q_{n+1}\simeq
 R_n\ot_{R_{n+1}}Q_{n+1}.
$$
\end{proof}

 The result/construction of
Lemma~\ref{strongly-right-taut-direct-image-right-left-systems}(a)
provides a functor of restriction of scalars for right systems,
which we will denote by $(f_n)_*\:\Sysr(S_n)\rarrow\Sysr(R_n)$.
 The result/construction of
Lemma~\ref{strongly-right-taut-direct-image-right-left-systems}(b)
provides a functor of restriction of scalars for left systems, which
we will similarly denote by $(f_n)_*\:(S_n)\Sys\rarrow(R_n)\Sys$.

\begin{lem} \label{right-left-systems-discrete-contra-direct-image}
 Let $(f_n\:R_n\to S_n)$ be a strongly right taut morphism of projective
systems of rings, and let\/ $\ff\:\fR\rarrow\fS$ be the related strongly
right taut continuous map of topological rings (the projective limit).
 Then \par
\textup{(a)} the functors of restriction of scalars\/
$\ff_\diamond\:\Discr\fS\rarrow\Discr\fR$ and
$(f_n)_*\:\allowbreak\Sysr(S_n)\rarrow\Sysr(R_n)$ form a commutative
square diagram with the equivalences of categories\/ $\Discr\fR\simeq
\Sysr(R_n)$ and\/ $\Discr\fS\simeq\Sysr(S_n)$ from
Lemma~\ref{left-right-systems-discr-contra-category-equivs}(a),
$$
 \xymatrix{
  \Discr\fS \ar@{=}[r] \ar[d]_{\ff_\diamond}
  & \Sysr(S_n) \ar[d]^{(f_n)_*} \\
  \Discr\fR \ar@{=}[r] & \Sysr(R_n)
 }
$$ \par
\textup{(b)} the functors of restriction of scalars\/
$\ff_\sharp\:\fS\Separ\rarrow\fR\Separ$ from
Lemma~\ref{separated-restriction-of-scalars} and
$(f_n)_*\:(S_n)\Sys\rarrow(R_n)\Sys$ form a commutative square
diagram with the equivalences of categories\/ $\fR\Separ\simeq
(R_n)\Sys$ and\/ $\fS\Separ\simeq(S_n)\Sys$ from
Lemma~\ref{left-right-systems-discr-contra-category-equivs}(b),
$$
 \xymatrix{
  \fS\Separ \ar@{=}[r] \ar[d]_{\ff_\sharp}
  & (S_n)\Sys \ar[d]^{(f_n)_*} \\
  \fR\Separ \ar@{=}[r] & (R_n)\Sys
 }
$$
\end{lem}

\begin{proof}
 In both parts~(a) and~(b), it may take some work to check that
the horizontal functors pointing rightwards (constructed as
the collections of reduction functors) commute with the vertical
functors of restriction of scalars.
 However, it is immediately clear that the horizontal functors
pointing leftwards (constructed as the inductive or projective
limit functors) commute with the vertical functors.
\end{proof}

\begin{cor} \label{tim-Omega-and-Lambda-over-R-and-S-compatible}
 Let\/ $\ff\:\fR\rarrow\fS$ be a strongly right taut continuous map
of complete, separated two-sided linear topological rings with
countable bases of neighborhoods of zero.
 Let\/ $\fI\subset\fS$ be an open two-sided ideal such that
the closure\/ $\fJ$ of the right ideal\/ $\ff(\fI)\fS\subset\fS$
is an open two-sided ideal in\/~$\fS$.
 Then, for any left\/ $\fS$\+contramodule\/ $\fQ$, one has
$\fI\tim\fQ=\fJ\tim\fQ\subset\fQ$.
 Consequently, the nonseparatedness kernels of
the\/ $\fS$\+contramodule\/ $\fQ$ and of its underlying\/
$\fR$\+contramodule\/ $\fQ$ agree,
$\Omega_\fR(\fQ)=\Omega_\fS(\fQ)\subset\fQ$, and the\/
$\fR$\+contramodule map\/ $\Lambda_\fR(\fQ)\rarrow\Lambda_\fS(\fQ)$
induced by\/~$\ff$ is an isomorphism.
 A left\/ $\fS$\+contramodule is separated if and only if its
underlying left\/ $\fR$\+contramodule is separated.
\end{cor}

\begin{proof}
 Let us first consider the case of a separated
$\fS$\+contramodule~$\fQ$.
 Following the proof of
Lemma~\ref{continuous-strongly-taut-ring-map-countable-base}, we can
pick two descending chains $(\fI_n)_{n\ge1}$ and $(\fJ_n)_{n\ge1}$
of open two-sided ideals in $\fR$ and $\fS$ satisfying the conditions
of Lemma~\ref{continuous-strongly-taut-ring-map-countable-base}(2)
in such a way that $\fI_1=\fI$ and $\fJ_1=\fJ$.
 Denote by $R_n=\fR/\fI_n$ and $S_n=\fS/\fJ_n$ the related discrete
quotient rings.
 Then we have a strongly right taut morphism of projective systems
of rings $(f_n\:R_n\to S_n)$ such that $\ff=\varprojlim_{n\ge1}f_n$.

 Let $(Q_n)_{n\ge1}$ be the left $(S_n)$\+system corresponding to
the separated left $\fS$\+con\-tra\-mod\-ule $\fQ$ under
the equivalence of categories from
Lemma~\ref{left-right-systems-discr-contra-category-equivs}(b).
 By Lemma~\ref{strongly-right-taut-direct-image-right-left-systems}(b),
the underlying $R_n$\+modules of the $S_n$\+modules $Q_n$ form
a left $(R_n)$\+system $(Q_n)_{n\ge1}$.
 By Lemma~\ref{right-left-systems-discrete-contra-direct-image}(b),
the $(R_n)$\+system $(Q_n$) corresponds to the underlying separated
$\fR$\+contramodule of $\fQ$ under the same equivalence of
categories from
Lemma~\ref{left-right-systems-discr-contra-category-equivs}(b).
 Looking into the construction of the later equivalence of
categories, we have isomorphisms (of abelian groups)
$\fQ/(\fJ_n\tim\fQ)\simeq Q_n\simeq\fQ/(\fI_n\tim\fQ)$ forming
a commutative diagram with the projections
$\fQ\rarrow\fQ/(\fJ_n\tim\fQ)$ and $\fQ\rarrow\fQ/(\fI_n\tim\fQ)$.
 Thus $\fJ_n\tim\fQ=\fI_n\tim\fQ$ for all $n\ge1$.
 In particular, $\fJ\tim\fQ=\fI\tim\fQ$, as desired.

 In the general case of a not necessarily separated
$\fS$\+contramodule $\fQ$, is suffices to represent $\fQ$ as
the quotient $\fS$\+contramodule of a separated
$\fS$\+contramodule~$\fP$ (such as, e.~g., a free left
$\fS$\+contramodule $\fP=\fS[[X]]$ for a suitable set~$X$).
 Then $\fJ\tim\fQ$ is the image of $\fJ\tim\fP$ under the map
$\fP\rarrow\fQ$, while $\fI\tim\fQ$ is the image of $\fI\tim\fQ$
under the same surjective map $\fP\rarrow\fQ$.
 (This observation holds quite generally for any surjective morphism
of contramodules $\fP\rarrow\fQ$ over a complete, separated right
linear topological ring.)
 The proof the first assertion of the corollary is finished, and 
the remaining assertions follow.
\end{proof}

\begin{cor} \label{proflat-contramodule-direct-image-cor}
 Let\/ $\ff\:\fR\rarrow\fS$ be a strongly right taut continuous map
of complete, separated two-sided linear topological rings with
countable bases of neighborhoods of zero.
 Then the map\/~$\ff$ is left proflat if and only if the functor
of restriction of scalars\/~$\ff_\sharp$ takes flat left\/
$\fS$\+contramodules to flat left\/ $\fR$\+contramodules.
\end{cor}

\begin{proof}
 The ``if'' assertion follows from
Corollary~\ref{left-flat-and-contramodule-flat-cor}.
 To prove the ``only if'', let $(\fI_n)_{n\ge1}$ and $(\fJ_n)_{n\ge1}$
be two descending chains of open two-sided ideals in $\fR$ and $\fS$
provided by
Lemma~\ref{continuous-strongly-taut-ring-map-countable-base}(2).
 So we have a strongly right taut morphism of projective systems
of rings $(f_n\:R_n\to S_n)$ such that $\ff=\varprojlim_{n\ge1}f_n$
(where $R_n=\fR/\fI_n$ and $S_n=\fS/\fJ_n$).
 By Lemma~\ref{left-proflat-strongly-taut-ring-map}(1), the ring $S_n$
is a flat left module over the ring $R_n$ for every $n\ge1$.
 Now it is clear that the functor of restriction of scalars
$(f_n)_*\:(S_n)\Sys\rarrow(R_n)\Sys$ takes flat left
$(S_n)$\+systems to flat left $(R_n)$\+systems.
 It remains to refer to
Lemma~\ref{right-left-systems-discrete-contra-direct-image}(b)
and the assertion concerning flatness in
Lemma~\ref{left-right-systems-discr-contra-category-equivs}(b).
\end{proof}

\Section{The Full-and-Faithfulness Theorem}

 We start with a simple lemma and an important proposition going
back to the paper~\cite{PR}.
 These should be understood in the context of \emph{hereditary
complete cotorsion pairs} in abelian categories;
Lemma~\ref{cotorsion-contramodules-well-behaved}
and Proposition~\ref{flat-cotorsion-pair-complete-prop} below can
be restated by saying that the pair of classes (flat left
$\fR$\+contramodules, cotorsion left $\fR$\+contramodules) form
such a cotorsion pair in $\fR\Contra$.

 The notion of a complete cotorsion pair was introduced by
Salce~\cite{Sal} in the context of infinitely generated abelian groups.
 The main result is the Eklof--Trlifaj theorem, which was proved
originally for modules over rings~\cite[Theorems~2 and~10]{ET}.
 We refer to the book~\cite[Chapters~5--6]{GT} for an in-depth
discussion of cotorsion pairs in module categories, and to
the paper~\cite[Definition~2.3]{Hov} for the generalization to
abelian categories.
 An exposition in the even more general context of exact categories
(in the sense of Quillen) can be found
in~\cite[Section~1]{Pal} or~\cite[Appendix~B]{Pcosh}.

\begin{lem} \label{cotorsion-contramodules-well-behaved}
 Let\/ $\fR$ be a complete, separated right linear topological ring with
a \emph{countable} base of neighborhoods of zero.
 The the full subcategory of cotorsion $\fR$\+contramodules is closed
under cokernels of monomorphisms in\/ $\fR\Contra$.
 One has\/ $\Ext^{\fR,i}(\fF,\fC)=0$ for all flat left\/
$\fR$\+contramodules $\fF$, all cotorsion left\/
$\fR$\+con\-tra\-mod\-ules\/ $\fC$, and all $i\ge1$.
\end{lem}

\begin{proof}
 The assertions follow from the facts that there are enough flat left
$\fR$\+con\-tra\-mod\-ules in $\fR\Contra$, and the full subcategory of
flat left $\fR$\+contramodules is closed under kernels of
epimorphisms~\cite[Section~7]{PR}.
\end{proof}

\begin{prop} \label{flat-cotorsion-pair-complete-prop}
 Let\/ $\fR$ be a complete, separated right linear topological ring with
a \emph{countable} base of neighborhoods of zero, and let\/ $\fP$ be
a left\/ $\fR$\+contramodule.
 Then \par
\textup{(a)} there exists a short exact sequence\/ $0\rarrow\fC\rarrow
\fF\rarrow\fP\rarrow0$ in\/ $\fR\Contra$ with a flat left\/
$\fR$\+contramodule\/ $\fF$ and a cotorsion left\/
$\fR$\+contramodule\/~$\fC$; \par
\textup{(b)} there exists a short exact sequence\/ $0\rarrow\fP\rarrow
\fC\rarrow\fF\rarrow0$ in\/ $\fR\Contra$ with a cotorsion left\/
$\fR$\+contramodule\/ $\fC$ and a flat left\/
$\fR$\+contramodule\/~$\fF$.
\end{prop}

\begin{proof}
 This is~\cite[Corollary~7.8]{PR}.
\end{proof}

\begin{lem} \label{flat-cotorsion-presentations-Hom-computation}
 Let\/ $\fR$ be a complete, separated right linear topological ring with
a \emph{countable} base of neighborhoods of zero.
 In this context: \par
\textup{(a)} For any left\/ $\fR$\+contramodule\/ $\fP$, there exists
a morphism of flat left\/ $\fR$\+con\-tra\-mod\-ules\/ $t\:\fG\rarrow\fF$
such that the image and the kernel of~$t$ are cotorsion left\/
$\fR$\+contramodules, while the cokernel of~$t$ is isomorphic
to\/~$\fP$. \par
\textup{(b)} Let\/ $\fP'$ and\/ $\fP''$ be two left\/
$\fR$\+contramodules represented as the cokernels of morphisms of
flat left\/ $\fR$\+contramodules $t'\:\fG'\rarrow\fF'$ and
$t''\:\fG''\rarrow\fF''$ with cotorsion images and kernels, as in
part~\textup{(a)}.
 Then there is a natural surjective map onto the abelian group\/
$\Hom^\fR(\fP',\fP'')$ from the kernel of the abelian group map
\begin{equation} \label{Hom-computed-by-flat-and-cotorsion}
 \Hom^\fR(\fF',\fF'')\oplus\Hom^\fR(\fG',\fG'')\lrarrow
 \Hom^\fR(\fG',\fF'').
\end{equation}
 A pair of morphisms $f\:\fF'\rarrow\fF''$ and $g\:\fG'\rarrow\fG''$
belonging to the kernel of~\eqref{Hom-computed-by-flat-and-cotorsion}
induces a zero map $p\:\fP'\rarrow\fP''$ if and only if
the map~$f$ arises from an\/ $\fR$\+con\-tra\-mod\-ule morphism
$h\:\fF'\rarrow\fG''$.
\end{lem}

\begin{proof}
 Part~(a): apply Proposition~\ref{flat-cotorsion-pair-complete-prop}(a)
twice.
 Firstly, there is a short exact sequence $0\rarrow\fC\rarrow\fF
\rarrow\fP\rarrow0$ with a flat $\fR$\+contramodule $\fF$ and
a cotorsion $\fR$\+contramodule~$\fC$.
 Secondly, there is a short exact sequence $0\rarrow\fD\rarrow\fG
\rarrow\fC\rarrow0$ with a flat $\fR$\+contramodule $\fG$ and
a cotorsion $\fR$\+contramodule~$\fD$.

 Part~(b): the kernel of~\eqref{Hom-computed-by-flat-and-cotorsion} is
the group of all commutative square diagrams of $\fR$\+contramodule
morphisms
\begin{equation} \label{morphism-of-flat-presentations-diagram}
\begin{gathered}
 \xymatrix{
  \fG' \ar[r]^{t'} \ar[d]_-g & \fF' \ar[d]^-f \\
  \fG'' \ar[r]^{t''} & \fF''
 }
\end{gathered}
\end{equation}
where the horizontal morphisms are fixed as $t'$ and~$t''$, while
the vertical morphisms $f$ and~$g$ vary.
 Every commutative
diagram~\eqref{morphism-of-flat-presentations-diagram}
induces an $\fR$\+contramodule morphism $p\:\fP'\rarrow\fP''$,
by passing to the cokernels of the horizontal maps.
 The map~$p$ vanishes whenever $f=t''h$ for some $\fR$\+contramodule
morphism $h\:\fF'\rarrow\fG''$.
 Hence we have a natural map into $\Hom^\fR(\fP',\fP'')$ from
the quotient group of the kernel
of~\eqref{Hom-computed-by-flat-and-cotorsion} by the subgroup formed
by the pairs $(f,g)$ for which $f=t''h$.
 It remains to show that this map is surjective and injective.
 This is where the flatness and cotorsion condition on the contramodules
involved will be used.

 Let $\fC''$ and $\fD''$ denote the image and the kernel
of the morphism~$t''$, respectively.
 So we have short exact sequences of $\fR$\+contramodules
$0\rarrow\fC''\rarrow\fF''\rarrow\fP''\rarrow0$ and $0\rarrow\fD''
\rarrow\fG''\rarrow\fC''\rarrow0$.
 Given a morphism of $\fR$\+contramodules $p\:\fP'\rarrow\fP''$,
consider the composition $\fF'\rarrow\fP'\rarrow\fP''$.
 Since $\Ext^{\fR,1}(\fF',\fC'')=0$, the morphism $\fF'\rarrow\fP''$
can be lifted to an $\fR$\+contramodule morphism $f\:\fF'\rarrow\fF''$.
 Now the composition $\fG'\rarrow\fF'\rarrow\fF''$ is annihilated
by the composition with the morphism $\fF''\rarrow\fP''$, hence
the morphism $\fG'\rarrow\fF''$ factorizes uniquely as
$\fG'\rarrow\fC''\rarrow\fF''$.
 Since $\Ext^{\fR,1}(\fG',\fD'')=0$, the morphism $\fG'\rarrow\fC''$
can be lifted to an $\fR$\+contramodule morphism $g\:\fG'\rarrow\fG''$.
 We arrive at a commutative
diagram~\eqref{morphism-of-flat-presentations-diagram} from which
the original map~$p$ can be obtained by passing to the cokernels.

 Now suppose given a commutative
diagram~\eqref{morphism-of-flat-presentations-diagram} such that
the related morphism of the cokernels vanishes, $p=0$.
 Then the morphism $f\:\fF'\rarrow\fF''$ factorizes uniquely as
$\fF'\rarrow\fC''\rarrow\fF''$.
 Since $\Ext^{\fR,1}(\fF',\fD'')=0$, the morphism $\fF'\rarrow\fC''$
can be lifted to an $\fR$\+contramodule morphism $h\:\fF'\rarrow\fG''$.
\end{proof}

\begin{thm} \label{full-and-faithfulness-equivalent-conditions-thm}
 Let\/ $\ff\:\fR\rarrow\fS$ be a continuous homomorphism of complete,
separated right linear topological rings with \emph{countable} bases
of neighborhoods of zero.
 Assume that the functor of restriction of scalars\/~$\ff_\sharp$
takes flat left\/ $\fS$\+contramodules to flat left\/
$\fR$\+contramodules.
 Then the following three conditions are equivalent:
\begin{enumerate}
\item the functor\/ $\ff_\sharp\:\fS\Flat\rarrow\fR\Flat$ is fully
faithful;
\item the functor\/ $\ff_\sharp\:\fS\Separ\rarrow\fR\Separ$ is fully
faithful;
\item the functor\/ $\ff_\sharp\:\fS\Contra\rarrow\fR\Contra$ is fully
faithful.
\end{enumerate}
\end{thm}

\begin{proof}
 (3)~$\Longrightarrow$~(2) This implication holds for any continuous
homomorphism of complete, separated right linear topological rings.
 One only needs to refer to
Lemma~\ref{separated-restriction-of-scalars} for the assertion that
the functor in~(2) is well-defined.

 (2)~$\Longrightarrow$~(1) Follows immediately from
Lemma~\ref{complete-separated-flat-contramodules-lemma}(c).

 (1)~$\Longrightarrow$~(3)
 Let $\fP'$ and $\fP''$ be two left $\fS$\+contramodules, and let
$\fG'\rarrow\fF'\rarrow\fP'\rarrow0$ and $\fG''\rarrow\fF''\rarrow
\fP''\rarrow0$ be their flat presentations with cotorsion images
and kernels, as in
Lemma~\ref{flat-cotorsion-presentations-Hom-computation}(a)
(applied to the topological ring~$\fS$).
 By assumption, the functor of restriction of scalars~$\ff_\sharp$ takes
flat left $\fS$\+contramodules to flat left $\fR$\+contramodules.
 By Corollary~\ref{cotorsion-restriction-of-scalars-cor}, the same
functor also takes cotorsion left $\fS$\+contramodules to cotorsion
left $\fR$\+contramodules.
 Hence $\ff_\sharp(\fG')\rarrow\ff_\sharp(\fF')\rarrow\ff_\sharp(\fP')
\rarrow0$ and $\ff_\sharp(\fG'')\rarrow\ff_\sharp(\fF'')\rarrow
\ff_\sharp(\fP'')\rarrow0$ are flat presentations of
the $\fR$\+contramodules $\ff_\sharp(\fP')$ and $\ff_\sharp(\fP'')$
that also satisfy the conditions of
Lemma~\ref{flat-cotorsion-presentations-Hom-computation}(a), i.~e.,
have cotorsion images and kernels.
 Computing $\Hom^\fS(\fP'',\fP'')$ and
$\Hom^\fR(\ff_\sharp(\fP'),\ff_\sharp(\fP''))$ as described in
Lemma~\ref{flat-cotorsion-presentations-Hom-computation}(b),
we can conclude that the map $\Hom^\fS(\fP'',\fP'')\rarrow
\Hom^\fR(\ff_\sharp(\fP'),\ff_\sharp(\fP''))$ is an isomorphism.
\end{proof}

\begin{lem} \label{R-S-annihilator-easy-lemma}
 Let $f\:R\rarrow S$ be a continuous map of right linear
topological rings, and let\/ $\cM$ be a discrete right $S$\+module.
 Let $I\subset R$ be an open right ideal such that the closure $J$ of
the right ideal $f(I)S\subset S$ is an open right ideal in~$S$.
 Then an element $y\in\cN$ is annihilated by $I$ if and only if it is
annihilated by~$J$.
\end{lem}

\begin{proof}
 Suppose $yf(I)=0$ in~$\cN$.
 Then also $yf(I)S=0$.
 The annihilator of~$y$ in $\cN$ is an open right ideal containing
$f(I)S$, so it also contains~$J$.
\end{proof}

\begin{prop} \label{full-and-faithfulness-equivalent-conditions-prop-I}
 Let $f\:R\rarrow S$ be a strongly right taut continuous map of
two-sided linear topological rings.
 Then the following two conditions are equivalent:
\begin{enumerate}
\item the functor of restriction of scalars $f_\diamond\:\Discr S
\rarrow\Discr R$ is fully faithful;
\item $f$~is a topological ring proepimorphism.
\end{enumerate}
\end{prop}

\begin{proof}
 (1)~$\Longrightarrow$~(2) This implication does not need the strong
right tautness assumption, and holds for any continuous homomorphism
of two-sided linear topological rings.
 Let $I\subset R$ and $J\subset S$ be two open two-sided ideals
such that $f(I)\subset J$.
 In order to prove that $R/I\rarrow S/J$ is a ring epimorphism, let
us show that the functor of restriction of scalars
$\Modr S/J\rarrow\Modr R/I$ is fully faithful.

 Indeed, any right $R/I$\+module can be viewed as a discrete right
$R$\+module, and any right $S/J$\+module can be viewed as a discrete
right $S$\+module.
 The functors $\Modr R/I\rarrow\Discr R$ and $\Modr S/J\rarrow
\Discr S$ are obviously fully faithful (just because the ring
homomorphisms $R\rarrow R/I$ and $S\rarrow S/J$ are surjective).
 If the functor $\Discr S\rarrow\Discr R$ is fully faithful, then it
follows that the functor $\Modr S/J\rarrow\Modr R/I$ is fully faithful,
too.

 (2)~$\Longrightarrow$~(1)
 Let $\cM$ and $\cN$ be two discrete right $S$\+modules, and let
$h\:\cM\rarrow\cN$ be an $R$\+module morphism.
 Let $x\in\cM$ and $s\in S$ be two elements.
 We need to prove that $h(xs)=h(x)s$ in~$\cN$.
 By the definition of strong right tautness (or by
Lemma~\ref{continuous-strongly-taut-ring-map}(2)), there exist
open two-sided ideals $I\subset R$ and $J\subset S$ satisfying
the following conditions:
\begin{itemize}
\item $J$~is the closure of the right ideal $f(I)S\subset S$;
\item $xI=0$ in~$\cM$.
\end{itemize}

 Let $\cM_I\subset\cM$ and $\cN_I\subset\cN$ be the $R$\+submodule
of all elements annihilated by $I$ in $\cM$ and the $R$\+submodule of
all elements annihilated by $I$ in~$\cN$.
 By Lemma~\ref{R-S-annihilator-easy-lemma}, $\cM_I$ and $\cN_I$ are
also annihilated by~$J$; so $\cM_I$ and $\cN_I$ are $S/J$\+modules.
 Since $h\:\cM\rarrow\cN$ is an $R$\+module map, it takes $\cM_I$ into
$\cN_I$; hence we obtain an $R/I$\+module map $h_I\:\cM_I\rarrow\cN_I$.
 By assumption, the induced ring map $R/I\rarrow S/J$ is a ring
epimorphism; so any $R/I$\+module map of $S/J$\+modules is
an $S/J$\+module map.
 Thus $h_I$~is an $S/J$\+module map.

 By construction, we have $x\in\cM_I$; hence $xs\in\cM_I$.
 Finally, $h(xs)=h_I(xs)=h_I(x)s=h(x)s$ in $\cN_I\subset\cN$,
as desired.
\end{proof}

\begin{prop} \label{full-and-faithfulness-equivalent-conditions-prop-II}
 Let\/ $\ff\:\fR\rarrow\fS$ be a strongly right taut continuous map of
complete, separated two-sided linear topological rings with 
\emph{countable} bases of neighborhoods of zero.
 Then the following two conditions are equivalent:
\begin{enumerate}
\item the functor of restriction of scalars\/ $\ff_\sharp\:\fS\Separ
\rarrow\fR\Separ$ from Lemma~\ref{separated-restriction-of-scalars}
is fully faithful;
\item $\ff$~is a topological ring proepimorphism.
\end{enumerate}
\end{prop}

\begin{proof}
 (1)~$\Longrightarrow$~(2) This implication does not need the strong
right tautness and countable topology base assumptions, and holds for
any continuous homomorphism of complete, separated two-sided linear
topological rings.
 Let $\fI\subset\fR$ and $\fJ\subset\fS$ be two open two-sided ideals
such that $\ff(\fI)\subset\fJ$.
 In order to prove that $\fR/\fI\rarrow\fS/\fJ$ is a ring epimorphism,
let us show that the functor of restriction of scalars
$\fS/\fJ\Modl\rarrow\fR/\fI\Modl$ is fully faithful.

 Indeed, any left $\fR/\fI$\+module can be viewed as a separated left
$\fR$\+contramodule, and any left $\fS/\fJ$\+module can be viewed as
a separated left $\fS$\+contramodule.
 The functors $\fR/\fI\Modl\rarrow\fR\Separ$ and $\fS/\fJ\Modl\rarrow
\fS\Separ$ are obviously fully faithful.
 If the functor $\fS\Separ\rarrow\fR\Separ$ is fully faithful, then
it follows that the functor $\fS/\fJ\Modl\rarrow\fR/\fI\Modl$ is fully
faithful, too.

 (2)~$\Longrightarrow$~(1) As explained in
Section~\ref{change-of-scalars-I-secn}, the ring map~$\ff$ arises
from a strongly right taut morphism of projective systems of rings
$(f_n\:R_n\to S_n)$.
 By Lemma~\ref{right-left-systems-discrete-contra-direct-image}(b),
it suffices to check that the functor of restriction of scalars
$(f_n)_*\:(S_n)\Sys\rarrow(R_n)\Sys$ is fully faithful.
 This follows immediately from the assumption that the functors
$S_n\Modl\rarrow R_n\Modl$ are fully faithful for all $n\ge1$.
 (It is helpful to notice that $(S_n)\Sys$ is a full subcategory in
the category of projective systems of left modules over the projective
system of rings $(S_n)$, and similarly for $(R_n)\Sys$.)
\end{proof}

\begin{cor} \label{full-and-faithfulness-equivalent-conditions-cor}
 Let\/ $\ff\:\fR\rarrow\fS$ be a left proflat continuous map of
complete, separated two-sided linear topological rings with
\emph{countable} bases of neighborhoods of zero.
 Then the following five conditions are equivalent:
\begin{enumerate}
\item the functor of restriction of scalars\/ $\ff_\diamond\:\Discr\fS
\rarrow\Discr\fR$ is fully faithful;
\item the functor of restriction of scalars\/ $\ff_\sharp\:\fS\Flat
\rarrow\fR\Flat$ is fully faithful;
\item the functor of restriction of scalars\/ $\ff_\sharp\:\fS\Separ
\rarrow\fR\Separ$ is fully faithful;
\item the functor of restriction of scalars\/ $\ff_\sharp\:\fS\Contra
\rarrow\fR\Contra$ is fully faithful;
\item $\ff$~is a topological ring proepimorphism.
\end{enumerate}
\end{cor}

\begin{proof}
 (1)~$\Longleftrightarrow$~(5) This is
Proposition~\ref{full-and-faithfulness-equivalent-conditions-prop-I}.

 (3)~$\Longleftrightarrow$~(5) This is
Proposition~\ref{full-and-faithfulness-equivalent-conditions-prop-II}.

 (2)~$\Longleftrightarrow$~(3)~$\Longleftrightarrow$~(4)
 The functor of restriction of scalars~$\ff_\sharp$ takes flat left
$\fS$\+contramodules to flat left $\fR$\+contramodules by
Corollary~\ref{proflat-contramodule-direct-image-cor}.
 So Theorem~\ref{full-and-faithfulness-equivalent-conditions-thm}
is applicable.
\end{proof}

\Section{Contramodule-to-Module Full-and-Faithfulness Condition}

 We start with a recollection of results from~\cite[Section~7.3]{PS1}
and~\cite[Section~6]{Pcoun}.

\begin{lem} \label{contramodule-to-module-full-and-faithful-lemma}
 Let\/ $\fR$ be a complete, separated right linear topological ring,
$R$ be a ring, and\/ $\rho\:R\rarrow\fR$ be a ring homomorphism with
dense image.
 Consider the following five conditions:
\begin{enumerate}
\item the composition of forgetful functors\/ $\fR\Contra\rarrow
\fR\Modl\rarrow R\Modl$ is fully faithful;
\item the composition of forgetful functors\/ $\fR\Separ\rarrow
\fR\Modl\rarrow R\Modl$ is fully faithful;
\item for any discrete right\/ $\fR$\+module\/ $\cN$ and any left\/
$\fR$\+contramodule\/ $\fP$, the natural surjective map from
the contratensor product to the tensor product
$$
 \cN\ot_R\fP\lrarrow\cN\ocn_\fR\fP
$$
is an isomorphism;
\item for any finitely generated discrete right\/ $\fR$\+module\/ $\cN$
and any free left\/ $\fR$\+con\-tra\-mod\-ule\/ $\fP=\fR[[X]]$ (where
$X$ is a set), the natural surjective map from the contratensor product
to the tensor product
$$
 \cN\ot_R\fP\lrarrow\cN\ocn_\fR\fP
$$
is an isomorphism;
\item for any cyclic discrete right\/ $\fR$\+module\/ $\cN=\fR/\fI$
(where\/ $\fI\subset\fR$ is an open right ideal) and any free
left\/ $\fR$\+contramodule\/ $\fP=\fR[[X]]$, the natural surjective map
from the contratensor product to the tensor product
$$
 \cN\ot_R\fP\lrarrow\cN\ocn_\fR\fP
$$
is an isomorphism.
\end{enumerate}
 Then the implications and equivalence
\textup{(1)~$\Longrightarrow$~(2)}
\textup{$\Longrightarrow$~(3)}
\textup{$\Longleftrightarrow$~(4)}
\textup{$\Longrightarrow$~(5)} hold.
 When the topological ring\/ $\fR$ has a countable base of neighborhoods
of zero, all the five conditions are equivalent,
\textup{(1)~$\Longleftrightarrow$~(2)}
\textup{$\Longleftrightarrow$~(3)} \textup{$\Longleftrightarrow$~(4)}
\textup{$\Longleftrightarrow$~(5)}.
\end{lem}

\begin{proof}
 (1)~$\Longrightarrow$~(2) Obvious.
 
 (1)~$\Longrightarrow$~(3) This is, essentially, \cite[Lemma~7.11]{PS1}.
 
 (2)~$\Longrightarrow$~(4) The point is that the free left
$\fR$\+contramodules are separated.
 Furthermore, for any discrete right $\fR$\+module $\cN$ and any abelian
group $W$, the left $\fR$\+contramodule $\Hom_\boZ(\cN,W)$ is
separated~\cite[Lemma~5.5]{Pcoun}.
 So the argument from~\cite[Lemma~7.11]{PS1} is applicable.

 (3)~$\Longrightarrow$~(4)~$\Longrightarrow$~(5) Obvious.
 
 (4)~$\Longrightarrow$~(3) Holds because all discrete right
$\fR$\+modules are direct unions of finitely generated ones, all
left $\fR$\+contramodules are cokernels of morphisms of free ones,
and both the functors~$\ot_R$ and~$\ocn_\fR$ preserve colimits in
the first argument and cokernels in the second one.
 (Notice that $\cN\ot_R{-}$ actually preserves all colimits in $R\Modl$
and $\cN\ocn_\fR{-}$ preserves all colimits in $\fR\Contra$; but
the forgetful functor $\fR\Contra\rarrow\fR\Modl$ preserves finite
colimits only.)

 (5)~$\Longrightarrow$~(1) This is, essentially,
\cite[Theorem~6.2\,(i)\,$\Rightarrow$\,(iii)]{Pcoun}.
\end{proof}

\begin{cor} \label{S-contra-to-S-mod-fully-faithful-cor}
 Let\/ $\fR$ and\/ $\fS$ be complete, separated right linear
topological rings with countable bases of neighborhoods of zero, and
let\/ $\ff\:\fR\rarrow\fS$ be a continuous ring map.
 Assume that the functor $\ff_\sharp\:\fS\Separ\rarrow\fR\Separ$ from
Lemma~\ref{separated-restriction-of-scalars} and the forgetful
functor\/ $\fR\Contra\rarrow\fR\Modl$ are fully faithful.
 Then the forgetful functor\/ $\fS\Contra\rarrow\fS\Modl$ is fully
faithful, too.
\end{cor}

\begin{proof}
 According to
Lemma~\ref{contramodule-to-module-full-and-faithful-lemma}%
\,(2)\,$\Rightarrow$\,(1), it suffices to show that the forgetful
functor $\fS\Separ\rarrow\fS\Modl$ is fully faithful.
 Indeed, let $\fP$ and $\fQ$ be separated left $\fS$\+contramodules,
and let $g\:\fP\rarrow\fQ$ be an $\fS$\+module map.
 Then $g$~is also an $\fR$\+module map.
 By assumption, it follows that $g$~is an $\fR$\+contramodule morphism.
 By another assumption, we can conclude that $g$~is
an $\fS$\+contramodule morphism.
\end{proof}

\begin{lem} \label{discrete-quotients-tensor-lemma}
 Let\/ $\fR$ and\/ $\fS$ be complete, separated right linear topological
rings with countable bases of neighborhoods of zero, and let\/
$\ff\:\fR\rarrow\fS$ be a taut continuous ring map.
 Assume that the forgetful functor\/ $\fR\Contra\rarrow\fR\Modl$ is
fully faithful.
 Let\/ $\fI\subset\fR$ be an open right ideal.
 Then the right ideal\/ $\fJ=\ff(\fI)\fS\subset\fS$ is open.
 Denote the related discrete cyclic modules by\/ $\cM=\fR/\fI$ and\/
$\cN=\fS/\fJ$.
 Then the induced right\/ $\fS$\+module map\/ $\cM\ot_\fR\fS\rarrow\cN$
is an isomorphism.
\end{lem}

\begin{proof}
 We have $\cM\ot_\fR\fS\simeq\fS/(\fI\cdot\fS)$ and
$\cM\ocn_\fR\fS\simeq\fS/(\fI\tim\fS)$ (see the notation and discussion
in Section~\ref{preliminaries-secn}).
 By Lemma~\ref{contramodule-to-module-full-and-faithful-lemma}%
\,(1)\,$\Rightarrow$\,(3), the map $\cM\ot_\fR\fS\rarrow\cM\ocn_\fR\fS$
is an isomorphism; so $\fI\cdot\fS=\fI\tim\fS$.
 Denote by $\fJ'$ the closure of the right ideal
$\ff(\fI)\fS\subset\fS$.
 By Lemma~\ref{taut-ring-maps-lemma}(1), \,$\fJ'$~is an open right
ideal in~$\fS$.
 By Lemma~\ref{continuous-and-taut-contramodule-tim-lemma},
we have $\fJ'=\fI\tim\fS$.
 Thus $\fJ'=\fI\tim\fS=\fI\cdot\fS=\ff(\fI)\fS=\fJ\subset\fS$ and
$\cN=\fS/\fJ=\fS/(\fI\cdot\fS)\simeq\cM\ot_\fR\fS$.
\end{proof}

\begin{cor} \label{reduction-S-module-cor}
 Let\/ $\fR$ and\/ $\fS$ be complete, separated right linear topological
rings with countable bases of neighborhoods of zero, and let\/
$\ff\:\fR\rarrow\fS$ be a taut continuous ring map.
 Assume that the forgetful functor\/ $\fR\Contra\rarrow\fR\Modl$ is
fully faithful.
 Let\/ $\fI\subset\fR$ be an open two-sided ideal such that\/
$\ff(\fI)\fS\subset\fS$ is an (open, by
Lemma~\ref{discrete-quotients-tensor-lemma}) two-sided ideal in\/~$\fS$.
 Consider the discrete quotient rings $R=\fR/\fI$ and $S=\fS/\fJ$.
 Let\/ $\fQ$ be a left\/ $\fR$\+contramodule whose underlying left\/
$\fR$\+module structure arises from a left\/ $\fS$\+module structure.
 Then the natural left $R$\+module structure on\/ $\fQ/(\fI\tim\fQ)$
arises from a left $S$\+module structure such that the projection\/
$\fQ\rarrow\fQ/(\fI\tim\fQ)$ is an\/ $\fS$\+module map.
\end{cor}

\begin{proof}
 By Lemma~\ref{contramodule-to-module-full-and-faithful-lemma}%
\,(1)\,$\Rightarrow$\,(3), we have $\fQ/(\fI\tim\fQ)\simeq
R\ocn_\fR\fQ=R\ot_\fR\fQ$.
 By Lemma~\ref{discrete-quotients-tensor-lemma}, the natural
$R$\+$\fS$\+bimodule map $R\ot_\fR\fS\rarrow S$ is an isomorphism,
hence $R\ot_\fR\fQ\simeq(R\ot_\fR\fS)\ot_\fS\fQ\simeq S\ot_\fS\fQ$
is an $S$\+module.
\end{proof}

\Section{Pseudopullback Diagram and Characterization of
the~Essential~Image}

 We start with the easier case of discrete modules before passing
to contramodules.

\begin{prop} \label{topol-ring-map-discreteness-characterized}
 Let\/ $\fR$ and\/ $\fS$ be complete, separated right linear
topological rings with \emph{countable} bases of neighborhoods of zero, 
and let\/ $\ff\:\fR\rarrow\fS$ be a taut continuous ring map.
 Assume that the forgetful functor\/ $\fR\Contra\rarrow\fR\Modl$ is
fully faithful.
 Then a right\/ $\fS$\+module is discrete if and only if its underlying
right\/ $\fR$\+module is discrete.
\end{prop}

\begin{proof}
 The ``only if'' assertion holds by
Lemma~\ref{continuous-ring-maps-lemma}\,(1)\,$\Rightarrow$\,(3).
 To prove the ``if'', let $N$ be a right $\fS$\+module, $y\in N$ be
an element, and $\fI\subset\fR$ be an open right ideal annihilating~$y$.
 By Lemma~\ref{discrete-quotients-tensor-lemma}, the right ideal
$\fJ=\ff(\fI)\fS$ is open in~$\fS$.
 Clearly, $y\fI=0$ implies $y\fJ=0$.
 Thus the annihilator of~$y$ is open in~$\fS$.
\end{proof}

 Given three categories $\sA$, $\sB$, and $\sC$, and two functors
$F\:\sA\rarrow\sC$ and $G\:\sB\rarrow\sC$, the \emph{pseudopullback}
$\sP$ of the pair of functors $F$ and $G$ is defined as follows.
 The objects of $\sP$ are triples $(A,B,\phi)$, where $A\in\sA$
and $B\in\sB$ are objects, and $\phi\:F(A)\simeq G(B)$ is
an isomorphism in~$\sC$.
 We skip the obvious definition of morphisms in~$\sP$.

 In this paper, we speak about \emph{commutative square diagrams of
functors} presuming that the compositions along the two sides of
the square are isomorphic (rather than just ``equal'') functors, and
such a natural isomorphism is given or fixed.
 Any commutative square of functors
\begin{equation} \label{commutative-square-of-functors}
\begin{gathered}
 \xymatrix{
  \sD \ar[r] \ar[d] & \sA \ar[d]^F \\
  \sB \ar[r]_G & \sC
 }
\end{gathered}
\end{equation}
induces a comparison functor $\Theta\:\sD\rarrow\sP$ from $\sD$
to the pseudopullback $\sP$ of the pair of functors $F$ and~$G$.
 We will say that~\eqref{commutative-square-of-functors} is
a \emph{pseudopullback diagram} if the functor $\Theta$ is a category
equivalence.

 We refer to the papers and book~\cite{JS}, \cite[Section~2.H and
Exercise~2.n]{AR}, \cite[Section~3.1]{CR}, \cite[Section~2]{RR},
and/or~\cite[Section~5]{Pacc} for a further discussion of
the construction of pseudopullbacks of categories.

\begin{prop} \label{separated-pseudopullback-prop}
 Let\/ $\fR$ and\/ $\fS$ be complete, separated two-sided linear
topological rings with countable bases of neighborhoods of zero, and
let\/ $\ff\:\fR\rarrow\fS$ be a strongly right taut continuous ring map.
 Assume that the forgetful functor\/ $\fR\Contra\rarrow\fR\Modl$ is
fully faithful.
 Then the commutative square diagram of forgetful functors
\begin{equation} \label{separated-pseudopullback-diagram}
\begin{gathered}
 \xymatrix{
  \fS\Separ \ar@{>->}[r] \ar[d]_-{\ff_\sharp}
  & \fS\Modl \ar[d]^-{\ff_*} \\
  \fR\Separ \ar@{>->}[r] & \fR\Modl
 }
\end{gathered}
\end{equation}
is a pseudopullback diagram.
 Here the two horizontal arrows with tails show fully faithful functors.
 So the forgetful functor\/ $\fS\Contra\rarrow\fS\Modl$ is fully
faithful, too.
\end{prop}

\begin{proof}
 Let $\fQ$ be a set/abelian group endowed with a separated left
$\fR$\+contramodule and a left $\fS$\+module structures whose
underlying $\fR$\+module structures agree.
 We need to construct a separated left $\fS$\+contramodule structure
on $\fQ$ with the given underlying $\fR$\+contramodule and
$\fS$\+module structures.

 Let $(\fI_n)_{n\ge1}$ and $(\fJ_n)_{n\ge1}$ be two descending chains
of open two-sided ideals in $\fR$ and $\fS$ provided by
Lemma~\ref{continuous-strongly-taut-ring-map-countable-base}(2).
 So we have a strongly right taut morphism of projective systems
of rings $(f_n\:R_n\to S_n)$ such that $\ff=\varprojlim_{n\ge1}f_n$
(where $R_n=\fR/\fI_n$ and $S_n=\fS/\fJ_n$).
 Furthermore, according to the proof of
Lemma~\ref{continuous-strongly-taut-ring-map-countable-base},
\,$\fJ_n$ is the closure of the right ideal $\ff(\fI_n)\fS\subset\fS$,
and by Lemma~\ref{discrete-quotients-tensor-lemma} it actually follows
that $\fJ_n=\ff(\fI_n)\fS$.

 Put $Q_n=\fQ/(\fI_n\tim\fQ)$; so $(Q_n)_{n\ge1}$ is the left
$(R_n)$\+system of modules corrresponding to the separated
left $\fR$\+contramodule $\fQ$ under the equivalence of categories
from Lemma~\ref{left-right-systems-discr-contra-category-equivs}(b).
 By Corollary~\ref{reduction-S-module-cor}, the left $R_n$\+module
structures on $Q_n$ can be extended to $S_n$\+module structures in
a compatible way; so the $(R_n)$\+system structure on $(Q_n)_{n\ge1}$
gets extended to an $(S_n)$\+system structure.
 By Lemma~\ref{right-left-systems-discrete-contra-direct-image}(b),
it follows that the separated left $\fR$\+contramodule structure on
$\fQ$ can be extended to a separated left $\fS$\+contramodule
structure.
 It is clear from the construction that the latter structure is
compatible with the given $\fS$\+module structure.

 Furthermore, let $\fQ'$ and $\fQ''$ be two separated left
$\fS$\+contramodules, and let $g\:\fQ'\rarrow\fQ''$ be
an $\fS$\+module map.
 Then $g$~is also an $\fR$\+module map, and by assumption it follows
that $g$~is an $\fR$\+contramodule map.
 We need to show that $g$~is an $\fS$\+contramodule map.

 Put $Q'_n=\fQ'/(\fI_n\tim\fQ')$ and $Q''_n=\fQ''/(\fI_n\tim\fQ'')$.
 Then we have a map of left $(R_n)$\+systems
$(g_n)\:(Q'_n)\rarrow(Q''_n)$ coresponding to the map of separated
left $\fR$\+con\-tra\-mod\-ules $g\:\fQ'\rarrow\fQ''$.
 Since $g$~is an $\fS$\+module map, it follows that $g_n$ is
an $S_n$\+module map for every $n\ge1$.
 Thus $g=\varprojlim_{n\ge1}g_n$ is an $\fS$\+contramodule map.
 
 The final assertion of the proposition follows by virtue of
Lemma~\ref{contramodule-to-module-full-and-faithful-lemma}%
\,(2)\,$\Rightarrow$\,(1).
\end{proof}

\begin{cor} \label{separated-essential-image-characterized-cor}
 Let\/ $\fR$ and\/ $\fS$ be complete, separated two-sided linear
topological rings with countable bases of neighborhoods of zero, and
let\/ $\ff\:\fR\rarrow\fS$ be a strongly right taut continuous ring map.
 Assume further that\/ $\ff$~is a topological ring proepimorphism and
the forgetful functor\/ $\fR\Contra\rarrow\fR\Modl$ is fully faithful.
 Then a separated left\/ $\fR$\+contramodule\/ $\fQ$ belongs to
the essential image of the fully faithful functor\/
$\ff_\sharp\:\fS\Separ\rarrow\fR\Separ$ from
Proposition~\ref{full-and-faithfulness-equivalent-conditions-prop-II}
if and only if the underlying left\/ $\fR$\+module structure of\/ $\fQ$
arises from some (necessarily unique) left\/ $\fS$\+module structure.
 The pseudopullback diagram of forgetful
functors~\eqref{separated-pseudopullback-diagram} takes the form
\begin{equation} \label{proepimorphism-separated-pseudopullback-diagram}
\begin{gathered}
 \xymatrix{
  \fS\Separ \ar@{>->}[r] \ar@{>->}[d]_-{\ff_\sharp}
  & \fS\Modl \ar[d]^-{\ff_*} \\
  \fR\Separ \ar@{>->}[r] & \fR\Modl
 }
\end{gathered}
\end{equation}
where the three arrows with tails show fully faithful functors.
\end{cor}

\begin{proof}
 The upper horizontal functor
in~\eqref{proepimorphism-separated-pseudopullback-diagram} is fully
faithful by Corollary~\ref{S-contra-to-S-mod-fully-faithful-cor},
as well as by Proposition~\ref{separated-pseudopullback-prop}.
 The leftmost vertical functor~$\ff_\sharp$ is fully faithful by
Proposition~\ref{full-and-faithfulness-equivalent-conditions-prop-II}.
 The ``only if'' assertion of the corollary is obvious.
 The ``if'' assertion follows immediately from
Proposition~\ref{separated-pseudopullback-prop}.
 Finally, the uniqueness of the $\fS$\+module structure on $\fQ$
follows from the uniqueness of the $\fS$\+contramodule structure
(which holds since the functor~$\ff_\sharp$ is fully faithful).
\end{proof}

\begin{lem} \label{separated-quotient-S-module-structure}
 Let\/ $\fR$ and\/ $\fS$ be complete, separated two-sided linear
topological rings with countable bases of neighborhoods of zero, and
let\/ $\ff\:\fR\rarrow\fS$ be a strongly right taut continuous ring map.
 Assume that the forgetful functor\/ $\fR\Contra\rarrow\fR\Modl$ is
fully faithful.
 Let\/ $\fQ$ be a left\/ $\fR$\+contramodule whose underlying left\/
$\fR$\+module structures arises from a left\/ $\fS$\+module structure.
 Then the maximal separated quotient\/ $\fR$\+contramodule of\/~$\fQ$,
$$
 \Lambda_\fR(\fQ)\,=\,
 \varprojlim\nolimits_{\fI\subset\fR}\fQ/(\fI\tim\fQ)
 \,\simeq\,\fQ\Big/\bigcap\nolimits_{\fI\subset\fR}(\fI\tim\fQ),
$$
admits an\/ $\fS$\+module structure such that\/
$\fQ\rarrow\Lambda_\fR(\fQ)$ is an\/ $\fS$\+module map.
 In other words, the nonseparatedness kernel\/
$\bigcap_{\fI\subset\fR}(\fI\tim\fQ)\subset\fQ$ is an\/ $\fS$\+submodule
of\/~$\fQ$.
 Here\/ $\fI$ ranges over the open right (or equivalently, two-sided)
ideals in\/~$\fR$.
\end{lem}

\begin{proof}
 The isomorphism $\varprojlim_{\fI\subset\fR}\fQ/(\fI\tim\fQ)
\simeq\fQ\big/\bigcap_{\fI\subset\fR}(\fI\tim\fQ)$ holds by
Lemma~\ref{complete-separated-flat-contramodules-lemma}(a).

 Let $(\fI_n)_{n\ge1}$ and $(\fJ_n)_{n\ge1}$ be two descending chains
of open two-sided ideals in $\fR$ and $\fS$ provided by
Lemma~\ref{continuous-strongly-taut-ring-map-countable-base}(2).
 According to the proof of
Lemma~\ref{continuous-strongly-taut-ring-map-countable-base},
\,$\fJ_n$ is the closure of the right ideal $\ff(\fI_n)\fS\subset\fS$,
and by Lemma~\ref{discrete-quotients-tensor-lemma} it actually follows
that $\fJ_n=\ff(\fI_n)\fS$.

 Then, by Corollary~\ref{reduction-S-module-cor}, \,$\fQ\rarrow
\fQ/(\fI_n\tim\fQ)$ is an $\fS$\+module map for every $n\ge1$.
 Hence $\fI_n\tim\fQ$ is an $\fS$\+submodule in~$\fQ$.
 Thus $\bigcap_{\fI\subset\fR}(\fI\tim\fQ)=
\bigcap_{n\ge1}(\fI_n\tim\fQ)$ is also an $\fS$\+submodule of~$\fQ$.
\end{proof}

 The following theorem is the most nontrivial result of this paper.

\begin{thm} \label{pseudopullback-thm}
 Let\/ $\fR$ and\/ $\fS$ be complete, separated two-sided linear
topological rings with countable bases of neighborhoods of zero, and
let\/ $\ff\:\fR\rarrow\fS$ be a left proflat continuous ring map.
 Assume that the forgetful functor\/ $\fR\Contra\rarrow\fR\Modl$ is
fully faithful.
 Then the commutative square diagram of forgetful functors
\begin{equation} \label{pseudopullback-diagram}
\begin{gathered}
 \xymatrix{
  \fS\Contra \ar@{>->}[r] \ar[d]_-{\ff_\sharp}
  & \fS\Modl \ar[d]^-{\ff_*} \\
  \fR\Contra \ar@{>->}[r] & \fR\Modl
 }
\end{gathered}
\end{equation}
is a pseudopullback diagram.
 Here the two horizontal arrows with tails show fully faithful functors.
\end{thm}

\begin{proof}
 Let $\fQ$ be a set/abelian group endowed with a left
$\fR$\+contramodule and a left $\fS$\+module structures whose
underlying $\fR$\+module structures agree.
 We need to construct a left $\fS$\+contramodule structure on $\fQ$
with the given underlying $\fR$\+con\-tra\-mod\-ule and $\fS$\+module
structures.

 Recall the shorthand notation $\Omega_\fR(\fP)=
\ker(\fP\to\Lambda_\fR(\fP))=\bigcap_{\fI\subset\fR}(\fI\tim\fP)$
for a left $\fR$\+contramodule~$\fP$.
 The $\fR$\+contramodule $\Lambda_\fR(\fP)$ is always separated
(see the beginning of Section~\ref{separated-and-flat-secn}).
 When the underlying $\fR$\+module structure on $\fP$ arises from
an $\fS$\+module structure, there are compatible $\fS$\+modules
structure on $\Lambda_\fR(\fP)$ and $\Omega_\fR(\fP)$ by
Lemma~\ref{separated-quotient-S-module-structure}.
 By Proposition~\ref{separated-pseudopullback-prop}, the structures of
a separated left $\fR$\+contramodule and a left $\fS$\+module on
$\Lambda_\fR(\fP)$ arise from a naturally defined left
$\fS$\+contramodule structure.

 Let $\aleph$~be a cardinal greater than the cardinality (of
the underlying set) of~$\fQ$.
 We consider $\aleph$ as an ordinal, and proceed by transfinite
induction on ordinals $0\le\alpha<\aleph$, constructing
an $\aleph$\+indexed projective system of left $\fR$\+contramodules
$\fQ_\alpha$ and surjective $\fR$\+contramodule maps $\fQ_\alpha\rarrow
\fQ_\beta$ for all $0\le\beta\le\alpha<\aleph$.
 For every~$\alpha$, the underlying left $\fR$\+module structure on
$\fQ_\alpha$ will arise from a left $\fS$\+module structure,
and the transition maps $\fQ_\alpha\rarrow\fQ_\beta$ will be
$\fS$\+linear.

 Moreover, applying the functor $\Omega_\fR$ to the surjective map
$\fQ_\alpha\rarrow\fQ_\beta$ will produce an \emph{injective}
$\fR$\+contramodule map $\Omega_\fR(\fQ_\alpha)\rarrow
\Omega_\fR(\fQ_\beta)$.
 The latter map will be a proper embedding (\emph{not} an isomorphism)
whenever $\beta<\alpha$ and $\Omega_\fR(\fQ_\beta)\ne0$.
 It will follow that $\Omega_\fR(\fQ_\alpha)=0$ for large enough
ordinals $\alpha<\aleph$.

 For $\alpha=0$, put $\fQ_0=\fQ$.
 For a successor ordinal $\alpha=\beta+1$, pick a flat left
$\fS$\+contramodule $\fG_\beta$ together with a surjective
$\fS$\+contramodule morphism $\fG_\beta\rarrow\Lambda_\fR(\fP_\beta)$.
 In particular, $\fG_\beta\rarrow\Lambda_\fR(\fP_\beta)$ is a morphism
in $\fR\Contra$ and $\fS\Modl$.
 Let $\fQ_{\beta+1}$ be the pullback of the pair of epimorphisms
$\fQ_\beta\rarrow\Lambda_\fR(\fQ_\beta)$ and $\fG_\beta\rarrow
\Lambda_\fR(\fQ_\beta)$ in the abelian category $\fR\Contra$,
as well as in the abelian category $\fS\Modl$,
\begin{equation} \label{separated-flat-contra-pullback-diagram}
\begin{gathered}
 \xymatrix{
  \fQ_\beta \ar@{->>}[r] & \Lambda_\fR(\fQ_\beta) \\
  \fQ_{\beta+1} \ar@{->>}[u] \ar@{->>}[r] & \fG_\beta \ar@{->>}[u]
 }
\end{gathered}
\end{equation}
 Obviously, the pullbacks in $\fR\Contra$ and in $\fS\Modl$ agree
(as the forgetful functors are exact).
 So the underlying left $\fR$\+module structure of the left
$\fR$\+contramodule $\fQ_{\beta+1}$ arises from a left $\fS$\+module
structure.

 By Corollary~\ref{proflat-contramodule-direct-image-cor},
the underlying left $\fR$\+contramodule of the flat left
$\fS$\+con\-tra\-mod\-ule $\fG_\beta$ is also flat.
 In particular, by
Lemma~\ref{complete-separated-flat-contramodules-lemma}(c),
the $\fR$\+contramodule $\fG_\beta$ is separated.
 Applying the functor $\Lambda_\fR$ to the surjective
$\fR$\+contramodule map $\fQ_{\beta+1}\rarrow\fG_\beta$, we obtain
a factorization of the latter map into a composition of surjective
$\fR$\+con\-tra\-mod\-ule maps
$$
 \xymatrix{
  \fQ_{\beta+1} \ar@{->>}[r] & \Lambda_\fR(\fQ_{\beta+1})
  \ar@{->>}[r] & \fG_\beta.
 }
$$
 It follows that $\Omega_\fR(\fQ_{\beta+1})$ is
an $\fR$\+subcontramodule of the kernel of the map
$\fQ_{\beta+1}\allowbreak\rarrow\fG_\beta$.
 Since \eqref{separated-flat-contra-pullback-diagram}~is a pullback
diagram, the morphism $\fQ_{\beta+1}\rarrow\fQ_\beta$ induces
an isomorphism $\ker(\fQ_{\beta+1}\to\fG_\beta)\simeq\ker(\fQ_\beta\to
\Lambda_\fR(\fQ_\beta))=\Omega_\fR(\fQ_\beta)$.
 Thus the $\fR$\+con\-tra\-mod\-ule epimorphism $\fQ_{\beta+1}\rarrow
\fQ_\beta$ induces an $\fR$\+contramodule monomorphism
\begin{equation} \label{monomorphism-on-Omega}
 \xymatrix{
  \Omega_\fR(\fQ_{\beta+1}) \ar@{>->}[r] & \Omega_\fR(\fQ_\beta),
 }
\end{equation}
as promised above.

 Now we have to show that \eqref{monomorphism-on-Omega}~is \emph{not}
an isomorphism whenever $\Omega_\fR(\fQ_\beta)\ne0$.
 Indeed, if \eqref{monomorphism-on-Omega}~were an isomorphism, it would
mean that $\Lambda_\fR(\fQ_{\beta+1})\simeq\fG_\beta$.
 So $\Lambda_\fR(\fQ_{\beta+1})$ would be a flat $\fR$\+contramodule.
 By Corollary~\ref{separated-quotient-contratensor-flatness-cor}(b),
it would follow that $\Omega_\fR(\fQ_{\beta+1})=0$, hence also
$\Omega_\fR(\fQ_\beta)=0$.

 For a limit ordinal~$\alpha$, we simply put $\fQ_\alpha=
\varprojlim_{\beta<\alpha}\fQ_\beta$.
 The projective limits in $\fR\Contra$ and in $\fS\Modl$ agree (as
the forgetful functors preserve all limits).
 So the underlying left $\fR$\+module structure of the left
$\fR$\+contramodule $\fQ_\alpha$ arises from a left $\fS$\+module
structure.

 Let us check that the map $\Omega_\fR(\fQ_\alpha)\rarrow
\varprojlim_{\beta<\alpha}\Omega_\fR(\fQ_\beta)$ is injective.
 The point is that the functor $\Omega_\fR$ comes together with
an injective natural transformation $\Omega_\fR\rarrow\Id$, i.~e.,
an injective $\fR$\+contramodule morphism $\Omega_\fR(\fP)\rarrow\fP$
for every left $\fR$\+contramodule~$\fP$.
 So we have a commutative diagram of $\fR$\+contramodule maps
$$
 \xymatrix{
  \varprojlim\nolimits_{\beta<\alpha}\Omega_\fR(\fQ_\beta)
  \ar@{>->}[r] & \varprojlim\nolimits_{\beta<\alpha}\fQ_\beta \\
  \Omega_\fR(\varprojlim\nolimits_{\beta<\alpha}(\fQ_\beta))
  \ar[u] \ar@{>->}[r] & \varprojlim\nolimits_{\beta<\alpha}\fQ_\beta
  \ar@{=}[u]
 }
$$
where arrows with tails show monomorphisms.
 Notice that the functor of projective limit takes injective maps
to injective maps.
 Now injectivity of the lower horizontal arrow implies the desired
injectivity of the leftmost vertical arrow.

 It follows that the map $\Omega_\fR(\fQ_\alpha)\rarrow
\Omega_\fR(\fQ_\beta)$ is injective for all $\beta\le\alpha$.
 Let us spell out the straightforward argument proving existence
of an ordinal $\gamma<\aleph$ for which $\Omega_\fR(\fQ_\gamma)=0$.
 Assume for the sake of contradiction that such an ordinal~$\gamma$
does not exist.
 Proceeding by transfinite induction, we construct an injective map
of sets $\psi\:\aleph\rarrow\fQ_0=\fQ$.
 For every ordinal $\beta<\aleph$, we have a proper inclusion of
subsets/subgroups $\im(\Omega_\fR(\fQ_{\beta+1})\to\Omega_\fR(\fQ))
\varsubsetneq\im(\Omega_\fR(\fQ_\beta)\to\Omega_\fR(\fQ))$.
 Let $\psi(\beta)\in\Omega_\fR(\fQ))$ be any element belonging to
$\im(\Omega_\fR(\fQ_\beta)\to\Omega_\fR(\fQ))$ but not to
$\im(\Omega_\fR(\fQ_{\beta+1})\to\Omega_\fR(\fQ))$.
 As the cardinality of $\aleph$ is greater than the cardinality
of $\fQ$, an injective map $\psi\:\aleph\rarrow\fQ$ cannot exist.
 The contradiction proves existence of the desired ordinal~$\gamma$.
 So the $\fR$\+contramodule $\fQ_\gamma$ is separated.

 Thus we have constructed a separated left $\fR$\+contramodule
$\fQ_\gamma$ whose underlying left $\fR$\+module structure arises
from a left $\fS$\+module structure.
 We have a surjective map $\fQ_\gamma\rarrow\fQ_0=\fQ$ that is both
an $\fR$\+contramodule map and an $\fS$\+module map.
 Put $\fK=\ker(\fQ_\gamma\to\fQ)$.
 So $\fK$ is also a separated left $\fR$\+contramodule whose underlying
left $\fR$\+module structure arises from a left $\fS$\+module structure,
and the inclusion $\fK\rarrow\fQ_\gamma$ is both an $\fR$\+contramodule
map and an $\fS$\+module map.
 By Proposition~\ref{separated-pseudopullback-prop}, both $\fK$ and
$\fQ_\gamma$ have natural left $\fS$\+contramodule structures
compatible with the given $\fR$\+contramodule and $\fS$\+module
structures, and the inclusion $\fK\rarrow\fQ_\gamma$ is
an $\fS$\+contramodule map.
 Finally, $\fQ=\coker(\fK\to\fQ_\gamma)$ acquires the left
$\fS$\+contramodule structure of the cokernel of a left
$\fS$\+contramodule morphism.

 This proves the assertion of the theorem on the level of objects.
 We still have to check it on the level of morphisms.
 Let $\fP$ and $\fQ$ be two left $\fS$\+contramodules, and let
$g\:\fP\rarrow\fQ$ be an $\fS$\+module map.
 Then $g$~is also an $\fR$\+module map, and by assumption it follows
that $g$~is an $\fR$\+contramodule map.
 We need to show that $g$~is an $\fS$\+contramodule map.
 For this purpose, it suffices to refer to the last assertion of
Proposition~\ref{separated-pseudopullback-prop}.
 This argument is based on
Lemma~\ref{contramodule-to-module-full-and-faithful-lemma}%
\,(2)\,$\Rightarrow$\,(1).
 Here is an alternative argument based on the transfinite construction
above in this proof.

 Pick a set $X$ together with a surjective morphism $\fS[[X]]\rarrow
\fP$ from the free left $\fS$\+contramodule $\fS[[X]]$ onto~$\fP$.
 Then the composition $\fS[[X]]\rarrow\fP\rarrow\fQ$ is
an $\fS$\+module map between two $\fS$\+contramodules.
 In order to show that $\fP\rarrow\fQ$ is an $\fS$\+contramodule
map, it suffices to check that the composition $\fS[[X]]\rarrow\fQ$
is an $\fS$\+contramodule map.
 Replacing $\fP$ by $\fS[[X]]$, we will assume without loss of
generality that $\fP$ is a projective $\fS$\+contramodule.

 The construction of the projective system
$(\fQ_\alpha)_{0\le\alpha<\aleph}$ from the preceding part of this
proof was applicable to any left $\fR$\+contramodule $\fQ$ whose
underlying left $\fR$\+module structure arose from a left
$\fS$\+module structure.
 Now we apply this construction to the left $\fS$\+contramodule~$\fQ$.
 Then it is important for our argument that the whole construction
is performed within the realm of left $\fS$\+contramodules.
 Specifically, if $\fQ_\beta$ is a left $\fS$\+contramodule, then
$\Lambda_\fR(\fQ_\beta)\simeq\Lambda_\fS(\fQ_\beta)$ by
Corollary~\ref{tim-Omega-and-Lambda-over-R-and-S-compatible};
hence $\fQ_\beta\rarrow\Lambda_\fR(\fQ_\beta)$ is
an $\fS$\+contramodule morphism.
 It follows that all the transition maps $\fQ_\alpha\rarrow\fQ_\beta$
are $\fS$\+contramodule maps.

 Proceeding by transfinite induction on $0\le\alpha<\aleph$, we will
construct a compatible cone of left $\fR$\+contramodule and
left $\fS$\+module maps $\fP\rarrow\fQ_\alpha$.
 For $\alpha=0$, we take our map $g\:\fP\rarrow\fQ=\fQ_0$.
 On the successor steps $\alpha=\beta+1$, we have the composition
$\fP\rarrow\fQ_\beta\rarrow\Lambda_\fR(\fQ_\beta)$.
 As both $\fP$ and $\Lambda_\fR(\fQ_\beta)$ are separated
$\fR$\+contramodules, the map $\fP\rarrow\Lambda_\fR(\fQ_\beta)$
is an $\fS$\+contramodule morphism by
Proposition~\ref{separated-pseudopullback-prop}.
 The map $\fG_\beta\rarrow\Lambda_\fR(\fQ_\beta)$ is
an $\fS$\+contramodule epimorphism; so the $\fS$\+contramodule
morphism $\fP\rarrow\Lambda_\fR(\fQ_\beta)$ (from a projective
left $\fS$\+contramodule~$\fP$) can be lifted to an $\fS$\+contramodule
morphism $\fP\rarrow\fG_\beta$.
 Now the pullback diagram~\eqref{separated-flat-contra-pullback-diagram}
allows us to lift the map $\fP\rarrow\fQ_\beta$ to
an $\fR$\+contramodule and $\fS$\+module map $\fP\rarrow\fQ_{\beta+1}$.
 On the limit steps~$\alpha$, we simply use the universal property
of the projective limit in order to obtain the desired map
$\fP\rarrow\fQ_\alpha$.

 Eventually we arrive at an $\fR$\+contramodule and $\fS$\+module map
$\fP\rarrow\fQ_\gamma$.
 As the $\fS$\+contramodule $\fQ_\gamma$ is separated, and all
projective $\fS$\+contramodules are separated, the map
$\fP\rarrow\fQ_\gamma$ is an $\fS$\+contramodule morphism by
Proposition~\ref{separated-pseudopullback-prop}
(the full-and-faithfulness assertion for separated contramodules).
 The transition map $\fQ_\gamma\rarrow\fQ$ is an $\fS$\+contramodule
morphism as explained above.
 Thus the composition $\fP\rarrow\fQ_\gamma\rarrow\fQ$ is
an $\fS$\+contramodule morphism, too, as desired.
\end{proof}

\begin{cor} \label{essential-image-characterized-cor}
 Let\/ $\fR$ and\/ $\fS$ be complete, separated two-sided linear
topological rings with countable bases of neighborhoods of zero, and
let\/ $\ff\:\fR\rarrow\fS$ be a left proflat continuous ring map.
 Assume further that\/ $\ff$~is a topological ring proepimorphism and
the forgetful functor\/ $\fR\Contra\rarrow\fR\Modl$ is fully faithful.
 Then a left\/ $\fR$\+contramodule\/ $\fQ$ belongs to the essential
image of the fully faithful functor\/ $\ff_\sharp\:\fS\Contra\rarrow
\fR\Contra$ from
Corollary~\ref{full-and-faithfulness-equivalent-conditions-cor}
if and only if the underlying left\/ $\fR$\+module structure of\/ $\fQ$
arises from some (necessarily unique) left\/ $\fS$\+module structure.
 The pseudopullback diagram of forgetful
functors~\eqref{pseudopullback-diagram} takes the form
\begin{equation} \label{proepimorphism-pseudopullback-diagram}
\begin{gathered}
 \xymatrix{
  \fS\Contra \ar@{>->}[r] \ar@{>->}[d]_-{\ff_\sharp}
  & \fS\Modl \ar[d]^-{\ff_*} \\
  \fR\Contra \ar@{>->}[r] & \fR\Modl
 }
\end{gathered}
\end{equation}
where the three arrows with tails show fully faithful functors.
\end{cor}

\begin{proof}
 The upper horizontal functor
in~\eqref{proepimorphism-pseudopullback-diagram} is fully
faithful by Corollary~\ref{S-contra-to-S-mod-fully-faithful-cor},
as well as by Theorem~\ref{pseudopullback-thm}.
 The leftmost vertical functor~$\ff_\sharp$ is fully faithful by
Corollary~\ref{full-and-faithfulness-equivalent-conditions-cor}.
 The ``only if'' assertion of the corollary is obvious.
 The ``if'' assertion follows immediately from
Theorem~\ref{pseudopullback-thm}.
 Finally, the uniqueness of the $\fS$\+module structure on $\fQ$
follows from the uniqueness of the $\fS$\+contramodule structure
(which holds since the functor~$\ff_\sharp$ is fully faithful).
\end{proof}

\begin{rem} \label{non-R-contramodule-S-module-structure-not-unique}
 Notice that the assertion of uniqueness of the $\fS$\+module structure
extending the given $\fR$\+module structure in
Corollaries~\ref{separated-essential-image-characterized-cor}
and~\ref{essential-image-characterized-cor} applies to
the underlying $\fR$\+modules of (separated or nonseparated)
$\fR$\+contramodules \emph{only}.
 Generally speaking, even under the more restrictive assumptions of
Corollary~\ref{essential-image-characterized-cor}, an $\fS$\+module
structure with the given underlying $\fR$\+module structure is
\emph{not} unique; not even for Noetherian commutative rings $\fR$
and $\fS$ with adic topologies.

 For a counterexample, take the rings $\fR=R$ and $\fS=S$ from
Example~\ref{proepimorphism-not-epimorphism-example}, together
with their natural (inclusion) morphism $\ff=f$.
 Then, following Example~\ref{proepimorphism-not-epimorphism-example},
\,$\ff\:\fR\rarrow\fS$ is a proflat proepimorphism of complete,
separated commutative linear topological rings with countable bases
of neighborhoods of zero.
 Nevertheless, $\ff$~is \emph{not} an epimorphism of the underlying 
abstract (nontopological) rings of $\fR$ and~$\fS$.
 So the multiplication map $\fS\ot_\fR\fS\rarrow\fS\ot_\fS\fS=\fS$
is not injective.
 It follows that the left and right $\fS$\+module structures on
the tensor product $\fS\ot_\fR\fS$ are \emph{not} the same, while
the underlying $\fR$\+module structures of these two $\fS$\+module
structures on $\fS\ot_\fR\fS$ obviously agree.
\end{rem}

\begin{exs} \label{nonflat-pseudopullback-not-true-counterex}
 The following counterexample~(1), based on~\cite[Remark~2.11.3]{Pform},
shows that the assertion of Theorem~\ref{pseudopullback-thm} with
the left proflatness assumption on the map~$\ff$ replaced by the strong
right tautness assumption is \emph{not} true.
 Moreover, we will see in~(2) that the assertion of
Corollary~\ref{essential-image-characterized-cor} with the strong
right tautness assumption in place of the left proflatness assumption
is \emph{not} true, either.

\smallskip
 (1)~Let $\fS$ be a commutative ring, separated and complete in the adic
topology of a finitely generated ideal $\fJ\subset\fS$.
 Following the discussion in
Section~\ref{introd-adic-Noetherian-toy-example-subsecn} of
the Introduction, based on~\cite[Proposition~1.5]{Pdc}, the forgetful
functor $\fS\Contra\rarrow\fS\Modl$ is fully faithful, and its
essential image consists of all the \emph{quotseparated\/
$\fJ$\+contramodule\/ $\fS$\+modules}.
 According to~\cite[Examples~1.8]{Pdc}, there exists a commutative
ring $\fS$ with a finitely generated ideal $\fJ\subset\fS$ such that
$\fS$ is separated and complete in the $\fJ$\+adic topology but
\emph{not all\/ $\fJ$\+contramodule\/ $\fS$\+modules are quotseparated}.
 
 Let $s_1$,~\dots, $s_m\in\fS$ be a finite set of generators of
the ideal $\fJ\subset\fS$.
 Consider the ring of polynomials $R=\boZ[x_1,\dotsc,x_m]$ in
$m$~variables with integer coefficients, and let $I\subset R$ be
the ideal spanned by the elements $x_1$,~\dots, $x_m\in R$.
 Let $\fR=\boZ[[x_1,\dotsc,x_m]]$ be the ring of formal power series
in the variables $x_1$,~\dots, $x_m$ with integer coefficients; so
$\fR$ is the $I$\+adic completetion of~$R$.
 Let $\fI=\fR I\subset\fR$ be the ideal spanned by the elements
$x_1$,~\dots, $x_m\in\fR$.
 Then the ring homomorphism $f\:R\rarrow\fS$ taking $x_j$ to~$s_j$ for
all $1\le j\le m$ is continuous with respect to the $I$\+adic topology
on $R$ and the $\fJ$\+adic topology on~$\fS$.
 So the ring homomorphism~$f$ can be extended to a continuous ring
homomorphism of the adic completions $\ff\:\fR\rarrow\fS$ in
a unique way.

 Now we have $\fJ^n=\fS\ff(\fI^n)\subset\fS$ for every $n\ge1$.
 The open ideals $\fI^n$ form a base of neighborhoods of zero in $\fR$,
while the open ideals $\fJ^n$ form a base of neiborhoods of zero
in~$\fS$; so $\ff$ is a strongly right taut continuous ring map.
 The forgetful functor $\fR\Contra\rarrow\fR\Modl$ is fully faithful
by~\cite[Theorem~B.1.1]{Pweak} or~\cite[Proposition~1.5]{Pdc}.
 In fact, the ring $\fR$ is Noetherian, all ideals in Noetherian
commutative rings are weakly proregular~\cite[Theorem~4.34]{PSY},
\cite[Section~1]{Pmgm}, and it follows that all $\fI$\+contramodule
$\fR$\+modules are quotseparated~\cite[Corollary~3.7]{Pdc} (see
also~\cite[Remark~3.8]{Pdc}).
 So the essential image of the forgetful functor $\fR\Contra\rarrow
\fR\Modl$ consists of all the $\fI$\+contramodule $\fR$\+modules.

 It follows from~\cite[Theorem~5.1 and Remark~5.5]{Pcta} that
an $\fS$\+module $\fQ$ is a $\fJ$\+contramodule if and only if its
underlying $\fR$\+module is an $\fI$\+contramodule (cf.\ the discussion
in Section~\ref{introd-adic-Noetherian-toy-example-subsecn} above).
 Now let $\fQ$ be a \emph{nonquotseparated} $\fJ$\+contramodule
$\fS$\+module.
 Then the underlying $\fR$\+module of the $\fS$\+module $\fQ$ is
an $\fI$\+contramodule, so the $\fR$\+module structure of $\fQ$
arises from an $\fR$\+contramodule structure.
 But the $\fS$\+module structure on $\fQ$ does \emph{not} arise from
an $\fS$\+contramodule structure, since the $\fQ$ is \emph{not}
a quotsepated contramodule $\fS$\+module.
 Thus the left proflatness assumption on~$\ff$ \emph{cannot} be
replaced with the strong right tautness assumption in
Theorem~\ref{pseudopullback-thm}.

\smallskip
 (2)~The topological ring $\fS$ in this example is the same as in
the previous item~(1), but the ring $\fR$ is different.

 Following the details of the construction of the counterexample
in~\cite[Example~2.6]{Pmgm} and~\cite[Examples~1.8]{Pdc}, let
$R=k[x_1,x_2,x_3,\dotsc;s]$ be the ring of polynomials in
a countable family of variables $x_i$, \,$i\ge1$, and an additional
variable~$s$ over a field~$k$.
 Let $S$ be the quotient ring of $R$ by the ideal spanned by
all the elements $x_ix_j$ and $s^ix_i$, where $i$, $j\ge1$.
 We will denote the coset represented by the element~$s$ in the
ring $S$ also by~$s$.
 So we have a natural surjective homomorphism of commutative rings
$f\:R\rarrow S$.
 Endow both $R$ and $S$ with the $s$\+adic topology, and let
$\ff\:\fR\rarrow\fS$ be the induced homomorphism of the completions.
 Put $\fI=\fR s\subset\fR$ and $\fJ=\fS s\subset\fS$.

 It is clear from the construction that $\ff$~is a strongly right taut
continuous ring map.
 By~\cite[Proposition~1.2]{Yek0}, the $s$\+adic completion functor
takes surjective maps to surjective maps; so the map~$\ff$
is surjective.
 It follows that~$\ff$ is a topological ring proepimorphism.
 According to~\cite[Examples~1.8]{Pdc}, there exists a nonquotseparated
$\fJ$\+contramodule $\fS$\+module~$\fQ$.
 On the other hand, the element~$s$ is regular (i.~e.,
a nonzero-divisor) in $\fR$, hence the ideal $\fI\subset\fR$ is
weakly proregular and all $\fI$\+contramodule $\fR$\+modules are
quotseparated by~\cite[Corollary~3.7]{Pdc}.
 Similarly to~(1), the forgetful functor $\fR\Contra\rarrow\fR\Modl$
is fully faithful, and its essential image consists of all
the $\fI$\+contramodule $\fR$\+modules~\cite[Proposition~1.5]{Pdc}.

 Once again, it is clear from~\cite[Remark~5.5]{Pcta} that
an $\fS$\+module is a $\fJ$\+contramodule if and only if its
underlying $\fR$\+module is an $\fI$\+contramodule.
 Now the underlying $\fR$\+module of the $\fS$\+module $\fQ$ is
an $\fI$\+contramodule, so $\fQ$ is a contramodule over the topological
ring~$\fR$.
 But the $\fS$\+module structure on $\fQ$ does \emph{not} arise from
an $\fS$\+contramodule structure.
 Thus the left proflatness assumption on~$\ff$ \emph{cannot} be
replaced with the strong right tautness assumption in
Corollary~\ref{essential-image-characterized-cor}, either.
\end{exs}

\Section{Change-of-Scalar Functors~II} \label{change-of-scalars-II-secn}

 Let $f\:R\rarrow S$ be a continuous homomorphism of right linear
topological rings.
 Then, following the discussion in
Section~\ref{change-of-scalars-I-secn}, the exact, faithful functor
of restriction of scalars $f_\diamond\:\Discr S\rarrow\Discr R$ has
a right adjoint functor of coextension of scalars
$f^\diamond\:\Discr R\rarrow\Discr S$.

 Let $\ff\:\fR\rarrow\fS$ be a continuous homomorphism of complete,
separated right linear topological rings.
 Then, according to Section~\ref{change-of-scalars-I-secn},
the exact, faithful functor of restriction of scalars
$\ff_\sharp\:\fS\Contra\rarrow\fR\Contra$ has a left adjoint functor
of contraextension of scalars $\ff^\sharp\:\fR\Contra\rarrow\fS\Contra$.

 We start with presenting simple counterexamples showing that
the adjoints on the other sides \emph{need not}
exist~\cite[Section~2.9]{Pproperf}.

\begin{ex} \label{nontaut-adjoints-need-not-exist}
 Let $k$~be a field, $\fR=k[x]$ be the ring of polynomials in one
variable~$x$ over~$k$ endowed with the discrete topology, and
$\fS=k[[x]]$ be the ring of formal power series in~$x$ over~$k$ endowed
with the $x$\+adic topology.
 Let $\ff\:\fR\rarrow\fS$ be the natural inclusion map.
 So $\ff$~is a continuous homomorphism of complete, separated
commutative linear topological rings.

 The category of discrete $\fS$\+modules is simply the category of
$x$\+torsion (i.~e., $x$\+power torsion) $k[x]$\+modules.
 The category of $\fS$\+contramodules is the category of
$k[x]$\+modules with $x$\+power infinite summation operation,
which is also a full subcategory in $k[x]\Modl$
(see~\cite[Theorem~B.1.1 or Lemma~B.5.1]{Pweak},
\cite[Section~1.6]{Prev}, or~\cite[Theorem~3.3]{Pcta}).

 Then the exact, fully faithful functor of restriction of scalars
$\ff_\diamond\:\Discr\fS\rarrow\Discr\fR=\Modr\fR$ does \emph{not}
have a left adjoint.
 The point is that the functor~$\ff_\diamond$ does \emph{not} preserve
countable products.
 Indeed, the $k[x]$\+modules $k[x]/(x^n)$, \,$n\ge1$, are $x$\+torsion,
but their countable product $\prod_{n=1}^\infty k[x]/(x^n)$, computed
in $\Modr k[x]$, is \emph{not} $x$\+torsion.
 The product of the $x$\+torsion $k[x]$\+modules $k[x]/(x^n)$ computed
in the category of $x$\+torsion $k[x]$\+modules $\Discr\fS$, is
the submodule of all $x$\+power torsion elements in
$\prod_{n=1}^\infty k[x]/(x^n)$.

 Dual-analogously, the exact, fully faithful functor of restriction of
scalars $\ff_\sharp\:\fS\Contra\rarrow\fR\Contra=\fR\Modl$ does
\emph{not} have a right adjoint.
 The point is that the functor~$\ff_\sharp$ does \emph{not} preserve
countable direct sums.
 Once again, the $k[x]$\+modules $k[x]/(x^n)$, \,$n\ge1$, have
$x$\+power infinite summation operations, but their countable direct
sum $\bigoplus_{n=1}^\infty k[x]/(x^n)$, computed in $k[x]\Modl$,
does \emph{not} have such an operation.
 For a construction and discussion of the direct sum of the objects
$k[x]/(x^n)$ computed in the category $\fS\Contra$ of $k[x]$\+modules
with $x$\+power infinite summation operation,
see~\cite[Section~1.5]{Prev}.
 The direct sum $\bigoplus_{n=1}^\infty k[x]/(x^n)$ computed in
$\fS\Contra$ is a nonseparated $\fS$\+contramodule.  \hbadness=1550

 The countable direct sum $\bigoplus_{n=1}^\infty k[x]/(x^n)$ also
exists in the category of separated $\fS$\+contramodules $\fS\Separ$,
and it is still \emph{not} preserved by the fully faithful forgetful
functor $\ff_\sharp\:\fS\Separ\rarrow\fR\Separ=\fR\Modl$.
 Consequently, the latter functor does not have a right adjoint
functor, either.

 One illuminating, relevant observation in connection with this series
of examples is that the continuous ring map~$\ff$ considered here
is \emph{not} taut.
\end{ex}

 One aim of this section is to show that under certain assumptions
a left adjoint functor to~$\ff_\diamond$ and a right adjoint functor
to~$\ff_\sharp$ do exist, nevertheless.
 Another aim is to construct such adjoint functors explicitly.

 Notice that, for any complete, separated right linear topological
ring $\fS$, the endomorphism ring of the free left $\fS$\+contramodule
$\fS$ is naturally isomorphic to the opposite ring $\fS^\rop$ to
the ring~$\fS$.
 In other words, the right action of $\fS$ on itself is an action by
left $\fS$\+contramodule endomorphisms.

\begin{prop} \label{discrete-module-extension-of-scalars-prop}
 Let\/ $\fR$ and\/ $\fS$ be complete, separated right linear topological
rings with \emph{countable} bases of neighborhoods of zero, and let\/
$\ff\:\fR\rarrow\fS$ be a taut continuous ring map.
 Let\/ $\cM$ be a discrete right\/ $\fR$\+module.
 Then the right\/ $\fS$\+module\/ $\cM\ocn_\fR\fS$ is discrete.
 The functor\/ $\ff^\bullet\:\cM\longmapsto\cM\ocn_\fR\fS$ is left
adjoint to the functor of restriction of scalars\/
$\ff_\diamond\:\Discr\fS \rarrow\Discr\fR$.
\end{prop}

\begin{proof}
 Every discrete right $\fR$\+module is a quotient $\fR$\+module of
a direct sum of cyclic discrete right $\fR$\+modules $\fR/\fI$,
where $\fI$ ranges over the open right ideals in~$\fR$.
 The full subcategory $\Discr\fS\subset\Modr\fS$ is closed under
quotients and infinite direct sums, and the functor ${-}\ocn_\fR\fS\:
\Discr\fR\rarrow\Modr\fS$ is right exact and preserves infinite
direct sums.
 Thus, in order to show that the essential image of the latter functor
is contained in $\Discr\fS$, it suffices to check that
the right $\fS$\+module $\fR/\fI\ocn_\fR\fS=\fS/(\fI\tim\fS)$ is
discrete for every open right ideal $\fI\subset\fR$.

 Indeed, by Lemma~\ref{taut-ring-maps-lemma}(1)
or~\ref{continuous-taut-ring-maps-lemma}(2), the closure $\fJ$ of
the right ideal $\ff(\fI)\fS\subset\fS$ is an open right ideal
in~$\fS$.
 By Lemma~\ref{continuous-and-taut-contramodule-tim-lemma}, we have
$\fI\tim\fS=\fJ\subset\fS$.
 Thus $\fS/(\fI\tim\fS)=\fS/\fJ$ is a discrete right $\fS$\+module.
 This proves the first assertion of the proposition.

 Furthermore, for every right $\fR$\+module $M$ and every right
$\fS$\+module $N$, we have an adjunction isomorphism of abelian groups
$\Hom_{\fS^\rop}(M\ot_\fR\fS,\>N)\simeq\Hom_{\fR^\rop}(M,N)$, which
induces a natural injective map $\Hom_{\fS^\rop}(\cM\ocn_\fR\fS,\>N)
\rarrow\Hom_{\fR^\rop}(\cM,N)$ for every discrete right
$\fR$\+module $\cM$, by precomposing with the natural surjective
right $\fS$\+module map $\cM\ot_\fR\fS\rarrow\cM\ocn_\fR\fS$.
 We need to show that, for a discrete right $\fS$\+module $\cN$,
the resulting map of abelian groups
\begin{equation} \label{discrete-modules-adjunction-map}
 \Hom_{\fS^\rop}(\cM\ocn_\fR\fS,\>\cN)\lrarrow
 \Hom_{\fR^\rop}(\cM,\cN)
\end{equation}
is an isomorphism.

 Indeed, every discrete right $\fR$\+module is the cokernel of
a morphism of infinite direct sums of cyclic discrete right
$\fR$\+modules.
 As both the left-hand side and the right-hand side
of~\eqref{discrete-modules-adjunction-map} take colimits in
the discrete $\fR$\+module argument $\cM$ to limits in $\Ab$,
it suffices to check that~\eqref{discrete-modules-adjunction-map} is
an isomorphism for cyclic discrete right $\fR$\+modules $\cM=\fR/\fI$.
 In this case, in the notation above, we have $\cM\ocn_\fR\fS=\fS/\fJ$,
and the assertion that~\eqref{discrete-modules-adjunction-map} is
an isomorphism simply means that an element $y\in\cN$ is annihilated
by $\fJ$ whenever it is annihilated by~$\fI$.
 This is the result of Lemma~\ref{R-S-annihilator-easy-lemma}.
\end{proof}

 We will call the functor $\ff^\bullet\:\cM\longmapsto\cM\ocn_\fR\fS$
the functor of \emph{discrete module extension of scalars}.
 The functor of discrete module extension of scalars
$\ff^\bullet\:\Discr\fR\rarrow\Discr\fS$ is well defined for any taut
continuous map of complete, separated right linear topological rings
with countable bases of neighborhoods of zero, as per
Proposition~\ref{discrete-module-extension-of-scalars-prop}.
 The functor~$\ff^\bullet$ agrees with the functor of extension
of scalars $\ff^*\:\Modr\fR\rarrow\Modr\fS$ whenever the forgetful
functor $\fR\Contra\rarrow\fR\Modl$ is fully faithful (as per
Lemma~\ref{contramodule-to-module-full-and-faithful-lemma}%
\,(1)\,$\Rightarrow$\,(3)).

\begin{prop} \label{separated-contramodule-coextension-of-scalars-prop}
 Let\/ $\fR$ and\/ $\fS$ be complete, separated two-sided linear
topological rings with countable bases of neighborhoods of zero, and
let\/ $\ff\:\fR\rarrow\fS$ be a strongly right taut continuous ring map.
 Let\/ $\fP$ be a separated left\/ $\fR$\+contramodule.
 Then the left\/ $\fS$\+module\/ $\Hom^\fR(\fS,\fP)$ carries a naturally
defined separated left\/ $\fS$\+contramodule structure from which its
left\/ $\fS$\+module structure arises.
 The functor\/ $\ff^\uds\:\fP\longmapsto\Hom^\fR(\fS,\fP)$ is right
adjoint to the functor of restriction of scalars\/ $\ff_\sharp\:
\fS\Separ\rarrow\fR\Separ$ from
Lemma~\ref{separated-restriction-of-scalars}.
 The adjunction counit\/ $\ff_\sharp\ff^\uds\rarrow\Id$ is the map\/
$\Hom^\fR(\fS,\fP)\rarrow\Hom^\fR(\fR,\fP)\simeq\fP$ induced by
the left\/ $\fR$\+contramodule morphism\/ $\ff\:\fR\rarrow\fS$.
\end{prop}

\begin{proof}
 Let $(\fI_n)_{n\ge1}$ and $(\fJ_n)_{n\ge1}$ be two descending chains of
open two-sided ideals in $\fR$ and $\fS$ provided by
Lemma~\ref{continuous-strongly-taut-ring-map-countable-base}(2).
 Denote by $R_n=\fR/\fI_n$ and $S_n=\fS/\fJ_n$ the related discrete
quotient rings.
 Then we have a strongly right taut morphism of projective systems
of rings $(f_n\:R_n\to S_n)$ such that $\ff=\varprojlim_{n\ge1}f_n$,
as in Section~\ref{change-of-scalars-I-secn}.

 Let $(P_n)_{n\ge1}$ be the left $(R_n)$\+system of modules
corresponding to the separated $\fR$\+contramodule~$\fP$.
 So we have $\fP\simeq\varprojlim_{n\ge1}P_n$, where the projective
limit is taken in the category $\fR\Contra$, or equivalently, in
its full subcategory $\fR\Separ$.
 Hence $\Hom^\fR(\fS,\fP)\simeq\varprojlim_{n\ge1}\Hom^\fR(\fS,P_n)$.
 We compute
$$
 \Hom^\fR(\fS,P_n)\simeq\Hom_{R_n}(\fS/(\fI_n\tim\fS),\>P_n)
 \simeq\Hom_{R_n}(S_n,P_n),
$$
as $\fI_n\tim\fS=\fJ_n$ by
Lemma~\ref{continuous-and-taut-contramodule-tim-lemma}.
 Now $\Hom^\fR(\fS,\fP)\simeq\varprojlim_{n\ge1}\Hom_{R_n}(S_n,P_n)$
is a projective limit of left $S_n$\+modules.
 All left $S_n$\+modules are separated left $\fS$\+contramodules, and
the inclusion/forgetful functors $\fS\Separ\rarrow\fS\Contra\rarrow
\fS\Modl$ preserve all limits (all limits exist in all the three
categories).
 Thus $\Hom^\fR(\fS,\fP)$ is naturally a left $\fS$\+contramodule.

 We have constructed the functor~$\ff^\uds$; it remains to establish
the adjunction.
 For any left $\fS$\+contramodule $\fQ$, we have
$$
 \Hom^\fS(\fQ,\Hom^\fR(\fS,\fP))\simeq
 \varprojlim\nolimits_{n\ge1}\Hom^\fS(\fQ,\Hom_{R_n}(S_n,P_n)).
$$
 We compute
\begin{multline*}
 \Hom^\fS(\fQ,\Hom_{R_n}(S_n,P_n))\simeq
 \Hom_{S_n}(\fQ/(\fJ_n\tim\fQ),\>\Hom_{R_n}(S_n,P_n)) \\
 \simeq\Hom_{R_n}(\fQ/(\fJ_n\tim\fQ),\>P_n)
 \simeq\Hom_{R_n}(\fQ/(\fI_n\tim\fQ),\>P_n)
 \simeq\Hom^\fR(\fQ,P_n),
\end{multline*}
as $\fJ_n\tim\fQ=\fI_n\tim\fQ$ by
Corollary~\ref{tim-Omega-and-Lambda-over-R-and-S-compatible}.
 Passing to the projective limit, we finally obtain the isomorphisms
$$
 \varprojlim\nolimits_{n\ge1}\Hom^\fS(\fQ,\Hom_{R_n}(S_n,P_n))\simeq
 \varprojlim\nolimits_{n\ge1}\Hom^\fR(\fQ,P_n)\simeq\Hom^\fR(\fQ,\fP),
$$
so $\Hom^\fS(\fQ,\ff^\uds(\fP))\simeq\Hom^\fR(\fQ,\fP)$, as desired.
 The description of the adjunction counit follows straightforwardly
from the constructions.
\end{proof}

 We will call the functor $\ff^\uds\:\fP\longmapsto\Hom^\fR(\fS,\fP)$
constructed in the proof above the functor of
\emph{separated contramodule coextension of scalars}.
 The functor of separated contramodule coextension of scalars
$\ff^\uds\:\fR\Separ\rarrow\fS\Separ$ is well defined for any strongly
right taut continuous map of complete, separated two-sided linear
topological rings with countable bases of neighborhoods of zero, as per
Proposition~\ref{separated-contramodule-coextension-of-scalars-prop}.
 The functor~$\ff^\uds$ agrees with the functor of coextension
of scalars $\ff^!\:\fR\Modl\rarrow\fS\Modl$ whenever the forgetful
functor $\fR\Contra\rarrow\fR\Modl$ is fully faithful.

\begin{thm} \label{contramodule-coextension-of-scalars-thm}
 Let\/ $\fR$ and\/ $\fS$ be complete, separated two-sided linear
topological rings with countable bases of neighborhoods of zero, and
let\/ $\ff\:\fR\rarrow\fS$ be a left proflat continuous ring map.
 Let\/ $\fP$ be a left\/ $\fR$\+contramodule.
 Then the left\/ $\fS$\+module\/ $\Hom^\fR(\fS,\fP)$ carries a naturally
defined left\/ $\fS$\+contramodule structure from which its left\/
$\fS$\+module structure arises.
 The functor\/ $\ff^\uds\:\fP\longmapsto\Hom^\fR(\fS,\fP)$ is right
adjoint to the functor of restriction of scalars\/ $\ff_\sharp\:
\fS\Contra\rarrow\fR\Contra$.
 The adjunction counit\/ $\ff_\sharp\ff^\uds\rarrow\Id$ is the map\/
$\Hom^\fR(\fS,\fP)\rarrow\Hom^\fR(\fR,\fP)\simeq\fP$ induced by
the left\/ $\fR$\+contramodule morphism\/ $\ff\:\fR\rarrow\fS$.
\end{thm}

\begin{proof}
 Let $0\rarrow\fC\rarrow\fF\rarrow\fP\rarrow0$ be a short exact
sequence of left $\fR$\+contramodules with a cotorsion
$\fR$\+contramodule $\fC$ and a flat $\fR$\+contramodule $\fF$,
as in Proposition~\ref{flat-cotorsion-pair-complete-prop}(a).
 The left $\fR$\+contramodule $\fS$ is flat by
Corollary~\ref{left-flat-and-contramodule-flat-cor}, so
the short sequence of left $\fS$\+modules
$0\rarrow\Hom^\fR(\fS,\fC)\rarrow\Hom^\fR(\fS,\fF)\rarrow
\Hom^\fR(\fS,\fP)\rarrow0$ is exact (by the definition of
a cotorsion left $\fR$\+contramodule).

 The flat $\fR$\+contramodule $\fF$ is separated by
Lemma~\ref{complete-separated-flat-contramodules-lemma}(c), hence
the $\fR$\+con\-tra\-mod\-ule $\fC$ is separated, too.
 By
Proposition~\ref{separated-contramodule-coextension-of-scalars-prop},
it follows that both the $\fS$\+modules
$\Hom^\fR(\fS,\fC)$ and $\Hom^\fR(\fS,\fF)$ carry natural separated left
$\fS$\+contramodule structures, and the map $\Hom^\fR(\fS,\fC)\rarrow
\Hom^\fR(\fS,\fF)$ is an $\fS$\+contramodule morphism.
 The forgetful functor $\fS\Contra\rarrow\fS\Modl$ preserves cokernels.
 We endow $\Hom^\fR(\fS,\fP)$ with the left $\fS$\+contramodule
structure of the cokernel of the $\fS$\+contramodule morphism
$\Hom^\fR(\fS,\fC)\rarrow\Hom^\fR(\fS,\fF)$.

 Now let $\fP'\rarrow\fP''$ be a morphism of left $\fR$\+contramodules.
 We need to check that the induced map $\Hom^\fR(\fS,\fP')\rarrow
\Hom^\fR(\fS,\fP'')$ is an $\fS$\+contramodule morphism.
 Let $0\rarrow\fC'\rarrow\fF'\rarrow\fP'\rarrow0$ and
$0\rarrow\fC''\rarrow\fF''\rarrow\fP''\rarrow0$ be two short exact
sequences of left $\fR$\+contramodules with cotorsion
$\fR$\+contramodules $\fC'$, $\fC''$ and flat $\fR$\+contramodules
$\fF'$, $\fF''$.

 Consider the composition $\fF'\rarrow\fP'\rarrow\fP''$.
 Since $\Ext^{\fR,1}(\fF',\fC'')=0$, the morphism $\fF'\rarrow\fP''$
lifts to an $\fR$\+contramodule morphism $\fF'\rarrow\fF''$.
 Hence we obtain a morphism of short exact sequences of left
$\fR$\+contramodules $(0\to\fC'\to\fF'\to\fP'\to0)\rarrow
(0\to\fC''\to\fF''\to\fP''\to0)$.
 The assertion that the map $\Hom^\fR(\fS,\fP')\rarrow
\Hom^\fR(\fS,\fP'')$ is an $\fS$\+contramodule morphism now follows from
the construction above of the $\fS$\+contramodule structures involved.

 We have constructed the promised functor~$\ff^\uds$; it remains to
establish the adjunction.
 It follows from
Proposition~\ref{separated-contramodule-coextension-of-scalars-prop}
and the construction above that the desired adjunction counit
$\Hom^\fR(\fS,\fP)\rarrow\Hom^\fR(\fR,\fP)\simeq\fP$ is a left
$\fR$\+contramodule morphism.
 For any left $\fS$\+contramodule $\fQ$, we have the induced map of
abelian groups
\begin{equation} \label{uds-adjunction-map}
 \Hom^\fS(\fQ,\Hom^\fR(\fS,\fP))\lrarrow\Hom^\fR(\fQ,\fP).
\end{equation}

 Viewed as functors of the $\fS$\+contramodule argument $\fQ$, both
the left-hand side and the right-hand side of~\eqref{uds-adjunction-map}
take cokernels of morphisms in $\fS\Contra$ to kernels of morphisms
of abelian groups.
 Hence it suffices to consider the case of a projective left
$\fS$\+contramodule~$\fQ$.
 Then the underlying left $\fR$\+contramodule of $\fQ$ is flat by
Corollary~\ref{proflat-contramodule-direct-image-cor}.

 Hence, in the notation above, we have short exact sequences of abelian
groups
\begin{multline} \label{left-hand-side-short-exact-sequence}
 0\lrarrow\Hom^\fS(\fQ,\Hom^\fR(\fS,\fC))\lrarrow
 \Hom^\fS(\fQ,\Hom^\fR(\fS,\fF)) \\
 \lrarrow\Hom^\fS(\fQ,\Hom^\fR(\fS,\fP))\lrarrow0
\end{multline}
and
\begin{equation} \label{right-hand-side-short-exact-sequence}
 0\lrarrow\Hom^\fR(\fQ,\fC)\lrarrow\Hom^\fR(\fQ,\fF)\lrarrow
 \Hom^\fR(\fQ,\fP)\lrarrow0.
\end{equation}
 The natural transformation~\eqref{uds-adjunction-map} provides
a morphism from the short exact
sequence~\eqref{left-hand-side-short-exact-sequence} to the short
exact sequence~\eqref{right-hand-side-short-exact-sequence},
which is an isomorphism on the leftmost and middle terms by
Proposition~\ref{separated-contramodule-coextension-of-scalars-prop}.
 Hence this map is also an isomorphism on the rightmost terms of
\eqref{left-hand-side-short-exact-sequence}
and~\eqref{right-hand-side-short-exact-sequence}, as desired.
\end{proof}

 We will call the functor $\ff^\uds\:\fP\longmapsto\Hom^\fR(\fS,\fP)$
constructed in the proof above the functor of
\emph{contramodule coextension of scalars}.
 The functor of contramodule coextension of scalars
$\ff^\uds\:\fR\Contra\rarrow\fS\Contra$ is well defined for any left
proflat map of complete, separated two-sided linear topological rings
with countable bases of neighborhoods of zero, as per
Theorem~\ref{contramodule-coextension-of-scalars-thm}.
 The functor~$\ff^\uds$ agrees with the functor of coextension
of scalars $\ff^!\:\fR\Modl\rarrow\fS\Modl$ whenever the forgetful
functor $\fR\Contra\rarrow\fR\Modl$ is fully faithful.

\begin{lem} \label{commutative-contramodule-structure-on-Hom}
 Let\/ $\fR$ be a complete, separated commutative linear topological
ring, and let\/ $\fP$ and\/ $\fQ$ be\/ $\fR$\+contramodules.
 Then the abelian group of\/ $\fR$\+contramodule morphisms\/
$\Hom^\fR(\fQ,\fP)$ has a natural\/ $\fR$\+contramodule structure.
\end{lem}

\begin{proof}
 The $\fR$\+contramodule structure on $\fH=\Hom^\fR(\fQ,\fP)$ is simply
induced by the $\fR$\+contramodule structure on~$\fP$.
 The contramodule infinite summation operations in\/ $\Hom^\fR(\fQ,\fP)$
are constructed as a kind of pointwise summation of $\fP$\+valued
functions.
 More precisely,
$$
 \pi_\fH\left(\sum\nolimits_{h\in\fH}r_hh\right)(q)=
 \pi_\fP\left(\sum\nolimits_{p\in\fP}
 \left(\sum\nolimits_{h\in\fH}^{h(q)=p}r_h\right)p\right).
$$
 Here $(r_h\in\fR)_{h\in\fH}$ is a zero-convergent family of elements
of $\fR$ indexed by elements of~$\fH$.
 The second (inner) summation sign in the right-hand side denotes
the topological limit of finite partial sums in the topology of~$\fR$,
while the first (outer) summation sign is the notation for an element
of $\fR[[\fP]]$ to which the contraction map~$\pi_\fP$ is applied.
 That element of $\fR[[\fP]]$ is obtained by applying the map
$\fR[[g_q]]\:\fR[[\fH]]\rarrow\fR[[\fP]]$ to the element
$\sum_{h\in\fH}r_hh\in\fR[[\fH]]$, where $g_q\:\fH\rarrow\fP$ is
the map of sets given by the rule $g_q(h)=h(q)$.
\end{proof}

\begin{rem}
 Let $\fR$ be a complete, separated commutative linear topological ring,
and let $\fS$ be a complete, separated two-sided linear topological
ring.
 Assume that both $\fR$ and $\fS$ have countable bases of neighborhoods
of zero.
 Let $\ff\:\fR\rarrow\fS$ be a continuous ring map whose image is
contained in the center of~$\fS$.
 Then the left and right $\fR$\+contramodule structures on $\fS$
coincide (as both are given by topological limits of finite partial
sums converging in the topology of~$\fS$).
 Assume that the ring map~$\ff$ is taut (or equivalently, strongly
right taut).
 Then, in the context of
Proposition~\ref{separated-contramodule-coextension-of-scalars-prop},
the underlying $\fR$\+contramodule structure of the left
$\fS$\+contramodule structure on $\Hom^\fR(\fS,\fP)$ provided by
the proposition coincides with the $\fR$\+contramodule structure
provided by Lemma~\ref{commutative-contramodule-structure-on-Hom}.
 Assuming further that the ring map~$\ff$ is (left) proflat, the same
assertion applies to the $\fS$\+contramodule and $\fR$\+contramodule
structures on $\Hom^\fR(\fS,\fP)$ appearing in the context of
Theorem~\ref{contramodule-coextension-of-scalars-thm} and
Lemma~\ref{commutative-contramodule-structure-on-Hom}.
 Both the statements follow straightforwardly from the constructions.
\end{rem}

\begin{rem} \label{exten-coexten-contraexten-exactness-properties-rem}
 Given a homomorphism of abstract (nontopological) rings
$f\:R\rarrow S$, the functor of extension of scalars $f^*\:R\Modl\rarrow
S\Modl$ is exact whenever $S$ is a flat right $R$\+module.
 For an arbitrary ring homomorphism~$f$, the functor $f^*$ is exact
on the full subcategory of flat $R$\+modules in $R\Modl$.
 The functor of coextension of scalars $f^!\:R\Modl\rarrow S\Modl$
is exact on the full exact subcategory of cotorsion $R$\+modules
in $R\Modl$ provided that $S$ is flat left $R$\+module.

 What exactness properties do the discrete module and contramodule
functors of extension/coextension/contraextension of scalars have,
for a morphism of complete, separated right linear topological rings
$\ff\:\fR\rarrow\fS$\,?

 The functor of discrete module extension of scalars
$\ff^\bullet\:\Discr\fR\rarrow\Discr\fS$ from
Proposition~\ref{discrete-module-extension-of-scalars-prop} is exact
whenever $\ff$~is a taut continuous ring map and $\fS$ is a flat left
$\fR$\+contramodule (by construction).
 By Corollary~\ref{left-flat-and-contramodule-flat-cor}, these
assumptions hold, in particular, whenever $\ff$~is a left proflat
continuous ring map.

 The functor of contramodule coextension of scalars
$\ff^\uds\:\fR\Contra\rarrow\fS\Contra$ from
Theorem~\ref{contramodule-coextension-of-scalars-thm} is exact on
the exact category of cotorsion left $\fR$\+contramodules (by
construction) whenever $\ff$~is a left proflat continuous ring map.
 The functor of contraextension of scalars
$\ff^\sharp\:\fR\Contra\rarrow\fS\Contra$ is always exact
on the exact category of flat left $\fR$\+contramodules, by
Lemma~\ref{flat-contramodules-contraextension-of-scalars}.

 The analogy with the module case may suggest that
the functor~$\ff^\sharp$ should be exact on the whole abelian category
of left $\fR$\+contramodules under suitable flatness assumptions
on~$\ff$.
 The following counterexample appears to refute such na\"\i ve
expectations decisively.
\end{rem}

\begin{ex} \label{contraextension-not-exact-counterex}
 Let $\ff\:\fR\rarrow\fS$ be the proflat proepimorphism of Noetherian
commutative rings with adic topologies discussed in
Example~\ref{proepimorphism-not-epimorphism-example}
and Remark~\ref{non-R-contramodule-S-module-structure-not-unique}.
 So we have $\fR=k[x][[y]]$ and $\fS=k[x,x^{-1}][[y]]$.
 Since $\fS$ is a flat $\fR$\+contramodule and $\fR$ is a Noetherian
commutative ring with adic topology, $\fS$ is actually a flat
$\fR$\+module by~\cite[Lemma~B.9.2]{Pweak},
\cite[Corollary~10.3(a)]{Pcta}, or~\cite[Theorem~1.6(2) or~6.11]{Yek2}.
 Nevertheless, the functor $\ff^\sharp\:\fR\Contra\rarrow
\fS\Contra$ is \emph{not} exact.

 For a counterexample, consider the following module $\fQ$ over the ring
of formal power series $k[[x,y]]$.
 Put $\fQ=\prod_{i,j\in\boZ}^{i+j\ge0}kx^iy^j$.
 Here $kx^iy^j$ is a one-dimensional $k$\+vector space with the basis
vector denoted formally by $x^iy^j$, and the infinite direct product
is taken over the integer half-plane $\{i,j\in\boZ\mid i+j\ge0\}
\subset\boZ^2$.
 The ring of formal power series $k[[x,y]]$ acts in $\fQ$ in the obvious
way.
 In fact, the $[x,y]$\+power infinite summation operations, in
the sense of~\cite[Section~4]{Pcta}, are easily defined on $\fQ$,
making $\fQ$ a contramodule over the topological ring $k[[x,y]]$
(endowed with the adic topology of the ideal $(x,y)\subset k[[x,y]]$),
cf.~\cite[Theorem~B.1.1 and Section~B.6]{Pweak}.
 The restriction of scalars with respect to the continuous identity
inclusion map of rings $\fR\rarrow k[[x,y]]$ endows $\fQ$ with
an $\fR$\+contramodule structure.

 Put $R=k[x,y]$ and $S=k[x,x^{-1},y]$; so $\fR$ is the $y$\+adic
completion of $R$ and $\fS$ is the $y$\+adic completion of~$S$.
 Then the identity inclusions of rings $R\rarrow\fR$ and
$S\rarrow\fS$ induce fully faithful forgetful functors
$\fR\Contra\rarrow R\Modl$ and $\fS\Contra\rarrow R\Modl$,
by~\cite[Theorem~B.1.1]{Pweak} (cf.~\cite[Section~1.6]{Prev}).
 Denote by $\Delta_y=\boL_0\Lambda_y\:S\Modl\rarrow\fS\Contra$
the left adjoint functor to the forgetful/inclusion functor
$\fS\Contra\rarrow S\Modl$.
 A construction of the functor $\Delta_y$ is spelled out
in~\cite[Proposition~2.1]{Pmgm} or~\cite[Section~6]{Pcta}, while
a discussion of the functor $\boL_0\Lambda_y$ and its comparison
with $\Delta_y$ can be found in~\cite[Section~1 and Remark~3.8]{Pdc}.

 The functor $\ff^\sharp\:\fR\Contra\rarrow\fS\Contra$ can be
computed by the rule $\ff^\sharp(\fP)=\Delta_y(S\ot_R\fP)$ for
any $\fR$\+contramodule~$\fP$.
 In particular, for the $\fR$\+contramodule $\fQ$, we observe
that a pair of mutually inverse automorphisms $\frac xy$ and $\frac yx$
is defined on~$\fQ$.
 Hence the element $y\in S$ acts by an invertible map on
the $S$\+module $S\ot_R\fQ=\fQ[x^{-1}]$.
 In other words, inverting the action of~$x$ on $\fQ$ also makes
the action of~$y$ invertible.
 One has $\Hom_S(M,\fP)=0$ for any $y$\+divisible $S$\+module $M$ and
any $\fS$\+contramodule~$\fQ$ (see~\cite[Theorem~B.1.1]{Pweak}
or~\cite[Lemma~3.2]{Pcta}); so $\Delta_y(M)=0$.
 Thus, in particular, $\ff^\sharp(\fQ)=\Delta_y(S\ot_R\fQ)=0$.

 On the other hand, there is an obvious injective $\fR$\+contramodule
map $g\:\fR\rarrow\fQ$ from the free $\fR$\+contramodule $\fR$ to
the $\fR$\+contramodule $\fQ$, given by the rule $g(1)=x^0y^0$.
 Furthermore, we have $\ff^\sharp(\fR)=\fS\ne0$.
 So $\ff^\sharp(g)\:\fS\rarrow0$ is \emph{not} an injective map,
and the functor~$\ff^\sharp$ is not exact.
\end{ex}


\bigskip
\begin{thebibliography}{99}
\smallskip

\bibitem{AR}
 J.~Ad\'amek, J.~Rosick\'y.
   Locally presentable  and accessible categories.
London Math.\ Society Lecture Note Series~189,
Cambridge University Press, 1994.

\bibitem{AJL}
 L.~Alonso Tarr\'\i o, A.~Jerem\'\i as L\'opez, J.~Lipman.
   Duality and flat base change on formal schemes.
In: Studies in duality on Noetherian formal schemes and non-Noetherian
ordinary schemes, p.~3--90, \emph{Contemporary Math.}\ \textbf{244},
American Math.\ Society, Providence, 1999.
\texttt{arXiv:alg-geom/9708006}
 Correction, \emph{Proceedings of the American Math.\ Society}
\textbf{131}, \#2, p.~351--357, 2003.  \texttt{arXiv:math.AG/0106239}

\bibitem{AJPV}
 L.~Alonso Tarr\'\i o, A.~Jerem\'\i as L\'opez,
M.~P\'erez Rodr\'\i guez, M.~J.~Vale Gonsalves.
   On the existence of a compact generator on the derived category
of a Noetherian formal scheme.
\textit{Appl.\ Categorical Struct.}\ \textbf{19}, \#6, p.~865--877,
2011.  \texttt{arXiv:0905.2063 [math.AG]}

\bibitem{BD2}
 A.~Beilinson, V.~Drinfeld.
   Quantization of Hitchin's integrable system and Hecke
eigensheaves.  February 2000.  Available from
\texttt{http://www.math.utexas.edu/\textasciitilde
benzvi/Langlands.html} or
\texttt{http://math.uchicago.edu/\textasciitilde
drinfeld/langlands.html}

\bibitem{BS}
 B.~Bhatt, P.~Scholze.
   The pro-\'etale topology for schemes.
\textit{Ast\'erisque} \textbf{369}, p.~99--201, 2015.
\texttt{arXiv:1309.1198 [math.AG]}

\bibitem{CR}
 B.~Chorny, J.~Rosick\'y.
   Class-locally presentable and class-accessible categories.
\textit{Journ.\ of Pure and Appl.\ Algebra} \textbf{216}, \#10,
p.~2113--2125, 2012.  \texttt{arXiv:1110.0605 [math.CT]}

\bibitem{ET}
 P.~C.~Eklof, J.~Trlifaj.
   How to make Ext vanish.
\textit{Bull.\ of the London Math.\ Soc.}\ \textbf{33}, \#1,
p.~41--51, 2001.

\bibitem{GT}
 R.~G\"obel, J.~Trlifaj.
   Approximations and endomorphism algebras of modules.
Second Revised and Extended Edition.
De Gruyter Expositions in Mathematics 41,
De Gruyter, Berlin--Boston, 2012.

\bibitem{EGAI}
 A.~Grothendieck, J.~A.~Dieudonn\'e.
   \'El\'ements de g\'eom\'etrie alg\'ebrique.~I.
Grundlehren der mathematischen Wissenschaften, 166.
Springer-Verlag, Berlin--Heidelberg--New York, 1971.

\bibitem{Har}
 R.~Hartshorne.
   Algebraic geometry.
Graduate Texts in Math., 52, Springer-Verlag,
New York--Heidelberg, 1977.

\bibitem{Hov}
 M.~Hovey.
   Cotorsion pairs, model category structures, and representation
theory.
\textit{Math.\ Zeitschrift} \textbf{241}, \#2, p.~553--592, 2002.

\bibitem{SP}
 A.~J.~de~Jong et al.
   The Stacks Project.
Available from \texttt{https://stacks.math.columbia.edu/}

\bibitem{JS}
 A.~Joyal, R.~Street.
   Pullbacks equivalent to pseudopullbacks.
\textit{Cahiers de topol.\ et g\'eom.\ diff\'er.\ cat\'egoriques}
\textbf{XXXIV}, \#2, p.~153--156, 1993.

\bibitem{PSY}
 M.~Porta, L.~Shaul, A.~Yekutieli.
   On the homology of completion and torsion.
\textit{Algebras and Represent.\ Theory} \textbf{17}, \#1, p.~31--67,
2014.  \texttt{arXiv:1010.4386 [math.AC]}.
Erratum in \textit{Algebras and Represent.\ Theory} \textbf{18},
\#5, p.~1401--1405, 2015.  \texttt{arXiv:1506.07765 [math.AC]}

\bibitem{Psemi}
 L.~Positselski.
   Homological algebra of semimodules and semicontramodules:
Semi-infinite homological algebra of associative algebraic structures.
 Appendix~C in collaboration with D.~Rumynin; Appendix~D in
collaboration with S.~Arkhipov.
 Monografie Matematyczne vol.~70, Birkh\"auser/Springer Basel, 2010. 
xxiv+349~pp. \texttt{arXiv:0708.3398 [math.CT]}

\bibitem{Pweak}
 L.~Positselski.
   Weakly curved A${}_\infty$-algebras over a topological local ring.
\textit{M\'emoires de la Soci\'et\'e Math\'ematique de France}
\textbf{159}, 2018.  vi+206~pp.  \texttt{arXiv:1202.2697 [math.CT]}

\bibitem{Pcosh}
 L.~Positselski.
   Contraherent cosheaves on schemes.
Electronic preprint \texttt{arXiv:1209.2995v24 [math.CT]}.

\bibitem{Prev}
 L.~Positselski.
   Contramodules.
\textit{Confluentes Math.}\ \textbf{13}, \#2, p.~93--182, 2021.
\texttt{arXiv:1503.00991 [math.CT]}

\bibitem{Pmgm}
 L.~Positselski.
   Dedualizing complexes and MGM duality.
\textit{Journ.\ of Pure and Appl.\ Algebra} \textbf{220}, \#12,
p.~3866--3909, 2016.  \texttt{arXiv:1503.05523 [math.CT]}

\bibitem{Pcta}
 L.~Positselski.
   Contraadjusted modules, contramodules, and reduced cotorsion modules.
\textit{Moscow Math.\ Journ.}\ \textbf{17}, \#3, p.~385--455, 2017.
\texttt{arXiv:1605.03934 [math.CT]}

\bibitem{Psm}
 L.~Positselski.
   Smooth duality and co-contra correspondence.
\textit{Journ.\ of Lie Theory} \textbf{30}, \#1, p.~85--144, 2020.
\texttt{arXiv:1609.04597 [math.CT]}

\bibitem{Pper}
 L.~Positselski.
   Abelian right perpendicular subcategories in module categories.
Electronic preprint \texttt{arXiv:1705.04960 [math.CT]}. 

\bibitem{Pcoun}
 L.~Positselski.
   Flat ring epimorphisms of countable type.
\textit{Glasgow Math.\ Journ.} \textbf{62}, \#2, p.~383--439, 2020.
\texttt{arXiv:1808.00937 [math.RA]}

\bibitem{Pproperf}
 L.~Positselski.
   Contramodules over pro-perfect topological rings.
\textit{Forum Mathematicum} \textbf{34}, \#1, p.~1--39, 2022.
\texttt{arXiv:1807.10671 [math.CT]}

\bibitem{Pdc}
 L.~Positselski.
   Remarks on derived complete modules and complexes.
\textit{Math.\ Nachrichten} \textbf{296}, \#2, p.~811--839, 2023.
\texttt{arXiv:2002.12331 [math.AC]}

\bibitem{Psemten}
 L.~Positselski.
   Semi-infinite algebraic geometry of quasi-coherent sheaves
on ind-schemes: Quasi-coherent torsion sheaves, the semiderived
category, and the semitensor product.
Birkh\"auser/Springer Nature, Cham, Switzerland, 2023.  xix+216~pp.
\texttt{arXiv:2104.05517 [math.AG]}

\bibitem{Pal}
 L.~Positselski.
   Local, colocal, and antilocal properties of modules and complexes
over commutative rings.
\textit{Journ.\ of Algebra} \textbf{646}, p.~100--155, 2024.
\texttt{arXiv:2212.10163 [math.AC]}

\bibitem{Pflcc}
 L.~Positselski.
   Flat comodules and contramodules as directed colimits, and cotorsion
periodicity.
\textit{Journ.\ of Homotopy and Related Struct.}\ \textbf{19}, \#4,
p.~635--678, 2024.  \texttt{arXiv:2306.02734 [math.RA]}

\bibitem{Pacc}
 L.~Positselski.
   Notes on limits of accessible categories.
\textit{Cahiers de topol.\ et g\'eom.\ diff\'er.\ cat\'egoriques}
\textbf{LXV}, \#4, p.~390--437, 2024.
\texttt{arXiv:2310.16773 [math.CT]}

\bibitem{Pphil}
 L.~Positselski.
   Philosophy of contraherent cosheaves.
Electronic preprint \texttt{arXiv:2311.14179 [math.AG]}.

\bibitem{Pform}
 L.~Positselski.
   Contraherent cosheaves of contramodules on Noetherian formal schemes.
Electronic preprint \texttt{arXiv:2603.27732v1 [math.AG]}.

\bibitem{PR}
 L.~Positselski, J.~Rosick\'y.
   Covers, envelopes, and cotorsion theories in locally presentable
abelian categories and contramodule categories.
\textit{Journ.\ of Algebra} \textbf{483}, p.~83--128, 2017.
\texttt{arXiv:1512.08119 [math.CT]}

\bibitem{PS1}
 L.~Positselski, J.~\v St\!'ov\'\i\v cek.
   The tilting-cotilting correspondence.
\textit{Internat.\ Math.\ Research Notices} \textbf{2021}, \#1,
p.~189--274, 2021.  \texttt{arXiv:1710.02230 [math.CT]}

\bibitem{RR}
 G.~Raptis, J.~Rosick\'y.
   The accessibility rank of weak equivalences.
\textit{Theory and Appl.\ of Categories} \textbf{30}, no.~19,
p.~687--703, 2015.  \texttt{arXiv:1403.3042 [math.AT]}

\bibitem{Sal}
 L.~Salce.
   Cotorsion theories for abelian groups.
\textit{Symposia Math.}\ \textbf{XXIII},
Academic Press, London--New York, 1979, p.~11--32.

\bibitem{Sch}
 P.~Schenzel.
   Proregular sequences, local cohomology, and completion.
\textit{Math.\ Scand.}\ \textbf{92}, \#2, p.~161--180, 2003.

\bibitem{Sim}
 A.-M.~Simon.
   Approximations of complete modules by complete big
Cohen--Macaulay modules over a Cohen--Macaulay local ring.
\textit{Algebras and Representation Theory} \textbf{12}, \#2--5,
p.~385--400, 2009.

\bibitem{Sten}
 B.~Stenstr\"{o}m.
   Rings of quotients. An Introduction to Methods of Ring Theory.
Die Grundlehren der Mathematischen Wissenschaften, Band~217.
Springer-Verlag, New York, 1975.

\bibitem{Yek0}
 A.~Yekutieli.
   On flatness and completion for infinitely generated modules over
noetherian rings.
\textit{Communicat.\ in Algebra} \textbf{39}, \#11,
p.~4221--4245, 2011.  \texttt{arXiv:0902.4378 [math.AC]}

\bibitem{Yek2}
 A.~Yekutieli.
   Flatness and completion revisited.
\textit{Algebras and Represent.\ Theory} \textbf{21}, \#4, p.~717--736,
2018.  \texttt{arXiv:1606.01832 [math.AC]}

\end{thebibliography}
\end{document}